%% file: 2topoi.tex
\pdfoutput=1
\documentclass{amsart}
\usepackage[utf8]{inputenc}
\usepackage{quiver}

\usepackage[margin=1in]{geometry}
\usepackage{graphicx}
\usepackage[expansion=false]{microtype}

\usepackage{amsmath, amsfonts, amssymb, amsthm, amsopn, mathtools}
\usepackage{eucal}
\usepackage{yfonts}
\usepackage{dsfont}

\usepackage{tikz-cd}
\usetikzlibrary{patterns}
\usepackage{bm}
\usepackage{rotating}
\usepackage{xparse}
\usepackage{pict2e}
\usepackage{enumitem}
\usepackage{csquotes}

\usepackage[pdfusetitle,unicode,colorlinks]{hyperref}
\usepackage[capitalise]{cleveref}

\numberwithin{equation}{subsection}
\theoremstyle{plain}
\newtheorem{theorem}[equation]{Theorem}
\newtheorem{proposition}[equation]{Proposition}
\newtheorem{observation}[equation]{Observation}
\newtheorem{lemma}[equation]{Lemma}
\newtheorem{corollary}[equation]{Corollary}

\newtheorem{theoremint}{Theorem}
\newtheorem*{fibdescent}{Fibrational descent}
\newtheorem*{lt1}{LT1}
\newtheorem*{lt2}{LT2}

\theoremstyle{definition}
\newtheorem{definition}[equation]{Definition}
\newtheorem{example}[equation]{Example}
\newtheorem{notation}[equation]{Notation}
\newtheorem{remark}[equation]{Remark}

\newtheorem{construction}[equation]{Construction}

\let\scr=\mathcal

\let\phi=\varphi

\let\into=\hookrightarrow

\makeatletter
\DeclareSymbolFont{lettersA}{U}{txmia}{m}{it}

\DeclareRobustCommand*{\varmathbb}[1]{\gdef\F@ntPrefix{m@thbbch@r}%
	\@EachCharacter #1\@EndEachCharacter}

\long\def\DoLongFutureLet #1#2#3#4{%
	\def\@FutureLetDecide{#1#2\@FutureLetToken
		\def\@FutureLetNext{#3}\else
		\def\@FutureLetNext{#4}\fi\@FutureLetNext}
	\futurelet\@FutureLetToken\@FutureLetDecide}
\def\DoFutureLet #1#2#3#4{\DoLongFutureLet{#1}{#2}{#3}{#4}}
\def\@EachCharacter{\DoFutureLet{\ifx}{\@EndEachCharacter}%
	{\@EachCharacterDone}{\@PickUpTheCharacter}}
\def\m@keCharacter#1{\csname\F@ntPrefix#1\endcsname}
\def\@PickUpTheCharacter#1{\m@keCharacter{#1}\@EachCharacter}
\def\@EachCharacterDone \@EndEachCharacter{}

\DeclareMathSymbol{\m@thbbch@rA}{\mathord}{lettersA}{129}
\DeclareMathSymbol{\m@thbbch@rB}{\mathord}{lettersA}{130}
\DeclareMathSymbol{\m@thbbch@rC}{\mathord}{lettersA}{131}
\DeclareMathSymbol{\m@thbbch@rD}{\mathord}{lettersA}{132}
\DeclareMathSymbol{\m@thbbch@rE}{\mathord}{lettersA}{133}
\DeclareMathSymbol{\m@thbbch@rF}{\mathord}{lettersA}{134}
\DeclareMathSymbol{\m@thbbch@rG}{\mathord}{lettersA}{135}
\DeclareMathSymbol{\m@thbbch@rH}{\mathord}{lettersA}{136}
\DeclareMathSymbol{\m@thbbch@rI}{\mathord}{lettersA}{137}
\DeclareMathSymbol{\m@thbbch@rJ}{\mathord}{lettersA}{138}
\DeclareMathSymbol{\m@thbbch@rK}{\mathord}{lettersA}{139}
\DeclareMathSymbol{\m@thbbch@rL}{\mathord}{lettersA}{140}
\DeclareMathSymbol{\m@thbbch@rM}{\mathord}{lettersA}{141}
\DeclareMathSymbol{\m@thbbch@rN}{\mathord}{lettersA}{142}
\DeclareMathSymbol{\m@thbbch@rO}{\mathord}{lettersA}{143}
\DeclareMathSymbol{\m@thbbch@rP}{\mathord}{lettersA}{144}
\DeclareMathSymbol{\m@thbbch@rQ}{\mathord}{lettersA}{145}
\DeclareMathSymbol{\m@thbbch@rR}{\mathord}{lettersA}{146}
\DeclareMathSymbol{\m@thbbch@rS}{\mathord}{lettersA}{147}
\DeclareMathSymbol{\m@thbbch@rT}{\mathord}{lettersA}{148}
\DeclareMathSymbol{\m@thbbch@rU}{\mathord}{lettersA}{149}
\DeclareMathSymbol{\m@thbbch@rV}{\mathord}{lettersA}{150}
\DeclareMathSymbol{\m@thbbch@rW}{\mathord}{lettersA}{151}
\DeclareMathSymbol{\m@thbbch@rX}{\mathord}{lettersA}{152}
\DeclareMathSymbol{\m@thbbch@rY}{\mathord}{lettersA}{153}
\DeclareMathSymbol{\m@thbbch@rZ}{\mathord}{lettersA}{154}

\makeatother

\def\AA{\scr A}
\def\BB{\scr B}
\def\CC{\scr C}
\def\DD{\scr D}

\def\GG{\scr G}
\def\II{\scr I}
\def\JJ{\scr J}

\def\SS{\scr S}

\def\XX{\scr X}
\def\YY{\scr Y}

\def\Spc{\scr S}

\def\bcat{\mathbb}
\def\bAA{\bcat A}
\def\bBB{\bcat B}
\def\bCC{\bcat C}
\def\bDD{\bcat D}
\def\bEE{\bcat E}

\def\bII{\bcat I}
\def\bJJ{\bcat J}
\def\bSS{\bcat S}
\def\bXX{\bcat X}
\def\bYY{\bcat Y}

\def\colimcone{\triangleright}
\def\limcone{\triangleleft}

\DeclareMathOperator{\Cat}{\mathfrak{C}\!\operatorname{at}}
\DeclareMathOperator{\CAT}{\widehat{\mathfrak{C}\!\operatorname{at}}}
\DeclareMathOperator{\Catidem}{\Cat^{\operatorname{idem}}}
\DeclareMathOperator{\CCat}{{\Cat_2}}
\DeclareMathOperator{\CCatm}{\mathfrak{m}\!\Cat_2}
\DeclareMathOperator{\CCats}{\mathfrak{s}\!\Cat_2}
\DeclareMathOperator{\CCAT}{\CAT_2}

\newcommand{\LTop}{\operatorname{Top}^{\operatorname{L}}}
\newcommand{\LTTop}{\operatorname{2Top}^{\operatorname{L}}}

\DeclareMathOperator{\Fun}{Fun}
\DeclareMathOperator{\mFun}{\mathfrak{m}Fun}

\DeclareMathOperator{\FUN}{\mathbb{F}un}
\DeclareMathOperator{\PSh}{PSh}
\DeclareMathOperator{\Ind}{Ind}
\DeclareMathOperator{\Shv}{Sh}

\DeclareMathOperator{\Loc}{Loc}

\DeclareMathOperator{\Grpd}{Grpd}

\DeclareMathOperator{\Grp}{Grp}

\DeclareMathOperator{\Sml}{Small}

\DeclareMathOperator{\Free}{Free}
\DeclareMathOperator{\Cofree}{Cofree}

\DeclareMathOperator{\Cart}{Cart}
\DeclareMathOperator{\eCart}{{Cart}_E}

\DeclareMathOperator{\Cocart}{Cocart}

\DeclareMathOperator{\Fib}{Fib}
\DeclareMathOperator{\FIB}{\mathbb{F}\!\operatorname{ib}}
\DeclareMathOperator{\BFib}{BFib}

\newcommand{\AR}{\operatorname{AR}}

\newcommand{\ARplax}{\AR^{\oplax}}
\newcommand{\ARlax}{\AR^{\lax}}

\DeclareMathOperator{\Adj}{Adj}
\DeclareMathOperator{\Mnd}{Mnd}

\DeclareMathOperator{\const}{const}

\DeclareMathOperator{\pr}{pr}
\DeclareMathOperator{\id}{id}
\DeclareMathOperator{\ev}{ev}

\DeclareMathOperator{\res}{res}
\DeclareMathOperator{\glue}{gl}

\DeclareMathOperator{\Un}{Un}

\newcommand{\cart}{{\operatorname{cart}}}

\newcommand{\fin}{{\operatorname{fin}}}

\newcommand{\cocart}{{\operatorname{cocart}}}
\newcommand{\cc}{{\operatorname{cc}}}

\newcommand{\op}{{\operatorname{op}}}

\newcommand{\lax}{{\operatorname{lax}}}
\newcommand{\epsilonlax}{{\epsilon\!\operatorname{lax}}}
\newcommand{\Eepsilonlax}{{E\text{-}}\epsilon\!\operatorname{lax}}
\newcommand{\oplax}{{\operatorname{oplax}}}
\newcommand{\elax}{{E\text{-}\!\operatorname{lax}}}
\newcommand{\eoplax}{{E\text{-}\!\operatorname{oplax}}}
\newcommand{\laxoplax}{{\operatorname{(op)lax}}}
\newcommand{\elaxoplax}{{E\text{-}\!\operatorname{(op)lax}}}
\newcommand{\arglax}[1]{#1\text{-}\operatorname{lax}}
\newcommand{\argoplax}[1]{#1\text{-}\operatorname{oplax}}

\newcommand{\cocont}[1]{{#1\operatorname{-cc}}}
\newcommand{\cont}{\operatorname{cont}}

\newcommand{\core}{\simeq}

\newcommand{\co}{{\operatorname{co}}}
\newcommand{\coop}{{\operatorname{coop}}}

\makeatletter
\newcommand{\fixed@sra}{$\vrule height 2\fontdimen22\textfont2 width 0pt\rightarrow$}
\newcommand{\shortarrowup}[1]{%
  \mathrel{\text{\rotatebox[origin=c]{65}{\fixed@sra}}}
}
\newcommand{\shortarrowdown}[1]{%
  \mathrel{\text{\rotatebox[origin=c]{250}{\fixed@sra}}}
}
\makeatother

\newcommand{\upslash}{\!\shortarrowup{1}}
\newcommand{\downslash}{\!\shortarrowdown{1}}

\let\lim=\relax
\DeclareMathOperator{\lim}{lim}
\DeclareMathOperator{\colim}{colim}

\newcommand{\Hom}{\underline{\operatorname{Hom}}}
\newcommand{\Nat}{{\operatorname{Nat}}}

\newcommand{\Over}[2]{#1_{/#2}}

\newcommand{\LaxUnder}[2]{#1_{#2\nearrow}}
\newcommand{\LaxOver}[2]{#1_{\nearrow#2}}
\newcommand{\OplaxUnder}[2]{#1_{#2\swarrow}}
\newcommand{\OplaxOver}[2]{#1_{\swarrow#2}}

\newcommand{\orientedtimeslr}{\mathbin{\overleftarrow{\times}}}
\newcommand{\orientedtimesrl}{\mathbin{\overrightarrow{\times}}}

\def\lra{\longrightarrow}
\def\lla{\longleftarrow}
\def\hra{\hookrightarrow}
\def\llra{\def\arraystretch{.1}\begin{array}{c} \lra \\ \lla \end{array}}

\NewDocumentCommand{\Gen}{m o}{%
	\IfNoValueTF{#2}{%
		\langle #1\rangle%
	}{%
		\langle #1\rangle_{#2}%
	}%
}

\makeatletter
\g@addto@macro\bfseries{\boldmath}
\makeatother

\title{($\infty,2$)-topoi and descent}
\author{Fernando Abell\'an}
\author{Louis Martini}
\date{\today}

\begin{document}
\begin{abstract}
	We set the foundations of a theory of Grothendieck $(\infty,2)$-topoi based on the notion of \emph{fibrational descent}, which axiomatizes both the existence of 
	a classifying object for
	fibrations internal to an $(\infty,2)$-category as well as the exponentiability of these fibrations. As our main result, we prove a 2-dimensional version of Giraud's theorem which characterizes $(\infty,2)$-topoi as those $(\infty, 2)$-categories that appear as localizations of $\mathfrak{C}\!\operatorname{at}$-valued presheaves in which the localization functor preserves certain partially lax finite limits which we call \emph{oriented pullbacks}. We develop the basics of a theory of partially lax Kan extensions internal to an $(\infty,2)$-topos, and we show that every $(\infty,2)$-topos admits an internal version of the Yoneda embedding. Our general formalism recovers the theory of categories internal to a $(\infty,1)$-topos (as develop by the second author and Sebastian Wolf) as a full sub-$(\infty,2)$-category of the $(\infty,2)$-category of $(\infty,2)$-topoi. As a technical ingredient, we prove general results on the theory of presentable $(\infty,2)$-categories, including lax cocompletions and 2-dimensional versions of the adjoint functor theorem, which might be of independent interest.
\end{abstract}

\maketitle

\tableofcontents

\input{intro}

\input{prelim}

\input{laxcocompletion}
\input{internfib}

\input{fibdescent}

\input{localic}

\bibliographystyle{halpha}
\bibliography{references.bib}

\end{document}

%% file: intro.tex
\section{Introduction}\label{sec:intro}

\subsection*{Motivation}
Topos theory, originally developed by Grothendieck and Verdier in \cite{GVTopos}, emerged as a generalization of sheaf theory which provided the natural habitat for more sophisticated algebro-geometric invariants such as étale cohomology and has found since then plentiful applications in geometry. Independently, the notion of a topos emerged from the works of Lawvere and Tierney (\cite{LT}) as a model for intuitionistic set theory, i.e.\ an abstraction of the category of sets. From this point of view, topoi can be viewed as mathematical universes in which constructive first order logic can be interpreted, so that consequently the topos axioms can be regarded as an axiomatic system for intuitionistic mathematics itself.

In modern, homotopy-coherent mathematics, the concept of a topos has been largely replaced by that of an $(\infty,1)$-topos as developed by Lurie and Joyal, \cite{LurieHTT}. The key innovation is that the latter notion incorporates \emph{homotopical} information, so that an $(\infty,1)$-topos does not only provide a model for intuitionistic set theory, but also for an (intuitionistic) theory of \emph{homotopy types}. The most well-known incarnation of the latter is \emph{homotopy type theory} (\cite{hottbook}), which has been shown by Shulman to have a model in every $(\infty,1)$-topos (\cite{shulman2019}). In other words, the $(\infty,1)$-topos axioms can be viewed as an axiomatic system for homotopy-coherent mathematics.

The atomic concept in homotopy type theory is that of (synthetic) \emph{$\infty$-groupoids}. In other words, the native language of an $(\infty,1)$-topos is \emph{undirected}: a priori, it is entirely agnostic about non-invertible paths. If one wishes to incorporate  directedness into this framework, one is forced to build an additional simplicial layer on top of the groupoidal theory. To a certain degree, this phenomenon is already present in classical topos theory: here the atomic concept is that of a (synthetic) \emph{set}. Hence, if one wishes to incorporate directedness by means of interpreting the theory of \emph{posets} in a topos, one is required to work with certain reflexive graphs, i.e.\ truncated simplicial objects internal to that topos. For applications in which the basic objects of interest are directed, i.e.\  $(\infty,1)$-categories rather than $\infty$-groupoids, such an indirect approach seems impractical. Rather, one would desire to have a framework in which the atomic objects already contain a notion of directedness, i.e.\ can be interpreted as (synthetic) $(\infty,1)$-categories. By now, there are several promising approaches that aim at providing such a framework, both as a syntactic theory (\cite{riehlshulman,simplicialhomotopy}) as well as through a categorical model (\cite{riehl2021, formalization}). 

What is notably absent in either of these approaches is the topos-theoretic point of view: as outlined above, the $(\infty,1)$-topos axioms themselves can be regarded as an axiomatic framework for the study of synthetic $\infty$-groupoids, so that a possible strategy for developing a theory of synthetic $(\infty,1)$-categories could be to identify a directed, i.e.\ $(\infty,2)$-categorical analogue of these axioms. This would then lead to an $(\infty,2)$-categorical version of $(\infty,1)$-topos theory for which the atomic concept is that of synthetic $(\infty,1)$-categories instead of synthetic $\infty$-groupoids. The goal of this work is to provide such an axiomatic system, giving rise to the notion of an \emph{$(\infty,2)$-topos}, and to investigate in which way the resulting theory can be used to argue synthetically about $(\infty,1)$-categories.

\subsection*{$(\infty,2)$-Topoi}

Just as (co)cartesian  fibrations are central to the study of $(\infty,1)$-categories, one should anticipate that a synthetic theory of fibrations will play a prominent role in synthetic $(\infty,1)$-category theory. Our approach to $(\infty,2)$-topoi takes a rather radical stance on this view and can be summarized with the following slogan:
\begin{displayquote}
	“An $(\infty,2)$-topos is a presentable $(\infty,2)$-category where internal fibrations satisfy a local-to-global principle.”
\end{displayquote}
This local-to-global principle, which we call \emph{fibrational descent}, is one of the main contributions of this project and is inspired by the descent axiom for $(\infty,1)$-topoi, due to Lurie and Rezk (\cite[Section~6.13]{LurieHTT}). Before we can give a description of our axiom, let us first explain what we mean by a fibration internal to an $(\infty,2)$-category. To that end, it will be useful to recast the notion of a cocartesian fibration of $(\infty,1)$-categories from an $(\infty,2)$-categorical perspective:

Let $p \colon \XX \to \CC$ be a functor of $(\infty,1)$-categories and consider the commutative triangle
\[\begin{tikzcd}
	\XX && {\Free^{0}_{\CC}(\XX) } \\
	& \CC
	\arrow["\eta_p",from=1-1, to=1-3]
	\arrow[from=1-1, to=2-2]
	\arrow[from=1-3, to=2-2]
\end{tikzcd}\]
where $\Free^{0}_{\CC}(\XX)\simeq \Fun([1],\CC)\times_{\CC}\XX$ denotes the free cocartesian fibration (see \cite{GHN}) on the functor $p$ and $\eta_p$ is the functor that carries $x\in\XX$ to $\id_{p(x)}\in\Free_{\CC}^0(\XX)$. Then it follows that $p$ is a cocartesian fibration if and only if $\eta_p$ admits a relative left adjoint $\ell_p$.\footnote{This is a classical observation due to Street, see \cite{streetfib}.} Similarly, given a functor $f \colon \XX \to \YY$ between cocartesian fibrations over $\CC$, it turns out that $f$ preserves cocartesian edges if and only if the commutative diagram
\[
 	\begin{tikzcd}
 		\XX \arrow[d,"\eta_p"]  \arrow[r,"f"]  &  \arrow[d,"\eta_q"] \YY \\
 		\Free^{0}_{\CC}(\XX) \arrow[r,"\Free^{0}_{\CC}(f)"] & \Free^{0}_{\CC}(\XX)
 	\end{tikzcd}
 \] 
 is (vertically) left adjointable. Now a crucial observation to make is that $\Free_{\CC}^0(\XX)$ fits into a universal laxly commuting diagram 
 \[\begin{tikzcd}
	{\Free^{0}_{\CC}(\XX)} & \XX \\
	\CC & \CC.
	\arrow[from=1-1, to=1-2]
	\arrow[from=1-1, to=2-1]
	\arrow[Rightarrow,shorten <=10pt, shorten >=10pt, from=1-2, to=2-1]
	\arrow["p", from=1-2, to=2-2]
	\arrow["\id"', from=2-1, to=2-2]
\end{tikzcd}\]
In other words, $\Free^{0}_{\CC}(\XX)$ can be expressed as a certain partially (op)lax limit that we refer to as an \emph{oriented pullback}. Consequently, every $(\infty,2)$-category $\bCC$ with oriented pullbacks admits an internal theory of \emph{0-fibrations} which reduce to cocartesian fibrations in the case where $\bCC=\Cat_{(\infty,1)}$. Likewise, one can develop an internal theory of \emph{1-fibrations}, generalising cartesian fibrations in $\Cat$. Both theories are fully $(\infty,2)$-functorial under strong pullbacks, which can be expressed by the fact that there are functors
\[
	\Fib_{\bCC}^0 \colon \bCC^\op \to \Cat_{(\infty,2)}, \enspace \enspace \Fib_{\bCC}^{1} \colon \bCC^{\coop} \to \Cat_{(\infty,2)},
\]
which assign to each object $c \in \bCC$ the corresponding $(\infty,2)$-category of 0- and 1-fibrations, respectively.

Using this internal theory of fibrations, we can now formulate the fibrational descent axiom. Recall from \cite[Section~6.13]{LurieHTT} that for a (presentable) $(\infty,1)$-category $\CC$, descent simply means that the functor
\begin{equation*}
	\CC_{/-}\colon\CC^\op\to\Cat_{(\infty,1)}
\end{equation*}
preserves small limits. One should think of the two functors $\Fib_{\bCC}^0$ and $\Fib_{\bCC}^1$ that we introduced above as $(\infty,2)$-categorical versions of $\CC_{/-}$. Hence, fibrational descent asserts that $\Fib_{\bCC}^0$ preserves (small) partially oplax limits: for every decomposition $c=\colim^{\elax}_{i\in \bII} c_i$ in $\bCC$ (where $E$ is some marking of the indexing $(\infty,2)$-category $\bII$), the canonical map
\begin{equation*}\label{eq:FibComparisonMap}\tag{$\ast$}
	\Fib^0_{\bCC}(c)\to \lim^{\eoplax}_{i\in\bII}\Fib^0_{\bCC}(c_i)
\end{equation*}
is an equivalence. Dually, fibrational descent also asserts that $\Fib_{\bCC}^1$ preserves (small) partially lax limits. However, it turns out that this is not quite enough: in the $(\infty,1)$-categorical context, descent automatically implies that colimits are universal in $\CC$, i.e.\ that every map in $\CC$ is exponentiable. The correct $(\infty,2)$-categorical translation of this property is exponentiability of 0- and 1-fibration in $\bCC$, but this does not automatically follow from asserting that $\Fib_{\bCC}^0$ and $\Fib_{\bCC}^1$ preserve partially (op)lax limits. Rather, this property will be satisfied once we additionally impose that the inverse of the equivalence in \eqref{eq:FibComparisonMap} takes on a specific form: its inverse is required to act by carrying a compatible family $(x_i\to c_i)_{i\in \bII}$ of $0$-fibrations to its $E$-lax colimit $\colim^{\elax}_{i\in\bII} x_i\to c$ in $\bCC$, and likewise for $1$-fibrations.

These two statements combined now comprise our axiom of fibrational descent. Using the language of \emph{marked cartesian transformations}, which is a suitable $(\infty,2)$-categorical analogue of the notion of cartesian transformation found in \cite[Definition 6.1.3.1]{LurieHTT}, this axiom can be phrased as follows:

\begin{fibdescent}
	Let $\bCC$ be a presentable $(\infty,2)$-category. We say that $\bCC$ satisfies $0$-\emph{fibrational descent}
	if for every marked $(\infty,2)$-category $(\bII,E)$, the following conditions hold:
	\begin{itemize}
			\item[F1)] Let $\overline{\alpha}\colon\overline p \xRightarrow{} \overline q$ be a natural transformation
			between $E$-lax cones
			$\overline{p},\overline{q}\colon \bII^{\colimcone}_{\lax} \to \bCC$ (see \cref{rem:conesimplified}) such that the restriction $\alpha\colon p\Rightarrow q$ of $\overline{\alpha}$ along the inclusion $\bII\into \bII^\colimcone_{\lax}$ is $(E,0)$-cartesian and such that $\overline{q}$ is an $E$-lax colimit cone.
			Then
			the following are equivalent:
			\begin{enumerate}
				\item The functor $\overline{p}$ defines an $E$-lax colimit cone.
				\item The natural transformation $\overline{p} \xRightarrow{}\overline{q}$ is 
				$(E^{\colimcone}_{\lax},0)$-cartesian.
			\end{enumerate}
			\item[F2)] Let $\overline\alpha\colon \overline{p} \Rightarrow \overline{k}$ and $\overline{\beta}\colon \overline{k} \Rightarrow \overline{q}$ be 
			natural transformations between lax colimit cones $\overline{p},\overline{k},\overline{q}\colon \bII^\colimcone_{\epsilonlax}\to \bCC$ such that the restrictions $\alpha \circ \beta$ and $\beta$ of $\overline{\alpha}\circ\overline{\beta}$ and $\overline{\beta}$, respectively, along the inclusion $\bII\into\bII^\colimcone_{\lax}$ are $(E,0)$-cartesian. 
			Then $\alpha$ is point-wise a map of $0$-fibrations if and only if $\overline{\alpha}$ is.
		\end{itemize}
		Dually, we say that $\bCC$ satisfies \emph{$1$-fibrational descent} if $\bCC^\co$ satisfies $0$-fibrational descent. We say that $\bCC$ satisfies \emph{fibrational descent} if it satisfies both $0$- and $1$-fibrational descent.
\end{fibdescent}
 Finally, we define $(\infty,2)$-topoi to be presentable $(\infty,2)$-categories (which is a technical condition ensuring that they admit all partially lax colimits whilst not being too large) that satisfy fibrational descent.

A substantial reason why $(\infty,1)$-topoi are convenient to work with is that they can be defined in multiple ways, each one shedding light onto a different aspect of the theory. We will provide a similar characterisation of $(\infty,2)$-topoi: in terms of fibrational descent as explained above, via an $(\infty,2)$-categorical version of the \emph{Lawvere-Tierney axioms}, and finally via \emph{Giraud's theorem}.

\begin{description}
	\item[Lawvere-Tierney axioms] Let $u \colon \colim_{\bII}^{\elax}a_i=a \to c$ be a morphism in $\bCC$, and consider a pullback diagram
	\[
	\begin{tikzcd}
		a \times_c x \arrow[d,"\pi"] \arrow[r] & x \arrow[d,"p"] \\
		a \arrow[r,"u"] & c  
	\end{tikzcd}
	\]
	where $p$ is 0-fibration. Then fibrational descent yields an equivalence $a \times_c x \simeq \colim_{\bII}^{\elax}(a_i \times_c x)$. Therefore, we conclude that the pullback functor
	\[
	p^{*} \colon \bCC_{/c} \to \bCC_{/x}
	\]
	preserves partially lax colimits. Invoking the machinery of presentable $(\infty,2)$-categories (cf. \cref{subsec:adjointFunctorTheorems}), we conclude that $p^*$ admits a right adjoint (and similary in the case where $p$ is a 1-fibration). This situation is summarized in the following statement:
	\begin{lt1}
	If $\bCC$ is an $(\infty,2)$-topos then fibrations are exponentiable (cf. \cref{def:exponentiability}).
	\end{lt1}
	Note how this statement stands in contrast to $(\infty,1)$-topoi where \emph{all} morphisms are exponentiable, but this cannot be true in an $(\infty,2)$-topos as it already fails in $\Cat$, which ought to be the archetypical $(\infty,2)$-topos.
	
	Let $(-)^{\leq 1} \colon \Cat_{(\infty,2)} \to \Cat_{(\infty,1)}$ be the 2-functor which sends an $(\infty,2)$-category to its underlying $(\infty,1)$-category. Ignoring certain set-theoretical technicalities (which we address in \cref{subsec:LTaxioms}), it follows from fibrational descent (invoking again presentability) that the presheaves $(\Fib_{\bCC}^{0})^{\leq 1}$ and $(\Fib_{\bCC}^1)^{\leq 1}$ are representable by objects $\Omega^{0},\Omega^{1} \in \bCC$. In other words, we find:
	\begin{lt2}
		Any $(\infty,2)$-topos $\bCC$ admits classifying objects for the theory of fibrations.
	\end{lt2}
	It turns out that a presentable $(\infty,2)$-category satisfying \textbf{LT1} and \textbf{LT2} is already an $(\infty,2)$-topos, so that these two conditions can serve as an alternative axiomatic approach to $(\infty,2)$-topos theory.
	\item[Giraud's theorem] An $(\infty,1)$-topos can be characterized as an $(\infty,1)$-category $\CC$ that admits a presentation as an accessible, left exact Bousfield localisation of an $(\infty,1)$-categories of presheaves on a small $(\infty,1)$-category. Likewise, we will show that an $(\infty,2)$-category $\bCC$ is an $(\infty,2)$-topos precisely if there is an adjunction
	\[
	L \colon \FUN(\bII^\op,\Cat_{(\infty,1)}) \llra \bCC \colon R
	\]
	where $R$ is fully faithful and accessible (meaning $R^{\leq 1}$ is an accessible functor of $(\infty,1)$-categories) and where $L$ preserves \emph{oriented pullbacks}: lax versions of pullbacks commuting up to non-invertible 2-cells. Note that in light of the fact that oriented pullbacks are the building blocks for the theory of internal fibrations, such a statement should not come as a surprise.
\end{description}

  We can summarize the different perspectives on $(\infty,2)$-topos theory as follows:
	\begin{theoremint}\label{thmint:def}
		Let $\bCC$ be a presentable $(\infty,2)$-category. Then the following are equivalent:
		\begin{enumerate}
			\item The $(\infty,2)$-category $\bCC$ satisfies {fibrational descent}.
			\item Fibrations in $\bCC$ are exponentiable, and for sufficiently large regular cardinals $\kappa$, there exist classifiers $\Omega^{\epsilon,\kappa} \in \bCC$ for relatively $\kappa$-compact $\epsilon$-fibrations, where $\epsilon \in \{0,1\}$.
			\item There exists a small $(\infty,2)$-category $\bII$ and an adjunction
			\[
				L \colon \FUN(\bII,\Cat_{(\infty,1)}) \llra \bCC \colon R
			\]
			such that $R$ is fully faithful and accessible and $L$ preserves oriented pullbacks.
		\end{enumerate}
		If $\bCC$ satisfies any of these conditions we say that $\bCC$ is an $(\infty,2)$-topos.
	\end{theoremint}
We would like to mention that in upcoming work by Loubaton, an additional characterization of $(\infty, 2)$-topoi is expected to appear, based on the notion of effective congruences in $(\infty, n)$-categories.

\subsection*{Synthetic $(\infty,1)$-category theory in an $(\infty,2)$-topos}
By design, the notion of an $(\infty,2)$-topos ought to capture and axiomatise the main structural properties of the $(\infty,2)$-category $\Cat_{(\infty,1)}$, in the same way as the notion of an $(\infty,1)$-topos is an axiomatisation of the properties of the $(\infty,1)$-category $\Spc$ of $\infty$-groupoids. In other words, the objects of an $(\infty,2)$-topos $\bCC$ can be regarded as synthetic $(\infty,1)$-categories, and one should expect that a fair amount of the theory of $(\infty,1)$-categories can be developed in this context. In this paper, we do not aim to give a systematic treatment of such an internal theory of $(\infty,1)$-categories, but rather focus on a few fundamental building blocks of that theory which can be derived from our axiom of fibrational descent: 
\begin{description}
	\item[$\infty$-Groupoids in an $(\infty,2)$-topos] Any $(\infty,2)$-topos $\bCC$ comes equipped with an intrinsic notion of $\infty$-groupoids: an object $x\in\bCC$ is an internal $\infty$-groupoid if it is $(\infty,0)$-truncated, i.e.\ if the presheaf $\bCC(-,x)$ takes values in $\Spc$. One obtains a full sub-2-category $\Grpd(\bCC)\subset\bCC$ of internal $\infty$-groupoids. By design, this is an $(\infty,1)$-category, and the inclusion admits a left adjoint $\lvert-\rvert$ as well as a right adjoint $(-)^\core$ on the level of the underlying $(\infty,1)$-categories. Thus, one can make sense of both the \emph{groupoidification} $\lvert c\rvert$ and the  \emph{groupoid core} $c^\core$ of an object $c\in\bCC$.
	
	An interesting aspect of our theory is that the groupoid core functor $(-)^\core$ can behave in a somewhat unexpected way: there are examples (see \cref{rem:emptycores}) of $(\infty,2)$-topoi $\bCC$ that have a non-trivial object $c\in\bCC$ (meaning that $c$ is not the initial object) with an empty groupoid core, i.e.\ where $c^\core$ is the initial object. This phenomenon is an obstruction for the validity of several classical theorems from higher category that rely on the ability to check certain properties of $(\infty,1)$-categories (or functors between $(\infty,1)$-categories) object-wise. For example, it is no longer true internal to an arbitrary $(\infty,2)$-topos that equivalences between fibrations can be detected object-wise, or that essentially surjective and fully faithful functors are automatically equivalences. This shows that the theory of $(\infty,2)$-topoi is substantially more general than most approaches to synthetic $(\infty,1)$-category theory, and that only a fraction of the latter can be interpreted internally in an $(\infty,2)$-topos.
	
	\item[Directed univalence]
	Every $(\infty,2)$-topos admits an internal mapping functor $\underline{\bCC}(-,-) \colon \bCC^\op \times \bCC \to \bCC$ which assings to each pair of synthetic $(\infty,1)$-categories its associated synthetic $(\infty,1)$-category of functors. On the other hand, for every $c \in \bCC$, we can construct an internal mapping groupoid by means of the oriented pullback
	\[\begin{tikzcd}
	{\Hom_c(x,y)} & \ast \\
	\ast & c.
	\arrow[from=1-1, to=1-2]
	\arrow[from=1-1, to=2-1]
	\arrow["y", from=1-2, to=2-2]
	\arrow[Rightarrow,shorten <=9pt, shorten >=6pt, from=2-1, to=1-2]
	\arrow["x"', from=2-1, to=2-2]
\end{tikzcd}\]
    Since every synthetic $(\infty,1)$-category $a$, determines a unique map from terminal object to one of the classifiers $a \colon \ast \to \Omega^{\epsilon}$, it is natural to ask what is the relationship between the groupoidal core of the synthetic functor category $\underline{\bCC}(a,b)^{\core}$, and the internal mapping groupoid object $\Hom_{\Omega^{\epsilon}}(a,b)$ of the classifyer. As expected, our formalism allows us to demonstrate (see \cref{thm:Univalence}) that both constructions are naturally equivalent thereby establishing that $(\infty, 2)$-topoi satisfy a directed analogue of the univalence axiom from homotopy type theory (see \cite{hottbook} for more details).
	
	\item[Synthetic Kan extensions]
	Given a morphism $f \colon c \to d$ in an $(\infty,2)$-topos and a fibration $p \colon x \to c$ then it is possible to functorially produce fibrations $f_!(p) \colon f_!x \to d$ and $f_*(p) \colon f_* x \to d$ yielding suitably internal versions of left and right Kan extensions respectively (cf. \cref{thm:basechangetheorem} and \cref{thm:kaninternal}). A surprisingly pleasant feature of $(\infty,2)$-topoi, is that they allow for a theory of partially lax Kan extensions. Given a class of morphisms $\Xi=\{a_i \to c\}_{i \in I}$ we define a map between fibrations $p\colon x \to c,\enspace q \colon y \to c$ to be $\Xi$-lax if its becomes an  morphism of fibrations after pulling back along each $a_i \to c$ in $\Xi$. It turns out (cf. \cref{thm:laxkan}) that given a fibration $p \colon x \to c$, one can functoriality construct a fibration $f_!^{\Xi}(p) \colon f_!^{\Xi}x \to d$ (resp. $f_*^{\Xi}(p)\colon f_*^{\Xi}x \to d$) with the property that the data of morphism of fibrations out of $f_!^{\Xi}(p)$ corresponds precisely to the data of an $\Xi$-lax morphism out of $p$ (and dually for $f_*^{\Xi}(p)$).
	
	\item[The Yoneda embedding]
	 Let $\CC$ be an $(\infty,1)$-category, and let $\Cart(\CC)$ denote the $(\infty,1)$-category of cartesian fibrations over $\CC$. The Grothendieck construction induces a natural equivalence 
	 \[
	 	\Fun(\AA, \Cart(\CC)) \simeq \BFib_{(\AA, \CC)}^{0},
	 \]
	where the later denotes the $(\infty,1)$-category of bifibrations (also studied under the name curved-orthofibrations in \cite{TwoVariable}. In particular, the bifibration 
	$\ev_0 \times \ev_1 \colon \Fun([1], \CC) \to \CC \times \CC$ defines a functor $h_\CC \colon\CC \to \Cart(\CC)$, which corresponds to the Yoneda embedding. In \cref{subsec:yoneda}, we introduce an internal version of the theory of bifibrations, and demonstrate that fibrational descent ensures that the functor
	 \[
	 	\BFib^{0}_{(-,c)} \colon \bCC^{\op} \to \CCat, \enspace a \mapsto  \BFib^{0}_{(a,c)},
	 \]
	 preserves partially (op)lax limits. Therefore, we obtain, again ignoring set theoretical issues, an object $\underline{\Fib}_{/c}^{1}$ representing the functor $\left(\BFib^{0}_{(-,c)}\right)^{\leq 1}$. We then establish (cf. \cref{thm:Yoneda}) the following :
	 \begin{theoremint}\label{theoremint:yoneda}
	 	The canonical map induced by the universal property of the cotensor $\ev_0 \times \ev_1 \colon c^{[1]} \to c \times c$, defines a bifibration and thus yields a morphism 
	 	\[
	 		h_c \colon c \to \underline{\Fib}_{/c}^{1}
	 	\]
	 	which is fully faithful.
	 \end{theoremint}
\end{description}

\subsection*{($\infty$,1)-localic $(\infty,2)$-topoi}
It is a standard fact in the theory of $(\infty,1)$-topoi that for a given $(\infty,1)$-topos $\XX$, its full subcategory of $(n-1)$-truncated objects, $\tau_{\leq n-1}\XX$, is an $(n,1)$-topos for $0 \leq n \leq \infty$. Moreover, the assignment $\XX \mapsto \tau_{\leq n-1}\XX$ admits a fully faithful left adjoin that carries an $(n,1)$-topos $\YY$ to the associated  \emph{$(n,1)$-localic} $(\infty,1)$-topos $\Shv(\YY)$ of \emph{sheaves} on $\YY$. As one of our main results, we study a categorification of this adjunction: as an elementary application of fibrational descent (see \cref{thm:underlying1topos}), we show that for every $(\infty,2)$-topos $\bCC$, the full subcategory $\Grpd(\bCC)$ of synthetic $\infty$-groupoids, which are by definition the $(\infty,0)$-truncated objects in $\bCC$, forms an $(\infty,1)$-topos. Conversely, to every $(\infty,1)$-topos $\XX$ we may assign the associated \emph{$(\infty,1)$-localic} $(\infty,2)$-topos $\Shv_{\Cat_{(\infty,1)}}(\XX)$ of \emph{sheaves} on $\XX$, i.e.\ of $\Cat_{(\infty,1)}$-valued presheaves on $\XX$ that preserve strong limits. We then obtain the following categorification of the localic reflection theorem:
 \begin{theoremint}\label{theoremint:localic}
	The functor $\Shv_{\Cat_{(\infty,1)}}(-)$ exhibits the $(\infty,2)$-category of $(\infty,1)$-topoi as a coreflective sub-2-category of the $(\infty,2)$-category of $(\infty,2)$-topoi, where the associated coreflection functor is given by $\Grpd(-)$.
\end{theoremint}
The synthetic $(\infty,1)$-category theory of $(\infty,1)$-localic $(\infty,2)$-topoi, i.e.\ those that are in the essential image of the functor $\Shv_{\Cat_{(\infty,1)}}(-)$, has been extensively studied by the second-named author and Sebastian Wolf (\cite{Martini2022}, \cite{MWPres},\cite{MWIntTop}). Their work shows that the synthetic $(\infty,1)$-category theory of $(\infty,1)$-localic $(\infty,2)$-topoi behaves quite classically -- essentially every higher categorical argument can be shown to be valid this context. The $(\infty,1)$-localic approximation of an $(\infty,2)$-topos can thus be regarded as a measure to what extent the internal logic of this $(\infty,2)$-topos exhibits non-classical behaviour.

More precisely, if $\bCC$ is an $(\infty,2)$-topos and if
\begin{equation*}
	L\colon\Shv_{\Cat_{(\infty,1)}}(\Grpd(\bCC))\leftrightarrows \bCC
\end{equation*}
is the counit of the adjunction in \cref{theoremint:localic}, then $L$ admits a right adjoint $R$ that sends $c\in\bCC$ to the sheaf $\bCC(\iota(-),c)$ (where $\iota\colon\Grpd(\bCC)\subset\bCC$ denotes the inclusion). Moreover, the counit of the adjunction $L\dashv R$ can be computed as the canonical map
\begin{equation*}
	\colim^{\elax}_{(x\to c)\in\LaxOver{\Grpd(\bCC)}{c}} x\to c
\end{equation*}
in which $\LaxOver{\Grpd(\bCC)}{c}$ is the lax slice $(\infty,2)$-category of synthetic $\infty$-groupoids in $\bCC$ over $c$ and where $E$ is the collection of maps that are contained in the \emph{strong} slice $\Over{\Grpd(\bCC)}{c}\subset\LaxOver{\Grpd(\bCC)}{c}$. One may regard the above map as the canonical approximation of $c$ by a lax colimit of synthetic $\infty$-groupoids. Thus, the $(\infty,1)$-localic approximation of $\bCC$ measures to what extend an object $c\in\bCC$ can be built as lax colimit of synthetic $\infty$-groupoids. In the most extreme case, where $c$ has an empty groupoid core, the left-hand side of the above map is the initial object in $\bCC$, so that the above map gauges how much information about $c$ is contained in its groupoid core.

\subsection*{Relation to previous work}
\begin{description}
	\item[Presentable $(\infty,2)$-categories]
	The more general theory of presentable $\scr{V}$-enriched categories has be studied by Mazel-Gee and Stern in \cite{secondaryK}, as well as by Heine (\cite{heineweighted}). Our approach to presentable $(\infty,2)$-categories does not rely in the formalism of enriched $\infty$-categories and emphasizes the use of partially (op)lax limits over the use of the (equivalent) theory of weighted (co)limits. 
	\item[Strict 2-topoi]  The notion of an strict 2-topos has been explored in the literature, most notably in the work of Weber (\cite{Web}), and further developed by Mesiti (\cite{mesiti}) and Helfer (\cite{helfer}). These approaches differ fundamentally from ours in that they postulate (or prove) the existence of an “op” duality, similar to the involution $\Cat \to \Cat^\co$, which sends a category to its opposite. In \cref{ex:monad} we give an example of an $(\infty,2)$-topos where such involution cannot possibly exist. In fact, we expect this to be the generally the case outside the class of $(\infty,1)$-localic $(\infty,2)$-topoi where a well-defined notion of “op” involution always exists (cf. \cref{prop:localicinvolutive}).
\end{description}

\subsection*{Structure of the paper}
 We commence the paper reviewing the basic ingredients of the theory $(\infty,2)$-categories which will be needed throughout this work in \cref{sec:prelim}. In \cref{sec:colimits}, we proceed to establish several results in the theory of partially lax colimits, including a decomposition theorem for colimit diagrams (cf. \cref{prop:diagramdecomposition}) and some basic results about oriented pullbacks (cf. \cref{subsec:oriented}). Later on in \cref{sec:kancocompletion}, we study $(\infty,2)$-categorical Kan extensions (using results from the first-named author) which we apply to the study of free cocompletions (cf. \cref{sec:freeCocompletion}). \cref{sec:2presentable} is devoted to the study of presentable $(\infty,2)$-categories. We conclude \cref{sec:2presentable} by establishing 2-dimensional analogues of Freyd's adjoint functor theorem. In \cref{sec:internfib} we develop the theory of fibrations internal to an $(\infty,2)$-category admitting oriented pullbacks. 

 Once the preliminary theory is built, we present the axiom of fibrational descent in \cref{sec:fibdescent} and explore its consequences. We give general procedures to produce examples of $(\infty,2)$-topoi in \ref{subsec:const} and spend the rest of the chapter proving the equivalent characterizations of $(\infty,2)$-topoi discussed in the introduction. Our study of synthetic category theory is performed in \cref{sec:synth} including internal $\infty$-groupoids, directed univalence, synthetic Kan extensions and the synthetic version of Yoneda's lemma. We conclude this work in \cref{sec:localic} with our study of $(\infty,1)$-localic $(\infty,2)$-topoi.

\subsection*{A notational remark}
 From this point on, we will systematically drop the terminology $(\infty,2)$-category (resp. $(\infty,1)$-category) and simply write 2-category (resp. 1-category). 
\subsection*{Acknowledgements}
We would like to thank Thomas Blom, Rune Haugseng and David Kern for helpful conversations. We also thank Denis-Charles Cisinski for encouraging us to explore the notion of directed univalence. Finally, we would like to thank Felix Loubaton for openly discussing his own research on $(\infty,n)$-topoi with us and for pointing a further consequence of the 2-categorical Giraud's theorem (see \cref{rem:felix}).


%% file: prelim.tex
\section{Preliminaries}\label{sec:prelim}

\subsection{Background on 2-categories}
 In this section we gather basic definitions in the theory of 2-categories from a model independent perspective. There are several implementations (models) for the theory of 2-categories which will not play an important role in this paper. The key takeaway from the existence of such models for us can be summarized as follows:
 \begin{itemize}
 	\item There exists a large 2-category of small 2-categories $\CCat$ (resp.\  very large 2-category of large 2-categories $\widehat{\CCat}$). 
 	\item There exist certain elementary shapes (\cref{def:theta2}) which allows us to build every 2-category as certain colimits thereof.
 	\item There exist several truncation functors relating the theory of 1-categories and the theory of 2-categories (cf. \cref{def:underlying1cat} and \cref{def:localization}).
 \end{itemize}
  At the end of this section we review some formal properties of the Gray tensor product which will be key for our study of partially (op)lax colimits in \cref{sec:colimits}. We refer to the reader to \cite{GHLGray} for a more detailed discussion (within the model of scaled simplicial sets) and to \cite{Abellan2023} for a comparison of different models of the Gray tensor product.

 \begin{notation}
 	Every 2-category $\bCC$ comes equipped with a mapping category functor $\bCC^{\op} \times \bCC \to \Cat$ which will usually be denoted as $\bCC(-,-)$. We detail below some exceptions to this convention:
 	\begin{itemize}
 		\item If $\bCC=\Cat$ (the 2-category of small 1-categories) we will use the notation $\Fun(-,-)$. Similarly, we will use the notation $\Fun_{/\CC}$ for the mapping category functor of the slice $\Cat_{/\CC}$.
 		\item If $\bCC=\CCat$ we will also use the notation $\Fun(-,-)$. However, in some situations it will be important to consider the 2-category of functors, which will be denoted as $\FUN(-,-)$.
 		\item Similarly if $\bCC=\FUN(\bII,\bCC)$ (see above), we will denote by $\Nat_{\bII,\bCC}(-,-)$ the corresponding mapping category functor. 
 	\end{itemize}
 \end{notation}

 \begin{definition}\label{def:subcategories}
 	Let $f \colon \bCC \to \bDD$ be a functor of 2-categories. 
 	\begin{itemize}
 		\item We say that $f$ is fully faithful (or that $\bCC$ is a full sub-2-category of $\bDD$) if for every $x,y \in \bCC$ the induced functor on mapping categories
 		\[
 			\bCC(x,y) \xrightarrow{\simeq} \bDD(fx,fy)
 		\]
 		is an equivalence.
 		\item We say that $f$ is locally fully faithful (or that $\bCC$ is a locally full sub-2-category of $\bDD$) if for every $x,y \in \bCC$ the induced functor on mapping categories
 		\[
 			\bCC(x,y) \xrightarrow{} \bDD(fx,fy)
 		\]
 		is fully faithful.
 	\end{itemize}
 \end{definition}

 \begin{definition}\label{def:conserfunctor}
   Let $f \colon \bCC \to \bDD$ be a functor of 2-categories. We say that $f$ is conservative on 2-morphisms if given a 2-morphism $\alpha$ then it follows that $\alpha$ is invertible if and only if $f(\alpha)$ is invertible. More generaly given a family of functors $\{f_i \colon \bCC_i \to \bDD\}_{i \in I}$ we say that $\{f_i\}_{i\in I}$ are jointly conservative on 2-morphisms if given a 2-morphism $\alpha$ then it follows that $\alpha$ is invertible if and only if each $f_i(\alpha)$ is invertible.
 \end{definition}

 \begin{remark}\label{rem:conservativepullback}
   It follows from  \cite[Observation 2.6.16]{AGH24} that a functor is conservative on 2-morphisms if and only if it is right orthogonal to the canonical localization map $C_2 \to [1]$ from the walking 2-morphism to the walking arrow.
 \end{remark}

 \begin{definition}\label{def:ffinternal}
 	A morphism $f \colon c \to d$ in a 2-category $\bCC$ is said to be \emph{fully faithful} if the natural transformation $\bCC(-,c) \to \bCC(-,d)$ is point-wise fully faithful.
 \end{definition}

 \begin{definition}\label{def:coandop}
 	There are two canonical involutions that can be applied to 2-categories. Given $\bCC \in \CCat$ we can consider $\bCC^{\op}$ which is specified by the requirement $\bCC^{\op}(x,y) \simeq \bCC(y,x)$. One can also define $\bCC^{\co}$ which characterised by $\bCC^{\co}(x,y)=\bCC(x,y)^{\op}$. Observe that in contrast to $(-)^{\op}$ the $\co$-duality satisfies $\CC^{\co} \simeq \CC$ whenever $\CC$ is a 1-category.
 \end{definition}

\begin{definition}\label{def:underlying1cat}
	There exists an adjunction of large 2-categories,
	\[
		\iota \colon \Cat \llra  \CCat \colon (-)^{\leq 1}
	\]
	where the left adjoint is the canonical inclusion. We call $\bCC^{\leq 1}$, the \emph{underlying 1-category} of $\bCC$.
\end{definition}

\begin{definition}\label{def:localization}
	There exists an adjunction of large 2-categories
	\[
		|-|_1 \colon \CCat \llra \Cat \colon \iota
	\]
	where the right adjoint is given by the canonical inclusion and left adjoint sends a 2-category $\bCC$ to the 1-category $|\bCC|_1$ obtained by inverting every 2-morphism. We call $|\bCC|_1$, the \emph{1-truncation} of $\bCC$.
\end{definition}

\begin{definition}\label{def:theta2}
  Let $[n]([i_{1}],\dots,[i_{n}])$ be the (strict) 2-category
    with objects $0,\ldots,n$ and mapping categories
    \[ [n]([i_{1}],\dots,[i_{n}])(s,t) =
      \begin{cases}
        \emptyset, & s > t, \\
        [i_{s+1}] \times \cdots \times [i_{t}], & s \leq t,                   
      \end{cases}
    \]
    and composition given by the obvious isomorphisms
    \[[n]([i_{1}],\dots,[i_{n}])(s,t) \times [n]([i_{1}],\dots,[i_{n}])(t,u) \cong [n]([i_{1}],\dots,[i_{n}])(s,u).\] The category
    $\Theta_{2}$ is the full sub-1-category of $\CCat^{\leq 1}$ on these objects,
    and the restricted Yoneda embedding
    \[ \CCat^{\leq 1} \to \Fun(\Theta_{2}^{\op}, \SS)\] identifies $\CCat^{\leq 1}$
    with the full subcategory satisfying certain Segal and
    completeness conditions introduced by Rezk~\cite{RezkThetaN}.
 \end{definition}

\begin{definition}\label{def:kappasmall2cat}
  We say that a 2-category $\bCC$ is $\kappa$-small, where $\kappa$ denotes a regular cardinal, if there exists a $\kappa$-small  1-category $\II$ and a diagram $d \colon \II \to \Theta_2$ such that 
  \[
    \colim_{\II} (\iota_{\Theta}d) \simeq \bCC
  \]
  where $\iota_{\Theta} \colon \Theta_2 \to \CCat$ is the canonical inclusion.
\end{definition}

\begin{remark}
	We observe that the restriction of $|-|_1$ to $\Theta_2$ yields a functor $\Theta_2 \to \Delta$. In particular, it follows that a 1-category $\AA$ is $\kappa$-small in the sense of \cref{def:kappasmall2cat} if and only if there exists a diagram $d \colon \II \to \Delta$ where $\II$ is $\kappa$-small (in the usual sense) such that $\colim_{\II} (\iota_{\Delta}d) \simeq \AA$ where $\iota_{\Delta} \colon \Delta \to \Cat$ is the canonical inclusion.
\end{remark}

\subsubsection{Lax natural transformations and the Gray tensor product}

\begin{definition}\label{def:scaledcats}
	 A \emph{scaled} 2-category is 2-category $\bCC$ together with
  a collection $S$ of functors (the scaling) $S \subset \Fun([2] ,\bCC)^{\simeq}$, such that $S$
  contains the subgroupoid $S_{\bCC}^{\flat}$ of triangles
  \[
    \begin{tikzcd}[column sep=tiny,row sep=small]
      {} & y \arrow{dr}{g} \\
      x \arrow{ur}{f} \arrow{rr} & & z
    \end{tikzcd}
  \] where either $f$ or $g$ is an equivalence. We write
  $\bCC^{\flat} := (\bCC, S_{\bCC}^{\flat})$ for $\bCC$
  equipped with this minimal scaling, and
  $\bCC^{\natural} = (\bCC, S_{\bCC}^{\natural})$ for $\bCC$
  equipped with the maximal scaling consisting of all commuting
  triangles in $\bCC$. A functor of scaled 2-categories $f: (\bCC,S) \to (\bDD,T)$ is an oplax unital functor (see \cite{Abellan2023} for details) $\bII \to \bSS$ such that $f(S) \subseteq T$. We denote by $\CCats$ the 2-category of scaled 2-categories.
\end{definition}

\begin{remark}\label{rem:2morphlocalization}
	There is a functor $\iota \colon \CCat \to \CCats$ which sends $\bCC$ to the $\bCC^{\natural}$. The functor $\iota$ admits a left adjoint $|-|_{\mathfrak{s}} \colon \CCats \to \CCat$. 
\end{remark}

\begin{definition}\label{def:unmarkedgray}
	 Suppose $\bAA$ and $\bBB$ are 2-categories. Let $S$ be the set of
  commuting triangles in $\bAA \times \bBB$ of the form
  \[\left(\begin{tikzcd}
        & {a_1} &&& {b_1} \\
        {a_0} && {a_2,} &  {b_0} && {b_2}
        \arrow["{f_{01}}", from=2-1, to=1-2]
        \arrow["{f_{12}}", from=1-2, to=2-3]
        \arrow["{f_{02}}"', from=2-1, to=2-3]
        \arrow["{g_{01}}", from=2-4, to=1-5]
        \arrow["{g_{02}}"', from=2-4, to=2-6]
        \arrow["{g_{12}}", from=1-5, to=2-6]
      \end{tikzcd}\right)\]
  where either $f_{12}$ or $g_{01}$ is an equivalence. The \emph{Gray
    tensor product} $\bAA \otimes \bBB$, is defined as $|(\bAA \times \bBB,S)|_{\mathfrak{s}}$, see \cref{rem:2morphlocalization}.
\end{definition}

\begin{definition}\label{def:marked2category}
	A marked 2-category $(\bII,E)$ is a pair consisting in a 2-category $\bII$ together with a collection of 1-morphisms $E \subset \Fun([1],\bII)^{\simeq}$ containing every equivalence. We define a functor $f \colon (\bII,E) \to (\bSS,E')$ of marked 2-categories to be an ordinary functor $\bII \to \bSS$ such that $f(E) \subseteq E'$. We denote by $\CCatm$ the 2-category of marked 2-categories and by $\mFun(-,-)$ the corresponding mapping category functor.
\end{definition}

\begin{definition}\label{def:flatsharpmarking}
	Given a 2-category $\bII$ there are two canonical ways of regarding $\bII$ as a marked 2-category: We can mark \emph{every} 1-morphism obtaining a marked 2-category which we dentoe $(\bII,\sharp)$. Dually, we can only mark the equivalences in which case we will use the notation $(\bII,\natural)$.
\end{definition}

\begin{remark}
	The assignment $\bII \mapsto (\bII,\natural)$ induces a fully faithful functor $\CCat \to \CCatm$.
\end{remark}

\begin{definition}\label{def:markedgray}
	Let $(\bAA,E)$ and $(\bBB,F)$ be marked 2-categories. We define the marked Gray tensor product $\bAA \otimes_{E,F} \bBB$ as the following pushout 
	\[
    \begin{tikzcd}
      \coprod_{S} [1] \otimes [1] \ar[r] \ar[d] & \coprod_{S} [1]
      \times [1] \ar[d] \\
      \bAA \otimes \bBB \ar[r] & \bAA \otimes_{E,F} \bBB,
    \end{tikzcd}
  \]
  where $S$ is the set of equivalence classes of pairs of maps
  $f \colon [1] \to \bAA$, $g \colon [1] \to \bBB$ such that either $f$
  lies in $E$ or $g$ lies in $F$, and neither $f$ nor $g$ is
  invertible.
\end{definition}

\begin{proposition}\label{proposition:markedgraycolim}
  Let $(\bAA, E)$ be a marked 2-category. Then the functors
  \[ - \otimes_{\natural,E} \bAA,\, \bAA \otimes_{E,\natural}
    - \colon \CCat^{\leq 1} \to \CCat^{\leq 1} \]
  preserve colimits.
\end{proposition}
\begin{proof}
  This follows from Corollary 4.1.10 in \cite{GagnaHarpazLanariLaxLim}.
\end{proof}

\begin{construction}\label{const:laxfunctorcat}
  Since $(\CCat)^{\leq 1}$ is a presentable 1-category, it follows from
  \cref{proposition:markedgraycolim} that for a marked 2-category $(\bAA,E)$,
  the functors $- \otimes_{\natural,E} \bAA$ and
  $\bAA \otimes_{E,\natural} -$ have right adjoints (see \cref{lem:adjunctions1vs2dimensional} below). We denote
  these by $\FUN(\bAA,-)^{\elax}$ and
  $\FUN(\bAA,-)^{\eoplax}$, respectively, so that we have natural
  equivalences on cores,
  \[ \Fun(\bXX \otimes_{\natural,E} \bAA, \bBB)^{\core} \simeq \Fun(\bXX,
    \FUN(\bAA,\bBB)^{\elax})^{\core},\]
   \[ \Fun(\bAA \otimes_{E,\natural} \bXX, \bBB)^{\core} \simeq \Fun(\bXX,
    \FUN(\bAA,\bBB)^{\eoplax})^{\core}.
  \]
  We write
  $\Nat^{\elaxoplax}_{\bAA,\bBB}(F,G)$ for the mapping categories in
  $\FUN(\bAA,\bBB)^{\elaxoplax}$ between functors $F,G$.
\end{construction}

\begin{construction}\label{const:functoriallaxfunct}
	There exist functors
	\[
		\left(\CCatm^{\leq 1}\right)^{\op} \times \CCat^{\leq 1} \to \CCat^{\leq 1}, \quad ((\bII,E),\bDD) \mapsto \FUN(\bII,\bDD)^{\elaxoplax},
	\]
	whose existence follows formally from \cref{proposition:markedgraycolim}.
\end{construction}

\begin{definition}\label{def:elaxcones}
	Let $(\bII,E)$ be a marked 2-category. We define $\bII^{\colimcone}_{\elax}$ and $\bII^{\limcone}_{\elax}$ as the pushouts
	\[
		\begin{tikzcd}
		   \bII \arrow[d,swap,"\{1\} \times \bII"] \arrow[r] & {[0]} \arrow[d,"\ast"] \\
			{[1]} \otimes_{\natural,E}\bII  \arrow[r] & \bII^{\colimcone}_{\elax}
		\end{tikzcd} \quad \quad 
		\begin{tikzcd}
		   \bII \arrow[d,swap,"\{0\} \times \bII"] \arrow[r] & {[0]} \arrow[d,"\ast"] \\
			{[1]} \otimes_{\natural,E}\bII  \arrow[r] & \bII^{\limcone}_{\elax}.
		\end{tikzcd}
	\]
	We define similarly $\bII^{\colimcone}_{\eoplax}$ and $\bII^{\limcone}_{\eoplax}$ by reversing the order of the Gray tensor product in the relevant pushouts. We will refer to the point selected by the right-most vertical map in the pushouts above as the \emph{cone point}.
\end{definition}

\begin{remark}\label{rem:conesimplified}
	To ease the notation we will simply denote $\bII^{\colimcone}_{\elaxoplax}$ by $\bII^{\colimcone}_{\laxoplax}$ whenever the choice of marking is clear from the context.
\end{remark}

\begin{remark}\label{rem:inclusionconeff}
	The canonical functor $\bII \to \bII^{\colimcone}_{\elaxoplax}$ induced by $\{0\} \to [1]$ (resp.\ $\bII^{\limcone}_{\elaxoplax}$ and $\{1\} \to [1]$) is fully faithful. To see this, we can implement this map using scaled simplicial sets (\cite{LurieGoodwillie}) and verify that the functor above is fully faithful after applying the rigidification functor $\mathfrak{C}^{\mathbf{sc}}$, cf. \cite[Definition 3.1.10, Theorem 4.2.2.]{LurieGoodwillie}. Unraveling the definitions, one sees that the corresponding map on mapping marked simplicial sets is an isomorphism and so the result holds.
\end{remark}

\subsection{Fibrations of 2-categories and the Grothendieck construction}
The goal of this section is to collect the main definitions regarding (cartesian) fibrations of 2-categories which are an efficient method of enconding functors with values in $\CCat$ as we will see in \cref{thm:str}.

\begin{definition}\label{definition:cartenriched}
  A functor of 2-categories $\pi \colon \bEE \to \bBB$ is
  \emph{cartesian-enriched} if
  \begin{enumerate}
  \item For all $x,y \in \bEE$, the functor $\bEE(x,y) \to \bBB(\pi x, \pi y)$ is a cartesian fibration of 1-categories.
  \item For all $x,y,z \in \bEE$, the commutative square
    \[
      \begin{tikzcd}
        \bEE(x,y) \times \bEE(y,z) \arrow{r} \arrow{d} & \bEE(x,z) \arrow{d} \\
        \bBB(\pi x, \pi y) \times \bBB(\pi y, \pi z) \arrow{r} & \bBB(\pi x, \pi z),
      \end{tikzcd}
    \]
    where the horizontal morphisms are given by composition, is a
    morphism of cartesian fibrations (i.e the top horizontal functor
    preserves cartesian morphisms).
  \end{enumerate}
  If $\pi \colon \bEE \to \bBB$ and $\pi' \colon \bEE' \to \bBB'$ are cartesian-enriched, we say that a commutative square of 2-categories
  \[
    \begin{tikzcd}
      \bEE \arrow{r}{\psi} \arrow{d}[swap]{\pi} & \bEE' \arrow{d}{\pi'} \\
      \bBB \arrow{r}{\phi} & \bBB'
    \end{tikzcd}
  \]
  is a \emph{morphism of cartesian-enriched functors} if for all $x, y \in \bEE$, the commutative square
  \[
    \begin{tikzcd}
      \bEE(x,y) \arrow{r} \arrow{d} & \bEE'(\psi x, \psi y) \arrow{d} \\
      \bBB(\pi x, \pi y) \arrow{r} & \bBB'(\pi' \psi x, \pi' \psi y)
    \end{tikzcd}
  \]
  is a morphism of cartesian fibrations.
  Dually, if $\pi^{\co}$ is cartesian-enriched, we say that $\pi$ is \emph{cocartesian-enriched}, and define morphisms of cocartesian-enriched functors similarly. 
\end{definition}

\begin{definition}
  We say that $\pi \colon \bEE \to \bBB$ is \emph{right-enriched} (or \emph{left-enriched}) if \[\bEE(x,y) \to \bBB(\pi x, \pi y)\] is a right (or left) fibration of 1-categories for all $x, y \in \bEE$.
\end{definition}

\begin{definition}
  Let $\pi \colon \bEE \to \bBB$ be a functor of 2-categories. If
  $\overline{f} \colon x \to y$ is a morphism in $\bEE$ lying over
  $f \colon a \to b$ in $\bBB$, then we say that $\overline{f}$ is a
  \emph{$p$-cocartesian} morphism if for every
  object $z \in \bEE$ over $c \in \bBB$, the commutative square of
  1-categories
  \[
    \begin{tikzcd}
      \bEE(y,z) \arrow{r}{\overline{f}^{*}} \arrow{d} & \bEE(x,z) \arrow{d} \\
      \bBB(b,c) \arrow{r}{f^{*}} & \bBB(a,c)
    \end{tikzcd}
  \]
  is a pullback. Dually, we say that $\overline{f}$ is a \emph{$p$-cartesian} morphism if all the squares
  \[
    \begin{tikzcd}
      \bEE(z,x) \arrow{r}{\overline{f}_{*}} \arrow{d} & \bEE(z,y) \arrow{d} \\
      \bBB(c,a) \arrow{r}{f_{*}} & \bBB(c,b)
    \end{tikzcd}
  \]
  are pullbacks. We say that $\bEE$ has all \emph{$p$-cocartesian lifts}
  of some class $T$ of morphisms in $\bBB$ if given a morphism
  $f \colon a \to b$ in $T$ and an object $x$ in $\bEE$ over $a$, there
  exists a $p$-cocartesian morphism $\overline{f} \colon x \to y$
  lying over $f$; similarly, we say that $\bEE$ has all
  \emph{$p$-cartesian lifts} of $T$ if the dual condition holds.
\end{definition}

\begin{notation}
  It will sometimes be notationally convenient to use the terms
  $0$- and $1$-cartesian fibrations for cocartesian and
  cartesian fibrations of 1-categories, respectively. Similarly, we will sometimes refer
  to $0$- and $1$-cartesian morphisms in both 1-categories and
  2-categories.
\end{notation}

\begin{definition}
  A functor $\pi \colon \bEE \to \bBB$ of 2-categories is a \emph{$(0,1)$-fibration} if:
  \begin{enumerate}
  \item $\bEE$ has $p$-cocartesian lifts of all morphisms in $\bBB$.
  \item $p$ is cartesian-enriched.
  \end{enumerate}
  If $\pi \colon \bEE \to \bBB$ and $\pi' \colon \bEE' \to \bBB'$ are $(0,1)$-fibrations, we say that a commutative square of 2-categories
  \[
    \begin{tikzcd}
      \bEE \arrow{r}{\psi} \arrow{d}[swap]{\pi} & \bEE' \arrow{d}{\pi'} \\
      \bBB \arrow{r}{\phi} & \bBB'
    \end{tikzcd}
  \]
  is a \emph{morphism of $(0,1)$-fibrations} if
  \begin{enumerate}
  \item the square is a morphism of
    cartesian-enriched functors,
  \item $\psi$ takes $\pi$-cocartesian morphisms to $\pi'$-cocartesian
    ones.
  \end{enumerate}
  Similarly, we have the notions
  of $(i,j)$-fibrations and their morphisms for all choices of
  $i,j \in \{0,1\}$, where $p \colon \bEE \to \bBB$ is an $(i,j)$-fibration if it is
  $j$-cartesian-enriched and $\bEE$ has all $p$-$i$-cartesian lifts
  over $\bBB$. We write $\FIB^{(i,j)}(\bBB)$ for the locally full sub-2-category (cf. \cref{def:subcategories}) of
  $\CCat_{/\bBB}$ containing the $(i,j)$-fibrations, the morphisms of
  $(i,j)$-fibrations, and all 2-morphisms among these. 
\end{definition}

\begin{remark}
	For an $(i,j)$-fibration $\pi: \bEE \to \bBB $, it follows that $\pi^{\op}$ is a $(1-i,j)$-fibration. Similarly, $\pi^{\co}$ is an $(i,1-j)$-fibration.
\end{remark}

\begin{observation}\label{obs:base1cat}
  Suppose that $\BB$ is a 1-category.  Then a functor of 2-categories
  $p \colon \bEE \to \BB$ is automatically cartesian-enriched, so $p$ is a $(0,1)$-fibration if and only if $\bEE$
  has $p$-cocartesian lifts of morphisms in $\BB$, in which case
  $p$ is also a $(0,0)$-fibration. To emphasize this, we will just
  call such functors $0$-fibrations (and, in the dual case,
  $1$-fibrations).
\end{observation}

\begin{theorem}[Nuiten~\cite{Nuiten}, Abell\'an--Stern~\cite{ASI23}]\label{thm:str}
  There are equivalences of 2-categories
  \begin{align*}
  	\FIB^{(0,1)}(\bBB) \xrightarrow{\simeq} \FUN(\bBB, \CCat), & & \FIB^{(1,0)}(\bBB) \xrightarrow{\simeq} \FUN(\bBB^\op, \CCat),
  \end{align*}
  \begin{align*}
  	\FIB^{(0,0)}(\bBB) \xrightarrow{\simeq} \FUN(\bBB^{\co}, \CCat), & & \FIB^{(1,1)}(\bBB) \xrightarrow{\simeq} \FUN(\bBB^\coop, \CCat),
  \end{align*}
  which is contravariantly natural in $\bBB$ with respect to pullback
  on the left and composition on the right.
\end{theorem}

\begin{remark}
	The theorem above can be sharpened to obtain a Grothendieck construction which takes into account partially (op)lax natural transformations between functors. We refer the reader to \cite{AGH24} for the details.
\end{remark}

\begin{proposition}\label{prop:conservative}
  Let $\bCC$ be a 2-category. Then for every 2-category $\bDD$ the functors,
  \[
    \iota_{c}^* \colon \FUN(\bCC,\bDD)^{\laxoplax} \to \FUN([0],\bDD)^{\laxoplax},
  \]
  given by restricting along the maps $\iota_c \colon [0] \to \bCC$ (where $\iota_c$ is indexed over the set of maps $[0]\to \bCC$), are jointly conservative on 2-morphisms (cf. \cref{def:conserfunctor}).
\end{proposition}
\begin{proof}
  If $\bDD=\Cat$ the claim follows easily from \cite[Theorem 3.5.4]{AGH24} since invertible 2-morphisms in the fibrational picture are clearly detected pointwise. More generally, we invoke \cite[Corollary 2.9.4.]{AGH24} to obtain a commutative diagram
  \[
    \begin{tikzcd}
      \FUN(\bCC,\bDD)^{\laxoplax} \arrow[d,"\iota_c"] \arrow[r] & \FUN(\bCC \times \bDD^{\op},\Cat)^{\laxoplax} \arrow[d,"\iota_c \times \id"] \\
       \FUN([0],\bDD)^{\laxoplax} \arrow[r] & \FUN([0] \times \bDD^{\op},\Cat)^{\laxoplax}
    \end{tikzcd}
  \]
  where the horizontal functors are conservative on 2-morphisms since they are inclusions of locally full sub-2-categories. The functors $\iota_c \times \id$ are jointly conservative on 2-morphisms by our previous discussion. The result now follows.
\end{proof}

\begin{definition}\label{def:oplaxarrow}
	Let $\bCC$ be a 2-category. The oplax arrow 2-category of $\bCC$ is defined as $\ARplax_{\bCC}=\FUN([1],\bCC)^{\oplax}$, dually, the lax arrow 2-category of $\bCC$ is defined as $\ARlax_{\bCC}=\FUN([1],\bCC)^{\lax}$. It then follows that (see \cite[Theorem 3.0.7]{GHLFib})
	\begin{itemize}
		\item restriction along $\{i\} \to [1]$ induces an $(\overline{i},i)$-fibration $\ev_{i}\colon \ARplax_{\bCC} \to \bCC$ where $\overline{i}=1-i$;
		\item restriction along $\{i\} \to [1]$ induces an $(\overline{i},\overline{i})$-fibration $\ev_{i}\colon \ARlax_{\bCC} \to \bCC$.
	\end{itemize}
\end{definition}

\begin{definition}\label{def:strongarrow}
	Let $\bCC$, we define $\AR_{\bCC}=\FUN([1],\bCC)$ and note that $\AR_{\bCC} \to \AR_{\bCC}^{\laxoplax}$ is locally fully faithful. 
\end{definition}

\begin{definition}\label{def:laxslices}
	Let $\bCC$ be a 2-category and let $c \in \bCC$. We define lax versions of the slice construction by means of the following pullback squares
  \[
  	\begin{tikzcd}
  		\bCC_{c \upslash} \arrow[r] \arrow[d] & \ARplax_{\bCC} \arrow[d,"\ev_0"] \\
  		{[0]} \arrow[r,"c"] & \bCC{,}
  	\end{tikzcd} \quad \quad \quad \begin{tikzcd}
  		\bCC_{c \downslash} \arrow[r] \arrow[d] & \ARlax_{\bCC} \arrow[d,"\ev_0"] \\
  		{[0]} \arrow[r,"c"] & \bCC{,}
  	\end{tikzcd}
  \]
  \[
  	\begin{tikzcd}
  		\bCC_{ \upslash c} \arrow[r] \arrow[d] & \ARplax_{\bCC} \arrow[d,"\ev_1"] \\
  		{[0]} \arrow[r,"c"] & \bCC{,}
  	\end{tikzcd} \quad \quad \quad \begin{tikzcd}
  		\bCC_{ \downslash c} \arrow[r] \arrow[d] & \ARlax_{\bCC} \arrow[d,"\ev_1"] \\
  		{[0]} \arrow[r,"c"] & \bCC{.}
  	\end{tikzcd}
  \]
  The lax slices come equipped with:
  \begin{itemize}
   	\item a functor $\bCC_{c\upslash } \to \bCC$ induced by $\ev_1$ which is a $(0,1)$-fibration;
   	\item a functor $\bCC_{c\downslash } \to \bCC$ induced by $\ev_1$ which is a $(0,0)$-fibration;
   	\item a functor $\bCC_{\upslash c } \to \bCC$ induced by $\ev_0$ which is a $(1,0)$-fibration;
   	\item a functor $\bCC_{\downslash c} \to \bCC$ induced by $\ev_0$ which is a $(1,1)$-fibration.
   \end{itemize} 
\end{definition}

\begin{remark}
	It follows from \cref{thm:str} and \cite[Theorem 3.17]{AS23} that:
	\begin{itemize}
	 	\item the $(0,1)$-fibration $\bCC_{c\upslash } \to \bCC$ is classified by the functor $\bCC(c,-) \colon \bCC \to \Cat$;
	 	\item the $(0,0)$-fibration $\bCC_{c\downslash } \to \bCC$ is classified by the functor $\bCC(c,-)^{\op} \colon \bCC^{\co} \to \Cat$; 
	 	\item the $(1,0)$-fibration  $\bCC_{\upslash c } \to \bCC$ is classified by the functor $\bCC(-,c) \colon \bCC^{\op} \to \Cat$;
	 	\item the $(1,1)$-fibration $\bCC_{\downslash c} \to \bCC$ is classified by the functor $\bCC(-,c)^{\op} \colon \bCC^{\coop} \to \Cat$.
	 \end{itemize} 
\end{remark}

\begin{remark}
	The notation chosen for the lax slices is designed to give a description of the 1-morphisms. More precisely, an object in $\bCC_{\upslash c}$ is given by a map $u \colon  x \to c$ and a morphism from $u$ to $v \colon y \to c$ can be represented by triangle
	\[\begin{tikzcd}[column sep=small]
	x && y \\
	& c
	\arrow[from=1-1, to=1-3]
	\arrow[""{name=0, anchor=center, inner sep=0}, "u"', from=1-1, to=2-2]
	\arrow["v", from=1-3, to=2-2]
	\arrow[shorten <=8pt, shorten >=16pt, Rightarrow, from=0, to=1-3].
\end{tikzcd}\]
which commutes up to (non-invertible) 2-morphism. Similarly, a 1-morphism in $\bCC_{\downslash c}$ is given by a laxly commuting triangle where the associated 2-cell points in the other direction.
\end{remark}

\begin{definition}
	We define the strong slice $\bCC_{/c}$ as the locally full sub-2-category of $\bCC_{\upslash c}$ on those triangles above which commute and similarly for the remaining lax slices.
\end{definition}

\begin{proposition}\label{prop:carttransun}
	Let us consider a commutative diagram of $(i,j)$-fibrations
	\[
		\begin{tikzcd}[column sep=small]
			\bXX \arrow[rr,"\pi"] \arrow[dr,"p"'] & & \bYY \arrow[dl,"q"] \\
			& \bSS &
		\end{tikzcd}
	\]
	such that :
	\begin{itemize}
		\item[i)] the functor $\pi$ is a morphism of $(i,j)$-fibrations;
		\item[ii)] for every $s \in \bSS$ the induce functor on fibres $\bXX_s \to \bYY_s$ is a $(i,j)$-fibration;
		\item[iii)] for every morphism $s_{0} \to s_{1}$ in $\bSS$ then the associated diagram
		 \[
		 	 \begin{tikzcd}
		 	 	\bXX_{s_{i}} \arrow[r] \arrow[d] & \bXX_{s_{1-i}} \arrow[d] \\
		 	 	\bYY_{s_{i}} \arrow[r] & \bYY_{t_{1-i}}
		 	 \end{tikzcd}
		 \]
		 is a morphism of $(i,j)$-fibrations.
	\end{itemize}
 Then it follows that $\pi$ is a $(i,j)$-fibration.
\end{proposition}
\begin{proof}
	We deal with the case $(i,j)=(0,1)$ without loss of generality. The fact that $f$ is cartesian-enrinched follows by adapting the proof of \cite[Lemma 3.4.7]{AGH24}. To finish the proof, we show that $\pi$ has enough $0$-$\pi$-cartesian 1-morphisms.

	 Let $\hat{x} \in \bXX$ and consider a morphism $f \colon \pi(\hat{x})=x \to y.$ It is easy to see that we can produce a morphism $\hat{f}: \hat{x} \to \hat{y}$ over $f$ such that
	 \[
	 	  \begin{tikzcd}
	 	   &     t  \arrow[dr,"v"] & \\
	 	  	\hat{x} \arrow[ur,"u"] \arrow[rr,swap,"\hat{f}"] & & \hat{y}
	 	  \end{tikzcd}
	 \]
   where $u$ is $0$-$p$-cartesian and $v$ is a $0$-cartesian edge for the functor $\bXX_{qy} \to \bYY_{qy}$. As a consequence of iii), it follows that the class of morphisms which can be expressed as a composition of a $0$-$p$-cartesian edge and a fibrewise $0$-cartesian edge are stable under composition. Therefore, it will be enough to show that $\hat{f}$ is locally $0$-$p$-cartesian (see 3.3 in \cite{AGH24}). For $\ell \in \bXX_{y}$, we wish to show that we have a pullback diagram
   \[
   	  \begin{tikzcd}
   	  	\bXX_{y}(\hat{y},\ell) \arrow[r,"(-)\circ \hat{f}"]  \arrow[d] & \bXX(\hat{x},\ell) \arrow[d] \\
   	  	{[0]} \arrow[r,"f"]  & \bYY(x,y).
   	  \end{tikzcd}
   \]
   To verify this claim we consider the commutative diagram
   \[\begin{tikzcd}
	{\bXX_{y}(\hat{y},\ell)} & {\bXX_{qy}(t,\ell)} & {\bXX(\hat{x},\ell)} \\
	{[0]} & {\bYY_{qy}(\pi(t),y)} & {\bYY(x,y)} \\
	& {[0]} & {\bSS(qx,qy)}
	\arrow["{(-)\circ v}", from=1-1, to=1-2]
	\arrow[from=1-1, to=2-1]
	\arrow["{(-)\circ  u}", from=1-2, to=1-3]
	\arrow[from=1-2, to=2-2]
	\arrow[from=1-3, to=2-3]
	\arrow["{\pi(v)}"', from=2-1, to=2-2]
	\arrow["{(-)\circ \pi(u)}"', from=2-2, to=2-3]
	\arrow[from=2-2, to=3-2]
	\arrow[from=2-3, to=3-3]
	\arrow["{q(f)}"', from=3-2, to=3-3]
\end{tikzcd}\]
and observe that the two squares on the right are pullbacks since $p$ and $q$ are $(0,1)$-fibrations and the edge $u$ (and its image under $\pi$) is $0$-cartesian. The left square is a pullback since $\bXX_{qy} \to \bYY_{qy}$ is also a $(0,1)$-fibration and $v$ is a $0$-cartesian edge. This concludes the proof.
\end{proof}

\subsection{Adjunctions in a 2-category}
Recall that an adjunction of $2$-categories is a functor from the free  walking adjunction $\Adj \to \CCat$, which specifies a pair of functors
\[
	L \colon \bCC \llra \bDD \colon R
\]
and natural transformations $LR \to \id$ and $\id \to RL$ satisfying the usual triangular identities. We refer the reader to \cite{AGH24} for the relevant proofs and a more extensive treatment.

\begin{proposition}
	A functor of 2-categories $L \colon \bCC \to \bDD$ is a left adjoint if and only if for every $d \in \bDD$ the functor
	\[
		\bCC^{\op} \to \Cat,\quad x \mapsto \bDD(L(x),d)
	\]
	is representable. Dually, a functor of 2-categories $R \colon \bDD \to \bCC$ is a right adjoint if and only if for every $c \in \bCC$ the functor
	\[
		\bDD \to \Cat,\quad x \mapsto \bCC(c,R(x))
	\]
	is correpresentable.\qed
\end{proposition}

\begin{remark}
	Let $L \colon \bCC \to \bDD$ be a functor with associated $0$-fibration (see \cref{obs:base1cat}) $\pi \colon \mathbb{F} \to [1]$. Then it follows $L$ is a left adjoint if and only if $\pi$ is a 1-fibration as well.
\end{remark}

\begin{definition}
	Let $c \in \bCC$ and let
	\[
		\begin{tikzcd}
			x \arrow[rr,"u"] \arrow[dr,"p",swap] & & y \arrow[dl,"q"] \\
			& c & 
 		\end{tikzcd}
	\]
  be a commmutative triangle. We say that $u \colon x \to y$ is a relative left (resp.\ right) adjoint if it defines a left adjoint (resp.\ right adjoint) in $\bCC_{/c}$. 
\end{definition}

\begin{definition}\label{def:reflections}
	A map $x\to y$ in $\bCC$ is said to be a \emph{left inclusion} (resp.\ \emph{right inclusion}) if it is a fully faithful right (resp.\ left) adjoint. In this case, its left (resp.\ right) adjoint is referred to as a \emph{left reflection} (resp.\ \emph{right reflection}).
\end{definition}



%% file: laxcocompletion.tex
\section{(Op)lax limits and colimits}\label{sec:colimits}
In the section, we will gather the necessary results in the theory of partially (op)lax limits which will be needed in the coming sections. We begin in \cref{subsec:LimitsDefinitions} with the definitions and a few basic results. In \cref{subsec:decompositionColimits}, we discuss how partially (op)lax colimits can be decomposed into simpler shapes. In \cref{subsec:filteredColimits}, we collect some facts about $\kappa$-filtered colimits in 2-categories. Lastly, we discuss oriented pullbacks in \cref{subsec:oriented}.

\subsection{Definitions and basic results}
\label{subsec:LimitsDefinitions}

\begin{definition}\label{def:constant}
	We denote by $\underline{(-)} \colon \bCC \to \Fun(\bII,\bCC)$ the diagonal functor, i.e.\ the one that transposes to the projection $\bII \times \bCC \to \bCC $ onto the second factor. We call this functor the \emph{constant diagram} functor.
\end{definition}

\begin{definition}\label{def:elaxlimit}
	Let $F: \bII \to \bCC$ be a functor and suppose that we are given a marking $E$ on $\bII$. We say that $c \in \bCC$ is:
	\begin{itemize}
		\item the $\elaxoplax$ limit colimit of $F$, denoted by $\lim^{\elaxoplax}_\bCC F$, if there exists an equivalence of functors
		\[
		\bCC(c,\lim^{\elaxoplax}_\bCC F) \simeq \Nat_{\bII,\bCC}^{\elaxoplax}(\underline{c},F)
		\]
		which is natural in $c$;
		\item the $\elaxoplax$ colimit of $F$, denoted by $\colim^{\elaxoplax}_\bCC F$, if there exists an equivalence of functors
		\[
		\bCC(\colim^{\elaxoplax}_{\bCC} F, c)\simeq \Nat_{\bII,\bCC}^{\elaxoplax}(F,\underline{c})
		\]
		which is natural on $c$.
	\end{itemize}
\end{definition}

\begin{remark}
	We will refer to $\elaxoplax$ natural transformations $F \to \underline{b}$ as $\elaxoplax$ cones (with tip $b$) (resp $\underline{b} \to F$). Note that producing such a cone is equivalent to giving a functor $T \colon \bII_{\elaxoplax}^{\colimcone} \to \bCC$ (resp. $\bII_{\elaxoplax}^{\limcone} $) such that $T(\ast)=b$ (see \cref{def:elaxcones}).
\end{remark}

\begin{definition}\label{def:sharpmarkingcolim}
	We say that a partially (op)lax (co)limit is a \emph{strong} (co)limit if it is indexed by a marked 2-category of the form $(\bII,\sharp)$ (cf.\  \cref{def:flatsharpmarking}). In this case we will use the notation $\lim_{\bII}F$ (resp. $\colim_{\bII}F$) and  omit the superscript $\sharp$-$\laxoplax$.
	
	Similarly if our marked category is of the form $(\bII,\natural)$ we will use the notation $\lim_{\bII}^{\laxoplax}F$ (resp $\colim_{\bII}^{\laxoplax}F$).
\end{definition}

\begin{remark}\label{rem:limfunctor}
	Suppose that $\bCC$ admits all partially (op)lax (co)limits indexed by a marked 2-category $(\bII,E)$. Then it follows that we have adjunctions
	\[
	\underline{(-)} \colon \bCC \llra \FUN(\bII,\bCC)^{\elaxoplax} \colon \lim^{\elaxoplax}_\bII, \quad \quad \colim^{\elaxoplax}_\bII \colon \FUN(\bII,\bCC)^{\elaxoplax} \llra \bCC \colon \underline{(-)},
	\]
	given by taking $\elaxoplax$ (co)limits.
\end{remark}

\begin{proposition}\label{prop:adjpreservelimtis}
	Let $L \colon \bCC \llra \bDD \colon R$ be an adjunction. Then the left adjoint $L$ preserves all partially (op)lax colimits which happen to exists in $\bCC$. Dually, the right adjoint $R$ preserves all partially (op)lax limits which happen to exist in $\bDD$.
\end{proposition}
\begin{proof}
	We give the proof for the left adjoint and leave the remaining case to the reader since it is formally dual. 
	
	To show our claim we observe that postcomposition with the adjunction $L \dashv R$ yields an adjunction for every marked 2-category $(\bII,E)$
	\[
	L_* \colon \FUN(\bII, \bCC)^{\elaxoplax} \llra \FUN(\bII, \bDD)^{\elaxoplax} \colon R_*.
	\]
	Given $F \colon \bII \to \bCC$ admiting an $\elaxoplax$ colimit we can produce the following chain of natural equivalences
	\[
	\Nat^{\elaxoplax}_{\bII,\bDD}(L_*F,\underline{d}) \simeq \Nat^{\elaxoplax}_{\bII,\bCC}(F,\underline{Rd})\simeq \bCC\left(\colim^{\elaxoplax}_{\bCC}F,Rd\right) \simeq \bDD\left(L\left(\colim^{\elaxoplax}_{\bCC}F\right),d\right)
	\]
	which concludes the proof.
\end{proof}

\begin{proposition}\label{prop:colimfunctorcat}
	Let $\bCC$ be a 2-category admitting partially (op)lax limits of shape $(\bII,E)$ (resp. colimits). Then for every 2-category $\bSS$ the functor 2-category $\FUN(\bSS,\bCC)$ admits partially (op)lax limits of shape $(\bII,E)$ (resp. colimits). 
\end{proposition}
\begin{proof}
	We will only deal with limits since the case of colimits is formally dual. Observe that our assumptions guarantee the existence of an adjunction
	\[
	L_* \colon \FUN(\bSS,\bCC) \llra \Fun(\bSS,\FUN(\bII,\bCC)^{\elaxoplax}) \colon R_*
	\]
	where the left adjoint is given by post-composition with the constant diagram functor and the right adjoint is given by precomposition along the $\elaxoplax$-limit functor. By \cite[Proposition 2.9.1]{AGH24} there exists a fully faithful functor
	\[
	j \colon \FUN(\bII,\FUN(\bSS,\bCC))^{\elaxoplax} \to \Fun(\bSS,\FUN(\bII,\bCC)^{\elaxoplax}) 
	\]
	whose essential image consists in those functors in the target which factor as
	\[
	\bSS \to \FUN(\bII,\bCC) \to \FUN(\bII,\bCC)^{\elaxoplax}.
	\]
	We claim that $R_* \circ j$ is the desired right adjoint to our constant diagram functor. To see this we note that $L_*$ factors through the essential image of $j$ and that it can be identified with the constant diagram functor. The result now follows easily.
\end{proof}

\begin{remark}\label{rem:cat2complete}
	The 2-category of 2-categories $\CCat$ admits all partially (op)lax (co)limits by Proposition 4.1.5 and Proposition 4.1.8 in \cite{AGH24}. We summarize the construction of partially (op)lax (co)limits below:

	Let $F \colon \bCC \to \CCat$ be a functor where $(\bCC,E)$ is a marked 2-category and let $p \colon \bXX \to \bCC^{\epsilon}$ be the associated $(i,j)$-fibration where $\epsilon \in \{\emptyset,\op,\co,\coop\}$ according to the parameters $i,j$. Then:
	\begin{itemize}
		\item The $\elax$ colimit of $F$ is obtained by taking the associated $(0,j)$-fibration $\bXX \to \bCC^{\epsilon}$ and localising at the collection of $j$-cartesian 2-morphisms as well as at the collection of $0$-cartesian 1-morphisms living over $E$.
		\item The $\eoplax$ colimit of $F$ is obtained by taking the associated $(1,j)$-fibration $\bXX \to \bCC^{\epsilon}$ and localising at the collection of $j$-cartesian 2-morphisms as well  as at the collection of $1$-cartesian 1-morphisms living over $E$.
		\item The $\elax$ limit of $F$ is obtained as the full sub-2-category of the 2-category of sections of the $(0,j)$-fibration $\bXX \to \bCC^{\epsilon}$ given by those maps $s: \bCC^{\epsilon} \to \bXX$ which sends every 2-morphism to a $j$-cartesian 2-morphism and every marked 1-morphism of a $0$-cartesian morphism.
		\item The $\eoplax$ limit of $F$ is obtained as the full sub-2-category of the 2-category of sections of the $(1,j)$-fibration $\bXX \to \bCC^{\epsilon}$ given by those maps $s: \bCC^{\epsilon} \to \bXX$ which sends every 2-morphism to a $j$-cartesian 2-morphism and every marked 1-morphism of a $1$-cartesian morphism.
	\end{itemize}
	We observe that it follows from \cref{def:underlying1cat} and \cref{def:localization} that if $F$ takes values in $\Cat$ then the descriptions of the partially (op)lax (co)limits retrieve the corresponding (co)limit in $\Cat$.
\end{remark}

\begin{proposition}\label{prop:yonedacont}
	Let $\bCC$ be a 2-category. Then the Yoneda embedding $h_{\bCC} \colon \bCC \to \PSh_{\Cat}(\bCC)$ preserves all partially (op)lax limits which exist in $\bCC$.
\end{proposition}
\begin{proof}
	Let $(\bII,E)$ be a marked 2-category and let $D \colon \bII \to \bCC$ be a diagram. Then it follows that we have natural equivalences
	\[
	\Nat_{\bCC^\op,\Cat}\left((h_{\bCC}(-),h_{\bCC}(\lim^{\elaxoplax}_{\bII}D)\right) \simeq \bCC\left(-,\lim_{\bII}^{\elaxoplax}D \right) \simeq \Nat_{\bII,\bCC}^{\elaxoplax}\left(\underline{(-)},D\right)
	\]
	As a consequence of \cite[Corollary 2.8.13]{AGH24} we obtain a natural equivalence
	\[
	\Nat_{\bCC^\op,\Cat}\left((h_{\bCC}(-),h_{\bCC}(\lim^{\elaxoplax}_{\bII}D)\right) \simeq \Nat_{\bII,\bCC}^{\elaxoplax}\left(\underline{(-)},D\right) \simeq \Nat_{\bII,\PSh_{\Cat}(\bCC)}^{\elaxoplax}(\underline{h_{\bCC}},h_{\bCC}D).
	\]
	The previous discussion together with the universal property of the limit implies that we have a canonical morphism $h_{\bCC}(\lim^{\elaxoplax}_{\bII}D) \to \lim^{\elaxoplax}_{\bII}h_{\bCC}D$. Invoking again the Yoneda lemma, we see that our map is point-wise an equivalence, which concludes the proof.
\end{proof}

\begin{proposition}\label{prop:limitscommute}
	Let $\bCC$ be a 2-category which admits partially (op)lax limits (resp. colimits) of shape $(\bII,E)$ and suppose that we are given another marked 2-category $(\bSS,T)$ and a functor $F \colon \bSS \to \FUN(\bII,\bCC)^{\elaxoplax}$ such that the $T\text{-}\laxoplax$ limit of $F$ exists and such that $\bCC$ admits partially (op)lax colimits indexed by $(\bSS,T)$ (resp. colimits). Then we have an equivalence in $\bCC$
	\[
	\lim^{\elaxoplax}_{\bII}\lim^{T\text{-}\laxoplax}_{\bSS} F \simeq \lim^{T\text{-}\laxoplax}_{\bSS} \lim^{\elaxoplax}_{\bII} F,
	\]
	and similarly for colimits.
\end{proposition}
\begin{proof}
	We exclusively prove the case of colimits since the case of limits is formally dual. Let $c \in \bCC$, and consider the following chain of natural equivalences
	\[
	\bCC\left(\colim^{\elaxoplax}_{\bII}\colim^{T\text{-}\laxoplax}_{\bSS} F,c\right) \simeq \Nat^{\elaxoplax}_{\bII,\bCC}\left(\colim^{T\text{-}\laxoplax}_{\bSS} F, \underline{c}\right) \simeq \Nat^{T\text{-}\laxoplax}_{\bSS,\Fun(\bII,\bCC)^{\elaxoplax}}\left(F, \underline{c}\right),
	\]
	\[
	\bCC\left(\colim^{T\text{-}\laxoplax}_{\bSS}\colim^{\elaxoplax}_{\bII} F,c\right) \simeq \Nat_{\bSS,\bCC}^{T\text{-}\laxoplax}\left(\colim^{\elaxoplax}_{\bII}F,\underline{c}\right) \simeq \Nat^{T\text{-}\laxoplax}_{\bSS,\Fun(\bII,\bCC)^{\elaxoplax}}\left(F, \underline{c}\right).
	\]
	The desired equivalence now follows from the Yoneda lemma.
\end{proof}

\begin{definition}\label{def:final}
	Let $f \colon (\bII,E) \to (\bJJ,T)$ be a functor of marked 2-categories and let $\bCC$ be a 2-category.
	\begin{itemize}
		\item  We say that the functor $f$ is \emph{$\laxoplax$ final} if for every 2-category $\bCC$ and every diagram $D \colon \bJJ \to \bCC$, then the $T\text{-}\laxoplax$ colimit of $D$ exists if and only if the $\elaxoplax$ colimit of $D\circ f$ does and the canonical comparison map
		\[
			\colim_{\bII}^{\elaxoplax}(D\circ f) \xrightarrow{\simeq} \colim_{\bJJ}^{T\text{-}\laxoplax}D
		\]
		(whenever defined)
		is an equivalence.
		\item We say that the functor $f$ is \emph{$\laxoplax$ initial} if for every 2-category $\bCC$ and every diagram $D \colon \bJJ \to \bCC$, then the $T\text{-}\laxoplax$ limit of $D$ exists if and only if the $\elaxoplax$ limit of $D\circ f$ does and the canonical comparison map
		\[
			\lim_{\bJJ}^{T\text{-}\laxoplax}D \xrightarrow{\simeq} \lim_{\bII}^{\elaxoplax}(D\circ f)
		\]
		(whenever defined)
		is an equivalence.
	\end{itemize}
\end{definition}

\begin{remark}
	The notion of final functor of 2-categories has been studied by the Stern and the first author in \cite{AS23} where a characterization of final functors is provided in the spirit of Quillen's Theorem A.
\end{remark}

\begin{proposition}\label{prop:locfinal}
	Let $p \colon \bCC \to \bDD$ be a functor of 2-categories which is obtained by localizing at a certain collection of 1 and 2-morphisms. Then $p$ is $\laxoplax$-final and $\laxoplax$-initial with respect to the collection of 1-morphisms inverted by $p$.
\end{proposition}
\begin{proof}
	Let $E$ by the collection of 1-morphisms inverted by $p$. By \cite[Theorem 3.31]{AS23} it will be enough to show that given an $(i,j)$-fibration $\pi \colon\bXX \to \bDD$ whose fibres are 1-categories then restriction along $p$ induces an equivalence of 2-categories
	\[
		\FUN_{/\bDD}(\bDD,\bXX) \xrightarrow{\simeq} \FUN^{\diamond}_{/\bDD}(\bCC,\bXX),
	\]
	where we denote by “$\diamond$“ the full sub-2-category on those morphisms which send the edges in $E$ to $i$-cartesian edges. This follows easily from the universal property of the localization after noting that:
	\begin{itemize}
		\item the fibration $\pi \colon \bXX \to \bDD$ detects invertible 2-morphisms, so that consequently the 2-morphisms in $\bCC$ which are inverted by $p$ are sent to invertible 2-morphisms in $\bXX$;
		\item every 1-morphism inverted by $p$ is sent to an $i$-cartesian edge in $\bXX$ lying over an equivalence and is consequently itself invertible.\qedhere
	\end{itemize}
\end{proof}

\subsection{Decomposition of (op)lax colimits}
\label{subsec:decompositionColimits}
In this section, we discuss how one can decompose (op)lax colimits in a 2-category into simpler shapes.

\begin{definition}\label{def:grayslice}
	For a 2-category $\bCC$, we let $\CCatm^{\laxoplax}_{\upslash \bCC}$ be the 2-category obtained as the unstraightening of the functor 
	\[
	\left(\CCatm^{\leq 1}\right)^{\op} \to \CCat,\quad (\bII,E) \mapsto \FUN(\bII,\bCC)^{\elaxoplax},
	\]
	 arising from \cref{const:functoriallaxfunct}. We observe that the inclusion of the fibre at $[0] \in \CCatm$ yields a functor $\bCC \hra \CCatm^{\laxoplax}_{\upslash \bCC}$.
\end{definition}

\begin{lemma}\label{lem:globallimits}
	Let $\bCC$ be a 2-category which admits all partially (op)lax limits. Then the functor $\bCC \to \CCatm^{\laxoplax}_{\upslash \bCC}$ (cf.\ \cref{def:grayslice}) admits a left adjoint
	\[
	\mathfrak{C}\!\operatorname{olim} \colon  \CCatm^{\laxoplax}_{\upslash \bCC} \to \bCC, \quad \quad F \mapsto \colim^{\elaxoplax}_{\bII}F.
	\]
	which we call the \emph{global colimit functor}. 
\end{lemma}
\begin{proof}
	Let $c \in \bCC$ and let $t_c \colon ([0],\sharp) \to \bCC$ be its image in $\CCatm^{\laxoplax}_{\upslash \bCC}$. Then there exists a natural equivalence
	\[
	\CCatm^{\laxoplax}_{\upslash \bCC} (F,t_c) \simeq \Nat^{\elaxoplax}_{\bII,\bCC}(F,\underline{c}),
	\]
	which shows the claim.
\end{proof}

\begin{remark}\label{rem:notsogloballimits}
	Let us suppose that $\bCC$ admits only a specific class of partially (op)lax colimits which we view as a full sub-1-category $\XX \to (\CCatm)^{\leq 1}$ which contains the terminal 2-category. Then we can adapt the proof of \cref{lem:globallimits} to show that the functor 
	\[
	\bCC \to \CCatm^{\laxoplax}_{\upslash \bCC} \times_{(\CCatm)^{\leq 1}} \XX,
	\]
	admits a left adjoint.
\end{remark}

\begin{proposition}\label{prop:diagramdecomposition}
	Let $(\bII,E)$ be a marked 2-category and consider a diagram $d \colon \bII \to \bCC$. Suppose further that we are given a diagram $\tau \colon \JJ \to \CCatm$ indexed by a 1-category $\JJ$ so that the strong colimit of $\tau$ is $(\bII,E)$ and that $\bCC$ admits strong colimits of shape $\JJ$. Denote $(\bII_j,E_j)= \tau(j)$ and let $d_j$ be the restriction of $d$ along the canonical map $\bII_j \to \bII$. If the partially (op)lax colimits of $d$ and each of the $d_j$ exists in $\bCC$, then we have an equivalence
	\[
	\colim_{\bII}^{\elaxoplax}d \simeq \colim_{\JJ}\colim_{\bII_j}^{E_{j}\text{-}\laxoplax}d_j.
	\]
\end{proposition}
\begin{proof}
	Let $\XX \to \CCatm^{\leq 1}$ be as in \cref{rem:notsogloballimits} so that we have an adjunction 
	\[
	\mathfrak{C}\!\operatorname{olim}_{\XX} \colon \CCatm^{\laxoplax}_{\upslash \bCC} \times_{(\CCatm)^{\leq 1}} \XX \llra \bCC .
	\]
	It is clear that the colimit cone for $\tau$ can be lifted to a diagram $\Psi \colon \JJ^{\colimcone} \to \CCatm^{\laxoplax}_{\upslash \bCC} \times_{(\CCatm)^{\leq 1}} \XX$. Therefore, since $\mathfrak{C}\!\operatorname{olim}_{\XX}$ is a left adjoint it will be enough to show that $\Psi$ is a (strong) colimit cone. Let us point out that since $\bCC$ admits strong limits of shape $\JJ$ it will be enough to show that $\Psi$ is a colimit cone in the 1-category $\left(\CCatm^{\laxoplax}_{\upslash \bCC} \times_{\CCatm^{\leq 1}} \XX\right)^{\leq 1}= \GG$.
	 Let $d \colon \bII \to \bCC$ and $G \colon \bYY \to \bCC$ be objects in $\GG$. Then it follows that we can compute the mapping space between these two objects as the partially oplax limit of the cospan
	\[
	\begin{tikzcd}
		\GG(d,G) \arrow[d] \arrow[r] & \mFun(\bII,\bYY)^{\core} \arrow[d,"G \circ -"] \\
		{[0]} \arrow[r,"d"] & \Fun(\bII,\bCC)^{\elaxoplax}
	\end{tikzcd}
	\]
	where we are marking the bottom horizontal morphism. The compatibility of the cartesian product and the Gray tensor product with colimits (\cref{proposition:markedgraycolim}) implies that the diagram above is equivalent to
	\[
	\begin{tikzcd}
		\GG(d,G) \arrow[d] \arrow[r] & \left(\lim_{\JJ^\op}\mFun(\bII_j,\bYY)\right)^{\core}\arrow[d,"G \circ -"] \\
		{[0]} \arrow[r,"d"] & \lim_{\JJ^\op}\Fun(\bII_j,\bCC)^{E_{j}\text{-}\laxoplax}
	\end{tikzcd}
	\]
	We invoke \cref{prop:limitscommute} to obtain a natural equivalence $\GG(d,G) \simeq \lim_{\JJ^{\op}}\GG(d_i,G)$ which identifies $d$ as the desired colimit. The claim now follows.
\end{proof}

\begin{proposition}\label{prop:minimalshapescolim}
	Let $\bCC$ and $\bDD$ be 2-categories admitting $\kappa$-small partially (op)lax (co)limits and consider a functor $F \colon \bCC \to \bDD$. Then $F$ preserves $\kappa$-small partially (op)lax (co)limits if and only if it preserves $\kappa$-small strong (co)limits indexed by 1-categories and (op)lax (co)limits indexed by $[1]$.
\end{proposition}
\begin{proof}
	We prove it in the case of colimits, so that the case of limits follows by taking opposite categories. The conditions are clearly necessary. To see that they are indeed sufficient, we note that if $F$ preserves (op)lax colimits indexed by $[1]$ it also preserves (op)lax limits indexed by $C_2$ since we can always find a final map $[1] \to C_2$. The result follows from \cref{prop:diagramdecomposition}.
\end{proof}

\subsection{Compact objects and filtered colimits in 2-categories}
\label{subsec:filteredColimits}

	

\begin{definition}\label{def:compactobject}
	An object $c \in \bCC$ is said to be $\kappa$-compact if and only if the functor
	\begin{equation*}
		\bCC(c,-)\colon\bCC \to \Cat
	\end{equation*} 
	preserves strong colimits indexed by a $\kappa$-filtered 1-category.
\end{definition}

\begin{lemma}\label{lem:1dimensionalsuffices}
	Let $c \in \bCC$ and let $[1] \otimes c$ denotes the lax colimit of the functor $[1] \to [0] \xrightarrow{c} \bCC$. Then $c$ is $\kappa$-compact if and only if $c$ and $[1] \otimes c$ are $\kappa$-compact in $\bCC^{\leq 1}$.
\end{lemma}
\begin{proof}
	Note that we only need to show that for every $\kappa$-filtered colimit $\colim_{\II}s_i=s$ in $\bCC$ the following maps of spaces
	\[
	\colim_{\II}\bCC(c, s_i)^{\core} \to \bCC(c,s)^{\core}, \quad \quad  \Cat([1],\colim_{\II}\bCC(c, s_i))^{\core} \to \Fun([1],\bCC(c,s))^{\core}
	\]
	are equivalences. The first map follows directly from the hypothesis. For the second map we use that $[1]$ is $\aleph_0$-compact in $\Cat^{\leq 1}$  to identify the map in question with
	\[
	\colim_{\II}\bCC([1] \otimes c,s_i)^{\core} \to \bCC([1]\otimes c,s)^{\core}
	\]
	which is also an equivalence by assumption.
\end{proof}

\begin{remark}\label{rem:generators}
	As a direct consequence of the previous lemma we find that the image of the functor $\Delta \to \Cat$ factors through the full sub-2-category on $\aleph_0$-compact objects. More generally, since 2-categories can be modelled using $\Theta_2$-spaces we see that the inclusion $\Theta_2 \hra \CCat$ factors through $\aleph_0$-compact objects as well.  
\end{remark}

\begin{proposition}\label{prop:filteredcolimitscommute}
	Let $\II$ be a $\kappa$-filtered 1-category. Then the (strong) colimit functor $\colim_{\II} \colon\FUN(\II,\CCat) \to \CCat$, preserves partially (op)lax limits indexed by $\kappa$-small 2-categories.
\end{proposition}
\begin{proof}
	It is clear that $\colim_{\II}$ preserves strong limits indexed by a $\kappa$-small 1-category. The result now follows from \cref{prop:minimalshapescolim} and \cref{rem:generators}.
\end{proof}

\begin{corollary}\label{cor:smallcolimofcompacts}
	Let $\bCC$ be a 2-category, and let $d \colon \bII \to \bCC$ be a diagram where $(\bII,E)$ is a $\kappa$-small marked 2-category. Suppose that for every $i \in \bII$ the object $d(i)$ is $\kappa$-compact. Then $\colim^{\elaxoplax}_{\bII} d$ is also $\kappa$-compact.
\end{corollary}
\begin{proof}
	Combine \cref{prop:yonedacont}, \cref{prop:filteredcolimitscommute} and \cref{prop:limitscommute}.
\end{proof}

\begin{lemma}\label{lem:compact2cats}
	If $\bII$ is a $\kappa$-small 2-category (cf.\ \cref{def:kappasmall2cat}), then $\bII$ is $\kappa$-compact.
\end{lemma}
\begin{proof}
	The result follows immediately from the definition of $\kappa$-smallness (see also \cref{rem:generators}) together with \cref{cor:smallcolimofcompacts}.
\end{proof}


\begin{proposition}\label{prop:2catsasfilteredcolimits}
	Let $\bII$ be a 2-category. Then $\bII$ can be expressed as a $\kappa$-filtered colimit of $\kappa$-compact 2-categories.
\end{proposition}
\begin{proof}
	Let $j \colon (\kappa\!\CCat)^{\leq 1} \hra (\CCat)^{\leq 1}$ denote the full sub-1-category spanned by the $\kappa$-compact 2-categories, and observe that by \cref{rem:generators} we have that $\iota \colon \Theta_2 \hra (\kappa\!\CCat)^{\leq 1}$. Thus, since $\Theta_2$ is dense in $(\CCat)^{\leq 1}$, we find that $j$ is dense. Consequently, every 2-category $\bII$ is the colimit of the functor
	\[
	(\kappa\!\CCat)^{\leq 1}_{/\bII} \to (\kappa\!\CCat)^{\leq 1} \to (\CCat)^{\leq 1}.
	\]
	Since \cref{cor:smallcolimofcompacts} implies that $(\kappa\!\CCat)^{\leq 1}_{/\bII}$ has $\kappa$-small colimits and is therefore $\kappa$-filtered, the claim follows.
\end{proof}

\subsection{Oriented pullbacks}\label{subsec:oriented}

In this section, we introduce a class of finite partially (op)lax limits, which we call \emph{oriented pullbacks}, serving as suitable lax analogues of ordinary (strong) pullbacks. Oriented pullbacks form the foundation of the theory of fibrations internal to a 2-category (see \cref{sec:internfib}) and as we will later see, play the same role in 2-topos theory as finite limits do in 1-topos theory (see \cref{thm:giraud}). 

\begin{definition}
	Let $\Lambda^2_2$ denote the 1-category $0 \to 2 \leftarrow 1$ (i.e. the walking cospan) and consider the marked category $(\Lambda^2_2,P_0)$ where we mark the edge $0 \to 2$. Given a cospan $F \colon \Lambda^2_2 \to \bCC$, we refer to the object $F(0) \orientedtimesrl_{F(2)}F(1)= \lim^{P_0\text{-}\oplax}_{\Lambda^2_2}F$ as the \emph{oriented pullback} of $F$. Note that the oriented pullback fits into a universal laxly commutative square of the form
	\begin{equation}\label{eq:opullback}
		\begin{tikzcd}
	F(0) \orientedtimesrl_{F(2)}F(1) & {F(1)} \\
	{F(0)} & {F(2)}.
	\arrow[from=2-1, to=2-2]
	\arrow[from=1-2, to=2-2]
	\arrow[from=1-1, to=2-1]
	\arrow[from=1-1, to=1-2]
	\arrow[Rightarrow, shorten <=15pt, shorten >=15pt, from=2-1, to=1-2]
\end{tikzcd}
	\end{equation}
\end{definition}

\begin{remark}
	Let $(\Lambda^2_2,P_1)$ denote the marked category where we now mark the edge $1 \to 2$ in $\Lambda^2_2$. Denoting the $P_1\text{-}\oplax$ limit as $F(0)\orientedtimeslr_{F(2)} F(1)$ we obtain again a canonical laxly commutative square
	\[
		\begin{tikzcd}
	F(0) \orientedtimeslr_{F(2)}F(1) & {F(1)} \\
	{F(0)} & {F(2)}.
	\arrow[from=2-1, to=2-2]
	\arrow[from=1-2, to=2-2]
	\arrow[from=1-1, to=2-1]
	\arrow[from=1-1, to=1-2]
	\arrow[Rightarrow, shorten <=15pt, shorten >=15pt, from=1-2, to=2-1]
\end{tikzcd}
	\]
	In addition, there exists equivalences in $\bCC$
	\[
		F(0) \orientedtimeslr_{F(2)}F(1) \simeq \lim^{P_0\text{-}\lax} F, \enspace \enspace F(0) \orientedtimesrl_{F(2)}F(1) \simeq \lim^{P_1\text{-}\lax} F,
	\]
	which shows that definition of oriented pullbacks is canonical up to reindexing.
\end{remark}

\begin{construction}\label{cons:freefib}
	We introduce an important class of oriented pullbacks which will play a significant in the following sections, namely,
	\[
	\begin{tikzcd}
	\Free_c^{0}(x) & {x} \\
	{c} & {c},
	\arrow[from=2-1, to=2-2]
	\arrow[from=1-2, to=2-2]
	\arrow[from=1-1, to=2-1]
	\arrow[from=1-1, to=1-2]
	\arrow[Rightarrow, shorten <=8pt, shorten >=6pt, from=1-2, to=2-1]
\end{tikzcd}
\quad \quad \quad 
		\begin{tikzcd}
	\Free_{c}^{1}(x)  & {x} \\
	{c} & {c},
	\arrow[from=2-1, to=2-2]
	\arrow[from=1-2, to=2-2]
	\arrow[from=1-1, to=2-1]
	\arrow[from=1-1, to=1-2]
	\arrow[Rightarrow, shorten <=8pt, shorten >=6pt, from=2-1, to=1-2]
\end{tikzcd} 
	\]
	which are obtained by taking oriented pullbacks (note the different orientations) along the identity map. We call $\Free_{c}^{(\epsilon)}(x)$ the \emph{free $\epsilon$-fibration} where $\epsilon \in \{0,1\}$. We refer the reader to \cref{sec:internfib} (and \cref{prop:freeFibrationIsFreeFibration} for a justification of our choice of terminology) where we will develop a theory of internal fibrations in a 2-category. 

	If the map $p \colon x \to c$ is chosen to be the identity then we have equivalences $\Free_{c}^{0}(x) \simeq \Free_{c}^{1}(x)$ and we will use the notation $c^{[1]}$.
\end{construction}

\begin{lemma}\label{lem:characterizationffcotensors}
	Let $\bCC$ be a 2-category admiting (op)lax limits of \emph{constant} diagrams $[1] \to \bCC$ and products. Given a morphism $f \colon c \to d$, then it follows that $f$ is fully faithful (cf.\ \cref{def:ffinternal}) if and only if it induces a pullback square
	\[
		\begin{tikzcd}
			c^{[1]} \arrow[d] \arrow[r] & d^{[1]}\arrow[d] \\
			c \times c \arrow[r] & d \times d.
		\end{tikzcd}
	\]
	where the vertical maps are determined by the canonical inclusions $d_i \colon [0] \to [1]$ for $i \in \{0,1\}$.

\end{lemma}
\begin{proof}
	The claim is clear if $\bCC=\Cat$, so that the general case follows from Yoneda's lemma in light of \cref{prop:yonedacont}.
\end{proof}

\begin{lemma}
	Let $\bCC$ be a 2-category admiting oriented pullbacks. Then $\bCC$ admits (op)lax limits indexed by $[1]$ and we have a commutative diagrams
	\[
		\begin{tikzcd}
			\FUN({[1]},\bCC)^{\laxoplax} \arrow[dr,"\lim^{\laxoplax}_{[1]}",swap] \arrow[r] & \FUN(\Lambda^2_2,\bCC)^{P_{0}\text{-}\laxoplax} \arrow[d,"\lim^{P_{0}\text{-}\laxoplax}_{\Lambda^2_2}"] \\
			& \bCC
		\end{tikzcd}
	\]
	where the top horizontal morphism is obtained by restricting along the map $\varphi \colon (\Lambda^2_2, P) \to ([1],\natural)$ which collapses the edge $0 \to 2$ onto $\{1\}$.
\end{lemma}
\begin{proof}
	The proof is immediate after noting that the map $\varphi$ is initial.
\end{proof}

\begin{proposition}\label{prop:pasting}
	Let $\bCC$ be a 2-category and consider a diagram of the form
	\[\begin{tikzcd}
	x & b \orientedtimesrl_{c}z & z \\
	a & b & c
	\arrow[from=1-1, to=1-2]
	\arrow[from=1-1, to=2-1]
	\arrow[from=1-2, to=1-3]
	\arrow[from=1-2, to=2-2]
	\arrow[from=1-3, to=2-3]
	\arrow[from=2-1, to=2-2]
	\arrow[Rightarrow, shorten <=10pt, shorten >=10pt, from=2-2, to=1-3]
	\arrow[from=2-2, to=2-3]
\end{tikzcd}\]
where the right square is an oriented pullback and the left square commutes. Then the composite square is an oriented pullback if and only if the left square is a strong pullback.
\end{proposition}
\begin{proof}
	It is clear that by construction we always have a map $\Psi \colon x \to a \orientedtimesrl_{c}z$ so our claim reduces to showing $\Psi$ is invertible if and only if $x$ is a (strong) pullback. Invoking \cref{prop:yonedacont} we can reduce to the case of $\bCC=\Cat$. After some elementary manipulations we see that $ a \orientedtimesrl_{c}z$ can be obtained as the strong pullback
	\[
		\begin{tikzcd}
			a \orientedtimesrl_{c}z \arrow[d] \arrow[r] & \Fun({[1]},c) \arrow[d] \\
			 a \times z \arrow[r] & c \times c.
		\end{tikzcd}
	\]
 The claim now follows easily from the usual pasting law for strong pullbacks.
\end{proof}

\begin{corollary}\label{cor:pullbackoffree}
	Every oriented pullback can be obtained as a strong pullback of a free fibration (cf.\ \cref{cons:freefib}).\qed
\end{corollary}

\begin{proposition}\label{prop:mapcatlaxarrow}
	Let $\bCC$ be a 2-category and let $f \colon x \to y$ and $g \colon a \to b $ be 1-morphisms in $\bCC$. Then we have oriented pullback squares
	\[
		\begin{tikzcd}
	\AR_{\bCC}^{\lax}(f,g) & {\bCC(x,a)} \\
	{\bCC(y,b)} & {\bCC(x,b)},
	\arrow["{\bCC(f,b)}",swap,from=2-1, to=2-2]
	\arrow["{\bCC(x,g)}",from=1-2, to=2-2]
	\arrow[from=1-1, to=2-1]
	\arrow[from=1-1, to=1-2]
	\arrow[Rightarrow, shorten <=8pt, shorten >=6pt, from=1-2, to=2-1]
\end{tikzcd} \enspace
	\begin{tikzcd}
	\AR_{\bCC}^{\oplax}(f,g) & {\bCC(x,a)} \\
	{\bCC(y,b)} & {\bCC(x,b)}.
	\arrow["{\bCC(f,b)}",swap,from=2-1, to=2-2]
	\arrow["{\bCC(x,g)}",from=1-2, to=2-2]
	\arrow[from=1-1, to=2-1]
	\arrow[from=1-1, to=1-2]
	\arrow[Rightarrow, shorten <=8pt, shorten >=6pt, from=2-1, to=1-2]
\end{tikzcd}
	\]
\end{proposition}
\begin{proof}
	We deal without loss of generality with the case of the oplax arrow 2-category. We start by observing that by definition we have a strong pullback square
	\[
		\begin{tikzcd}
			\AR_{\bCC}^{\oplax}(f,g) \arrow[r] \arrow[d] & \Fun([1],\Fun([1],\bCC)^\oplax)^\oplax \arrow[d] \\
			{[0]} \arrow[r,"f \times g"] & \Fun([1],\bCC)^\oplax \times \Fun([1],\bCC)^\oplax
		\end{tikzcd}
	\]
	Moreover, it also follows directly from the defintions that we have an equivalence of 2-categories 
	\[
		\FUN([1],\FUN([1],\bCC)^\oplax)^{\oplax} \simeq \FUN([1] \otimes [1],\bCC)^\oplax.
	\] 
	Let $\partial C_2$ be the underlying 1-category of the walking 2-morphism $C_2$ and denote $[2] \coprod_{\{0\}\coprod \{2\}}[2]=P$. Then it follows that we have a pushout square
	\[
		\begin{tikzcd}
			\partial C_2 \arrow[r] \arrow[d] & C_2 \arrow[d] \\
			P \arrow[r] & {[1]} \otimes {[1]}.
		\end{tikzcd}
	\]
	Moreover, since the Gray tensor product is compatible with (strong) colimits we obtain a pullback diagram
	\[
		\begin{tikzcd}
			\FUN([1]\otimes [1],\bCC)^\oplax \arrow[r] \arrow[d] & \FUN(C_2,\bCC)^\oplax \arrow[d] \\
			\FUN(P,\bCC)^\oplax \arrow[r] & \FUN(\partial C_2,\bCC)^\oplax.
		\end{tikzcd}
	\]
	Passing to the corresponding fibres we obtain another pullback diagram
	\[
		\begin{tikzcd}
			\AR_{\bCC}^\oplax(f,g) \arrow[r] \arrow[d] & Q \arrow[d] \\
			\bCC(x,a)\times \bCC(y,b) \arrow[r,swap] & \bCC(x,b) \times \bCC(x,b).
		\end{tikzcd}
	\]
	where $Q= \FUN(C_2,\bCC)^\oplax \times_{\bCC \times \bCC}(\{x\} \times\{b\})$ and the bottom horizontal map is given by $\bCC(x,g) \times \bCC(f,b)$. Note that $Q$ is a 1-category by \cref{prop:conservative}. To finish the proof, we will identify the right-most vertical functor with $\ev_0 \times \ev_1 \colon \bCC(x,b)^{[1]} \to \bCC(x,b) \times \bCC(x,b)$.

	There exists a canonical map $k \colon [1] \otimes [1] \to C_2$ which collapses the edges $\{i\}\otimes [1]$ for $i \in \{0,1\}$. This in turn yields a map 
	\[
		\widetilde{\Psi} \colon \FUN(C_2,\bCC)^{\oplax} \to \FUN([1]\otimes [1],\bCC)^{\oplax} \simeq \FUN([1],\FUN([1],\bCC)^{\oplax})^{\oplax}.
	\]
	Let us consider the following pullback diagram
	\[
		\begin{tikzcd}
			\bCC(x,b)^{[1]} \arrow[r] \arrow[d] & \FUN({[1]},\FUN({[1]},\bCC)^{\oplax})^{\oplax} \arrow[d] \\
			{[0]} \arrow[r] & \FUN({[1]},\bCC)^{\oplax} \times \FUN({[1]},\bCC)^{\oplax}
		\end{tikzcd}
	\]
	where the right-most vertical map is given by $\FUN({[1]},\ev_0)^{\oplax} \times \FUN({[1]},\ev_1)^{\oplax}$ and the bottom horizontal map selects the constant functors on $x$ and $b$ respectively. To see that this pullback diagram is precisely given by $\bCC(x,b)^{[1]}$ we observe that $\Fun([1],-)^{\oplax}$ preserves limits and consequently we have that the pullback of the diagram above is given by
	\[
		\Fun([1],\bCC(x,b))^{\oplax} \simeq \bCC(x,b)^{[1]}.
	\]
	Unraveling the definitions, we see that $\widetilde{\Psi}$ descends to a map $\Psi \colon Q \to \bCC(x,b)^{[1]}$ which commutes with the projections to $\bCC(x,b) \times \bCC(x,b)$. Our last verification will consist in showing that we have an equivalence of spaces
	\[
		\Gamma_n \colon \Fun([n],Q)^{\simeq} \to \Fun([n],\bCC(x,b)^{[1]})^{\simeq}, \enspace \text{ for } n \in \{0,1\}.
	\]
	We proceed by cases:
	\begin{itemize}
		\item[$n=0$)] Then it follows that $\Fun([0],Q)^{\simeq}= \Fun(C_2,\bCC)^{\simeq}$ and that $\Fun([0],\bCC(x,y)^{[1]}) \simeq \Fun(L,\bCC)^{\simeq}$ where $L$ is obtained from $[1] \otimes [1]$ by collapsing the edges $\{i\}\otimes [1]$ for $i \in \{0,1\}$. We conclude that $\Gamma_0$ is obtained by restriction along the equivalence $L \simeq C_2$.
		\item[$n=1$)] Let $N$ be obtained from $C_2 \otimes [1]$ by collapsing the edges $\{i\}\otimes [1]$ for $i \in \{0,1\}$ and let $M$ be obtained from $[1] \otimes [1] \otimes [1]$ by collapsing $\{i\}\otimes ([1]\otimes [1])$ for $i \in \{0,1\}$ to point. Let $\partial N$ be defined similarly by replacing $C_2$ with $\partial C_2$. As before, $\Gamma_1$ is obtained by restrictiong along the canonical map $M \to N$ which we claim is an equivalence. 

		Let $h_{j}\colon [2] \to P$ for $j\in\{1,2\}$ be the canonical morphisms in the colimit cone defining $P$. We define $U_1$ from $[2] \otimes [1]$ by collapsing $\{0 \to 1\} \otimes [1]$, as well as  $\{2\} \otimes [1]$, to a point. Similarly, we construct $\partial U$ from $[1]\otimes [1]$ by collapsing the edges $\{i\} \otimes [1]$ for $i \in \{0,1\}$. We further define $U_2$ by collapsing $\{1 \to 2\} \otimes [1]$ as well as $\{0\}\otimes [1]$ to a point. It is now easy to verify that the maps
		\[
			r_1 \colon [2] \otimes [1] \to [1] \otimes [1], \enspace (i,j) \mapsto \begin{cases}
				(0,j), \enspace \text{ if } i=0,\\
				(1-i,j), \enspace \text{ else,}
			\end{cases}
		\]
		\[
			r_2 \colon [2] \otimes [1] \to [1] \otimes [1], \enspace (i,j) \mapsto \begin{cases}
				(1,j), \enspace \text{ if } i= 2,\\
				(i,j), \enspace \text{ else,}
			\end{cases}
		\]
		descend to equivalences $U_i \xrightarrow{\simeq} \partial U$. Since the Gray tensor product is compatible with colimits we obtain a pushout square
		\[\begin{tikzcd}
	{\partial N} & N \\
	{U_1 \coprod_{\{0\}\coprod\{2\}}U_2} & M.
	\arrow[from=1-1, to=1-2]
	\arrow[from=1-1, to=2-1]
	\arrow[from=1-2, to=2-2]
	\arrow[from=2-1, to=2-2]
\end{tikzcd}\]
 The previous computation reveals that $U_1 \coprod_{\{0\}\coprod \{2\}}U_2 \simeq \partial N$ and so $\Gamma_1 \colon M \to N$ is an equivalence.

	\end{itemize}
This concludes the proof.
\end{proof}

\begin{corollary}\label{cor:mapcatstrongarrow}
	Let $\bCC$ be a 2-category and let $f \colon x \to y$ and $g \colon a \to b $ be 1-morphisms in $\bCC$. Then we have a strong pullback square
	\[
		\begin{tikzcd}
	\AR_{\bCC}(f,g) & {\bCC(x,a)} \\
	{\bCC(y,b)} & {\bCC(x,b)}.
	\arrow["{\bCC(f,b)}",swap,from=2-1, to=2-2]
	\arrow["{\bCC(x,g)}",from=1-2, to=2-2]
	\arrow[from=1-1, to=2-1]
	\arrow[from=1-1, to=1-2]
\end{tikzcd}
	\]
	of categories.\qed
\end{corollary}



\section{Kan extensions and free cocompletions}\label{sec:kancocompletion}
In this chapter, we develop a version of the theory of Kan extensions and free cocompletions for 2-categories. In \cref{sec:Kanext}, we built upon previous work of the first-named author to derive a general existence theorem of Kan extension for functors between 2-categories. We use this result in \cref{sec:freeCocompletion} to deduce the universal property of 2-categories of categorical presheaves as free cocompletions. We use this universal property in \cref{sec:tensoring} to show that every 2-category with sufficiently many colimits is tensored over $\Cat$ in an essentially unique way.

\subsection{Kan extensions in 2-category theory}\label{sec:Kanext}

\begin{definition}\label{def:oplaxComplete}
	A (large) 2-category $\bCC$ is said to be \emph{(op)lax complete} if $\bCC$ admits all partially (op)lax small limits, and a functor $f\colon \bCC\to\bDD$ is \emph{(op)lax continuous} if it preserves all partially (op)lax small limits that exist in $\bCC$. We simply call a 2-category \emph{complete} if it is both lax and oplax complete, and we similarly call a functor \emph{continuous} if it is both lax and oplax continuous. 
	
	Dually, a (large) 2-category $\bCC$ is said to be \emph{(op)lax cocomplete} if $\bCC$ admits all partially (op)lax small colimits, and a functor $f\colon \bCC\to\bDD$ is \emph{(op)lax cocontinuous} if it preserves all partially (op)lax small colimits that exist in $\bCC$. Likewise, we call a 2-category \emph{cocomplete} if it is both lax and oplax cocomplete, and we call a functor \emph{cocontinuous} if it is both lax and oplax cocontinuous.
\end{definition}

\begin{example}\label{ex:presheavesComplete}
	If $\bII$ is a 2-category, then $\PSh_{\Cat}(\bII)$ is both complete and  cocomplete (see \cref{rem:cat2complete} and \cref{prop:colimfunctorcat}). Moreover, the Yoneda embedding $h_{\bII}\colon\bII\into\PSh_{\Cat}(\bII)$ is continuous (see \cref{prop:yonedacont}).
\end{example}

\begin{definition}\label{def:RKE}
	Let $f\colon\bCC\to\bDD$ be a functor of 2-categories and let $\bEE$ be an arbitrary 2-category. Then the functor $f_\ast$ of \emph{right Kan extension along $f$} is defined as the right adjoint of the restriction map
	\begin{equation*}
		f^\ast\colon\FUN(\bDD,\bEE)\to\FUN(\bCC,\bEE),
	\end{equation*}
	provided that it exists.
	
	Dually, the functor $f_!$ of \emph{left Kan extension along $f$} is defined as the left adjoint of $f^\ast$, provided that it exists.
\end{definition}

\begin{proposition}\label{prop:coYoneda}
	Let $\bII$ be a 2-category and let $F\colon\bII^\op\to \Cat$ be a presheaf. Let $p\colon \LaxOver{\bII}{F}\to \bII$ be the associated  $(1,0)$-fibration, and let $\cart$ denote the marking on $\LaxOver{\bII}{F}$ given by the cartesian edges. Then the canonical map
	\begin{equation*}
		{\colim}^{\arglax{\cart}}_{\LaxOver{\bII}{F}}(h_{\bII}p)\to F
	\end{equation*}
	is an equivalence. Similarly, let $p^\prime\colon \OplaxOver{\bII}{F}\to\bII$ be the $(1,1)$-fibration associated to $F(-)^\op$. Then the canonical map
	\begin{equation*}
		\colim^{\argoplax{\cart}}_{\OplaxOver{\bII}{F}} (h_{\bII}p^\prime)\to F
	\end{equation*}
	is an equivalence.
\end{proposition}
\begin{proof}
	We begin by showing the first statement.
	By \cite[Corollary~4.2.4]{AGH24}, we have an identification
	\begin{equation*}
		{\colim}^{\arglax{\cart}}_{\LaxOver{\bII}{F}}(h_{\bII}p)\simeq{\colim}^F h_{\bII},
	\end{equation*}
	where $\colim^F h_{\bII}$ is the $F$-weighted colimit of $h_{\bII}$. Together with the fact that the universal property of weighted colimits and Yoneda's lemma give rise to equivalences
	\begin{equation*}
		\Nat_{\bII^\op,\Cat}({\colim}^F h_{\bII}, G)\simeq \Nat_{\bII^\op,\Cat}(F, \Nat_{\bII^\op,\Cat}(h_{\bII}(-), G)) \simeq  \Nat_{\bII^\op,\Cat}(F, G)
	\end{equation*}
	that are natural in $G\in\PSh_{\Cat}(\bII)$, this proves the claim.
	
	As for the second statement, observe that $(p^\prime)^\co\colon (\OplaxOver{\bII}{F})^\co\to\bII^\co$ is the $(1,0)$-fibration that straightens to $F(-)^\op\colon(\bII^\co)^\op\to\Cat$. Thus the first part of the proof implies that the canonical map
	\begin{equation*}
		\colim^{\arglax{\cart}}_{(\OplaxOver{\bII}{F})^\co}(h_{\bII^\co} (p^\prime)^\co)\to F(-)^\op
	\end{equation*}
	is an equivalence in $\PSh_{\Cat}(\bII^\co)$. In light of the commutative diagram 
	\begin{equation*}
		\begin{tikzcd}
			& \bII^\co\arrow[dr, hookrightarrow, "h_{\bII^\co}"]\arrow[dl, "h_{\bII}^\co"', hookrightarrow]&\\
			\PSh(\bII)^\co\arrow[rr, "\simeq"] && \PSh(\bII^\co)
		\end{tikzcd}
	\end{equation*}
	in which the horizontal equivalence acts by carrying a presheaf $G$ on $\bII$ to $G(-)^\op$, the claim follows.
\end{proof}

\begin{lemma}\label{lem:productYoneda}
	Let $\bCC$ and $\bEE$ be 2-categories. Then there is a commutative square
	\begin{equation*}
		\begin{tikzcd}
			\bCC\times\bDD\arrow[r, "h_{\bCC\times\bDD}"]\arrow[d, "h_{\bCC}\times h_{\bDD}"] & \PSh_{\Cat}(\bCC\times\bDD)\\
			\PSh_{\Cat}(\bCC)\times\PSh_{\Cat}(\bDD)\arrow[r, "\pr_0^\ast\times\pr_1^\ast"] & \PSh_{\Cat}(\bCC\times\bDD)\times\PSh_{\Cat}(\bCC\times\bDD).\arrow[u, "-\times -"]
		\end{tikzcd}
	\end{equation*}
\end{lemma}
\begin{proof}
	As in the proof of \cite[Lemma~6.2.6]{Martini2022}.
\end{proof}

\begin{lemma}\label{lem:PresheavesProductBilinear}
	For any 2-category $\bCC$,  the product functor $-\times -\colon \PSh_{\Cat}(\bCC)\times\PSh_{\Cat}(\bCC)\to\PSh_{\Cat}(\bCC)$ preserves partially (op)lax colimits in both variables.
\end{lemma}
\begin{proof}
	It suffices to show that for a fixed presheaf $F\colon\bCC^\op\to\Cat$, the functor $-\times F$ is cocontinuous.
	Let $p\colon\LaxOver{\bCC}{F}\to\bCC$ be the unstraightening of $F$. Then $-\times F$ can be identified with the composition
	\begin{equation*}
		\Fib^{(1,0)}(\bCC)\xrightarrow{p^\ast} \Fib^{(1,0)}_{\arglax{\cart}}(\LaxOver{\bCC}{F})\xrightarrow{p_!} \Fib^{(1,0)}(\bCC),
	\end{equation*}
	where $p^\ast$ is given by pullback along $p$ and $p_!$ is the forgetful functor. Now $p_!$ is left adjoint to $p^\ast$ and $p^\ast$ itself has a right adjoint by \cite[Theorem~4.33]{Abellan2023}, hence the claim follows.
\end{proof}

\begin{theorem}\label{thm:existenceRKE}
	Let $f\colon\bCC\to\bDD$ be a functor, and let $\bEE$ be a 2-category which admits either partially oplax limits indexed by $(\OplaxUnder{\bCC}{d}, \cocart)$ or partially lax limits indexed by $(\LaxUnder{\bCC}{d}, \cocart)$, for each $d\in\bDD$.
	Then the functor
	\begin{equation*}
		f^\ast\colon \FUN(\bDD,\bEE)\to\FUN(\bCC,\bEE)
	\end{equation*}
	admits a right adjoint $f_\ast$. In particular, this right adjoint exists if $\bCC$ is small, $\bDD$ is locally small and $\bEE$ is (op)lax complete. Moreover, if $f$ is fully faithful, then $f_\ast$ is fully faithful as well.
\end{theorem}
\begin{proof}
	We begin by showing that $f_\ast$ exists. To that end, note that we have a commutative square
	\begin{equation*}
		\begin{tikzcd}
			\FUN(\bDD,\bEE)\arrow[r, "f^\ast"]\arrow[d, hookrightarrow] & \FUN(\bCC, \bEE)\arrow[d, hookrightarrow]\\
			\FUN(\bEE^\op\times\bDD, \CAT)\arrow[r, "(\id\times f)^\ast"] & \FUN(\bEE^\op\times\bCC, \CAT)
		\end{tikzcd}
	\end{equation*}
	where the vertical maps are induced by postcomposition with the Yoneda embedding $\bEE\into \PSh_{\CAT}(\bEE)$. By \cite[Theorem~4.33]{Abellan2023}, the lower horizontal map has a right adjoint $(\id\times f)_\ast$, so that it suffices to show that this functor restricts to a right adjoint of $f^\ast$. Thus, we have to show that if $F\colon \bCC\to\bEE$ is a functor and if we set $G=\bEE(-, F(-))\colon \bEE^\op\times \bCC\to \CAT$, then $(\id\times f)_\ast(G)(-,d)$ is representable for every $d\in\bDD$. By Yoneda's lemma and \cref{lem:productYoneda}, we may compute
	\begin{align*}
		(\id \times f)_\ast(G)(d)&\simeq \Nat_{\bEE^\op\times\bCC,\CAT}((\id\times f)^\ast h_{\bEE^\op\times\bDD}(-,d), G)\\
		&\simeq \Nat_{\bEE^\op\times\bCC,\CAT}(\pr_0^\ast h_{\bEE}(-)\times\pr_1^\ast f^\ast h_{\bDD^\op}(d), G).
	\end{align*}
	Now by \cref{prop:coYoneda}, we have an equivalence
	\begin{equation*}
		f^\ast h_{\bDD^\op}(d)\simeq {\colim}^{\arglax{\cart}}_{\LaxOver{\bCC^\op}{d}}(h_{\bCC^\op}p),
	\end{equation*}
	where $p\colon \LaxOver{\bCC^\op}{d}\to\bCC^\op$ is the $(1,0)$-fibration classified by $f^\ast h_{\bDD^\op}(d)\colon\bCC\to\Cat$. Thus, by combining \cref{lem:PresheavesProductBilinear} with the fact that $\pr_1^\ast$ is a left adjoint (using again \cite[Theorem~4.33]{Abellan2023}), it follows that we obtain equivalences
	\begin{align*}
		(\id\times f)_\ast(G)(d)&\simeq {\lim}^{\argoplax{\cocart}}_{\OplaxUnder{\bCC}{d}}\Nat_{\bEE^\op\times\bCC,\CAT}(\pr_0^\ast h_{\bEE}(-)\times \pr_1^\ast h_{\bCC^\op}p(-), G)\\
		&\simeq {\lim}^{\argoplax{\cocart}}_{\OplaxUnder{\bCC}{d}}\Nat_{\bEE^\op\times\bCC,\CAT}(h_{\bEE\times\bCC^\op}(-, p(-)), G)\\
		&\simeq {\lim}^{\argoplax{\cocart}}_{\OplaxUnder{\bCC}{d}}G(-, p(-))\\
		&\simeq {\lim}^{\argoplax{\cocart}}_{\OplaxUnder{\bCC}{d}}\bEE(-, Fp(-)).
	\end{align*}
	Consequently, the presheaf $(\id\times f)_\ast(G)(d)$ is a partially oplax limit of a $(\OplaxUnder{\bCC}{d},\cocart)$-indexed diagram of representables. If $\bEE$ admits limits of this shape, we thus find (by \cref{ex:presheavesComplete}) that this presheaf is representable. If $\bEE$ however admits partially lax limits indexed by $(\LaxUnder{\bCC}{d},\cocart)$, we use the equivalence
	\begin{equation*}
		f^\ast h_{\bDD^\op}(d)\simeq \colim^{\argoplax{\cart}}_{\OplaxOver{\bCC^\op}{d}}(h_{\bCC^\op} p^\prime)
	\end{equation*}
	from \cref{prop:coYoneda} instead in the above argument, and thus deduce that the presheaf $(\id\times f)_\ast(g)(d)$ is a partially lax limit of a $(\LaxUnder{\bCC}{d},\cocart)$-indexed diagram of representables and hence representable in this case as well.
	Thus, the claim follows in both cases.
	
	We finish the proof by showing the second statement, i.e.\ that $f_\ast$ is fully faithful whenever $f$ is. Since the first part of the proof shows that $f_\ast$ arises as the restriction of $(\id\times f)_\ast$, it suffices to show that the latter is fully faithful. In other words, we may assume without loss of generality that $\bEE=\Cat$. In this case, Yoneda's lemma implies that we have an equivalence
	\begin{equation*}
		f_\ast(F)\simeq \Nat_{\bCC,\Cat}(f^\ast h_{\bDD^\op}(-), F)
	\end{equation*}
	that is natural in $F\in\PSh_{\Cat}(\bCC^\op)$. Consequently, we may compute
	\begin{equation*}
		f^\ast f_\ast(F)\simeq \Nat_{\bCC,\Cat}(f^\ast h_{\bDD^\op}f(-), F)\simeq \Nat_{\bCC,\Cat}(h_{\bDD^\op}(-), F)\simeq F,
	\end{equation*}
	where the equivalence $f^\ast h_{\bDD^\op}f(-)\simeq h_{\bDD^\op}$ follows from $f$ being fully faithful. As the above chain of equivalences is natural in $F$, we conclude that $f^\ast f_\ast\simeq \id$. It is well known that this already implies that $f_\ast$ is fully faithful, as desired.
\end{proof}

\begin{remark}\label{rem:formulaRKE}
	In the situation of \cref{thm:existenceRKE}, the proof shows that if $F\colon \bCC\to\bEE$ is a functor and $d\in\bDD$ is an arbitrary object, we may compute
	\begin{equation*}
		f_\ast(F)(d)\simeq{\lim}^{\argoplax{\cocart}}_{\OplaxUnder{\bCC}{d}} Fp
	\end{equation*}
	where $p\colon \OplaxUnder{\bCC}{d}\to\bCC$ is the $(0,0)$-fibration classified by $\bDD(d, f(-))^\op\colon\bCC^\co\to\Cat$, provided that this limit exists in $\bEE$. Moreover, we may extract from \cite[Definition~4.32]{Abellan2023} that if $G\colon \bDD\to\bEE$ is a functor, the adjunction unit $G\to f_\ast f^\ast(G)$ evaluates at $d\in\bDD$ to the canonical map
	\begin{equation*}
		G(d)\simeq {\lim}^{\argoplax{\cocart}}_{\OplaxUnder{\bDD}{d}} Gq\to {\lim}^{\argoplax{\cocart}}_{\OplaxUnder{\bCC}{d}} Gfp
	\end{equation*}
	(where $q\colon \OplaxUnder{\bDD}{d}\to \bDD$ is the $(0,0)$-fibration classified by $\bDD(d,-)^\op$)
	that is determined by the map $\OplaxUnder{\bCC}{d}\to\OplaxUnder{\bDD}{d}$.
	
	Likewise, if $\bEE$ is lax complete, we may compute
	\begin{equation*}
		f_\ast(F)(d)\simeq{\lim}^{\arglax{\cocart}}_{\LaxUnder{\bCC}{d}} Fp^\prime,
	\end{equation*}
	where $p^\prime\colon \LaxUnder{\bCC}{d}\to\bCC$ is the $(0,1)$-fibration classified by $\bDD(d, f(-))$, provided that this limit exists in $\bEE$. Furthermore, the adjunction unit $G\to f_\ast f^\ast(G)$ evaluates at $d\in\bDD$ to the canonical map
	\begin{equation*}
		G(d)\simeq {\lim}^{\arglax{\cocart}}_{\LaxUnder{\bDD}{d}} Gq^\prime\to {\lim}^{\arglax{\cocart}}_{\LaxUnder{\bCC}{d}} Gfp^\prime
	\end{equation*}
	(where $q^\prime\colon \LaxUnder{\bDD}{d}\to \bDD$ is the $(0,1)$-fibration classified by $\bDD(d,-)$)
	that is determined by the map $\LaxUnder{\bCC}{d}\to\LaxUnder{\bDD}{d}$.
\end{remark}

By passing to opposite 2-categories, \cref{thm:existenceRKE} implies:
\begin{corollary}\label{cor:existenceLKE}
	Let $f\colon\bCC\to\bDD$ be a functor, and let $\bEE$ be a 2-category which admits either partially oplax colimits indexed by $(\OplaxOver{\bCC}{d}, \cart)$ or partially lax colimits indexed by $(\LaxOver{\bCC}{d}, \cart)$, for each $d\in\bDD$.
	Then the functor
	\begin{equation*}
	f^\ast\colon \FUN(\bDD,\bEE)\to\FUN(\bCC,\bEE)
	\end{equation*}
	admits a left adjoint $f_!$. In particular, this left adjoint exists if $\bCC$ is small, $\bDD$ is locally small and $\bEE$ is (op)lax cocomplete. Moreover, if $f$ is fully faithful, then $f_!$ is fully faithful as well.\qed
\end{corollary}

\begin{remark}\label{rem:formulaLKE}
	In the situation of \cref{cor:existenceLKE}, if $F\colon \bCC\to\bEE$ is a functor and $d\in\bDD$ is an arbitrary object, we may compute 
	\begin{equation*}
		f_!(F)(d)\simeq{\colim}^{\arglax{\cart}}_{\LaxOver{\bCC}{d}} Fp,
	\end{equation*}
	where $p\colon \LaxOver{\bCC}{d}\to\bCC$ is the  $(1,0)$-fibration classified by $\bDD(f(-), d)$, provided that this colimit exists in $\bEE$. Moreover, if $G\colon \bDD\to\bEE$ is a functor, the adjunction counit $f_! f^\ast(G)\to G$ evaluates at $d\in\bDD$ to the canonical map
	\begin{equation*}
		{\colim}^{\arglax{\cart}}_{\LaxOver{\bCC}{d}} Gfp\to {\colim}^{\arglax{\cart}}_{\LaxOver{\bDD}{d}} Gq\simeq G(d)
	\end{equation*}
	(where $q\colon \LaxOver{\bDD}{d}\to \bDD$ is the $(1,0)$-fibration classified by $\bDD(-, d)$)
	that is determined by the map $\LaxOver{\bCC}{d}\to\LaxOver{\bDD}{d}$.
	
	Likewise, we may compute
	\begin{equation*}
		f_!(F)(d)\simeq{\colim}^{\argoplax{\cart}}_{\OplaxOver{\bCC}{d}} Fp^\prime,
	\end{equation*}
	where $p^\prime\colon \OplaxOver{\bCC}{d}\to\bCC$ is the  $(1,1)$-fibration classified by $\bDD(f(-), d)^\op$, provided that this colimit exists in $\bEE$. Moreover, if $G\colon \bDD\to\bEE$ is a functor, the adjunction counit $f_! f^\ast(G)\to G$ evaluates at $d\in\bDD$ to the canonical map
	\begin{equation*}
		{\colim}^{\argoplax{\cart}}_{\OplaxOver{\bCC}{d}} Gfp^\prime\to {\colim}^{\argoplax{\cart}}_{\OplaxOver{\bDD}{d}} Gq^\prime\simeq G(d)
	\end{equation*}
	(where $q^\prime\colon \OplaxOver{\bDD}{d}\to \bDD$ is the $(1,1)$-fibration classified by $\bDD(-, d)^\op$)
	that is determined by the map $\OplaxOver{\bCC}{d}\to\OplaxOver{\bDD}{d}$.
\end{remark}

\begin{example}\label{ex:LKEPresheaves}
	\cref{cor:existenceLKE} implies in particular that for any functor $f\colon \bCC\to\bDD$ between small 2-categories, the functor $f_!\colon\PSh_{\Cat}(\bCC)\to\PSh_{\Cat}(\bDD)$ exists. Moreover, the computation
	\begin{align*}
	\Nat_{\bDD^\op,\Cat}(f_! h_{\bCC}(-), G) &\simeq \Nat_{\bCC^\op,\Cat}(h_{\bCC}(-), f^\ast(G))\\
		&\simeq f^\ast(G)\\
		&\simeq \Nat_{\bDD^\op,\Cat}(h_{\bDD}f(-), G)
	\end{align*}
	that is natural in $G\in\PSh_{\Cat}(\bDD)$ provides us with an equivalence $f_!h_{\bCC}\simeq h_{\bDD}f$.
\end{example}

\begin{theorem}\label{thm:laxcones}
	Let $(\bII,E)$ be a marked 2-category and let $\bCC$ be a 2-category which admits partially oplax colimits of shape $(\bII,E)$. Then the left Kan extension along $\iota \colon \bII \hra \bII_{\elaxoplax}^{\colimcone}$ exists, is fully faithful and has as essential image the full sub-2-category on $E$-(op)lax colimit cones. 
\end{theorem}
\begin{proof}
	It is clear that $\iota_{!}$ will be fully faithful the moment it exists by \cref{cor:existenceLKE}. We start by considering the pullback diagram
	\[
		\begin{tikzcd}
			\bII(j)=\left(\bII^{\colimcone}_{\elaxoplax}\right)_{ \upslash j}\times_{\bII^{\colimcone}_{\elaxoplax}}\bII  \arrow[d] \arrow[r] & \left(\bII^{\colimcone}_{\elaxoplax}\right)_{ \upslash j} \arrow[d] \\
			\bII \arrow[r] &\bII^{\colimcone}_{\elaxoplax}
		\end{tikzcd}
	\]
	and note that $\bII(j)$ can be identified with $\bII_{ \upslash j}$ (cf.\ \cref{rem:inclusionconeff}) whenever $j\neq  *$ where $*$ denotes the cone point. Our first goal is to show is that $\bCC$ admits colimits of shape $\bII(j)$. If $j \neq *$ then this is clear by \cite[Theorem 3.17]{AS23} since then the functor $[0] \to \bII(j)$ selecting the identity map is final (cf.\ \cref{def:final}). For the remaining case, we use the description of the straightening functor \cite[Definition 3.5.1]{LurieGoodwillie}  to see that $\bII(*) \to \bII$ is the $(1,0)$-fibration classifying the functor
	\[
	 	\bII^{\op } \to \Cat, \enspace i \mapsto L_E(\bII_{i\upslash})
	\] 
	where $L_E(\bII_{i\upslash})$ is the localization at all 2-morphisms and at those $0$-cartesian morphisms whose image in $\bII$ belongs to $E$. It also follows from \cite[Theorem 3.17]{AS23} that we have a final map $(\bII,E) \to (\bII(*),E')$ where $E'$ is the collection of $1$-cartesian morphims whose image in $\bII$ belongs to $E$. We conclude that $\iota_!$ exists by \cref{cor:existenceLKE}.

	To finish the proof we need to show that a functor $\hat{F} \colon  \bII_{\elaxoplax}^{\colimcone} \to \bCC$ defines an $\elaxoplax$ colimit cone for $F\simeq \iota^* \hat{F}$ if and only if $\hat{F} \simeq \iota_{!} F$. Let $\iota_!F \to \hat{F}$ be the morphism obtained via the adjunction from the equivalence $F\simeq \iota^* \hat{F}$. Since $\iota_{!}$ is fully-faithful it follows that for every $j \neq *$ the map $\iota_! F(j) \to \hat{F}(j)$ is an equivalence. Let us suppose that $\hat{F}$ is a colimit cone, then it follows that the canonical cone for $\colim_{\bII}^{\elaxoplax} F$ defines a cone for the diagram 
	\[
		\bII(*) \xrightarrow{p} \bII \xrightarrow{F} \bCC.
	\]
	In particular, the map $\iota_!F(*)\simeq \colim_{\bII}^{E'\text{-}\laxoplax}F \circ p  \to \colim_{\bII}^{\elaxoplax}F $ is induced by such cone. It follows from our previous discussion that we have a section of $p$ which is final, so $p$ must itself be final and so the map above is an equivalence. We conclude that we can identify $\iota_! F$ with a colimit cone. The result follows.
\end{proof}

\subsection{cocompletion}
\label{sec:freeCocompletion}

\begin{definition}\label{def:smallPresheaves}
	Let $\bII$ be a (not necessarily small) 2-category and let $F\colon \bII^\op\to\CAT$ be a presheaf. We say that $F$ is \emph{(op)lax small} if there is a small marked 2-category $(\bCC, E)$ and a functor $d\colon \bCC\to \bII$ such that $F\simeq \colim^{\elaxoplax}_{\bCC}(h_{\bII}d)$. We denote by $\Sml_{\CAT}^{\laxoplax}(\bII)$ the full sub-2-category of $\PSh_{\CAT}(\bII)$ that is spanned by the (op)lax small presheaves.
\end{definition}
\begin{remark}\label{rem:representablesAreSmall}
	In the situation of \cref{def:smallPresheaves}, every representable presheaf is trivially (op)lax small. Hence the Yoneda embedding induces an inclusion $\bII\into \Sml_{\CAT}^{\laxoplax}(\bII)$.
\end{remark}

\begin{proposition}\label{prop:smallPresheavesReflection}
	Let $\bII$ be an (op)lax cocomplete 2-category. Then the Yoneda embedding $h\colon \bII\into \Sml_{\CAT}^{\laxoplax}(\bII)$ admits a left adjoint $\lambda$.
\end{proposition}
\begin{proof}
	Let $F$ be an (op)lax small presheaf, so that there is a marked small 2-category $(\bCC, E)$ and a functor $d\colon \bCC\to \bII$ such that $F\simeq \colim^{\elaxoplax}_{\bCC}(h_{\bII}d)$.
	Since $\bII$ is (op)lax cocomplete, we may define $\lambda(F)=\colim^{\elaxoplax}_{\bCC}(d)$, so that we obtain a natural map $\eta\colon F\to h\lambda(F)$. It now follows readily from the universal property of partially lax colimits that the composition
	\begin{equation*}
		\bII(\lambda(F), -)\simeq \Sml_{\CAT}^{\laxoplax}(\bII)(h\lambda(F), h(-))\xrightarrow{\eta^\ast}\Sml_{\CAT}^{\laxoplax}(\bII)(F, h(-))
	\end{equation*}
	is an equivalence. Hence the claim follows.
\end{proof}

\begin{theorem}\label{thm:laxCocompletion}
	For any small 2-category $\bII$, the Yoneda embedding $h\colon\bII\into\PSh_{\Cat}(\bII)$ exhibits $\PSh_{\Cat}(\bII)$ as the free (op)lax cocompletion of $\bII$, in the sense that $\PSh_{\Cat}(\bII)$ is (op)lax cocomplete and the functor of left Kan extension $(h_{\bII})_!\colon \FUN(\bII,\bEE)\into \FUN(\PSh_{\Cat}(\bII),\bEE)$ induces an equivalence
	\begin{equation*}
		\FUN(\bII,\bEE) \simeq \FUN^{\cocont{\laxoplax}}(\PSh_{\Cat}(\bII),\bEE)
	\end{equation*}
	for any (op)lax cocomplete 2-category $\bEE$, where the right-hand side is the full sub-2-category of $\FUN(\PSh_{\Cat}(\bII), \bEE)$ spanned by the (op)lax cocontinuous functors.
\end{theorem}
\begin{proof}
	We begin by showing that every (op)lax cocontinuous functor $g\colon \PSh_{\Cat}(\bII)\to\bEE$ is a left Kan extension of its restriction along $h_{\bII}$. We first show the lax case. If $F\colon \bII^\op\to\Cat$ is an arbitrary presheaf, we obtain a pullback square
	\begin{equation*}
		\begin{tikzcd}
			\LaxOver{\bII}{F}\arrow[r, hookrightarrow, "\LaxOver{(h_{\bII})}{F}"]\arrow[d, "q"] & \LaxOver{\PSh_{\Cat}(\bII)}{F}\arrow[d, "p"]\\
			\bII\arrow[r, hookrightarrow, "h_{\bII}"] & \PSh_{\Cat}(\bII),
		\end{tikzcd}
	\end{equation*}
	and by \cref{rem:formulaLKE} we only need to show that the inclusion $\LaxOver{(h_{\bII})}{F}$ induces
	an equivalence
	\begin{equation*}
		{\colim}^{\arglax{\cart}}_{\LaxOver{\bII}{F}} g \LaxOver{(h_{\bII})}{F}\simeq {\colim}^{\arglax{\cart}}_{\LaxOver{\PSh_{\Cat}(\bII)}{F}} g p.
	\end{equation*}
	As $g$ is by assumption lax cocontinuous, it suffices to show that the map
	\begin{equation*}
		{\colim}^{\arglax{\cart}}_{\LaxOver{\bII}{F}} p\LaxOver{(h_{\bII})}{F}\to {\colim}^{\arglax{\cart}}_{\LaxOver{\PSh_{\Cat}(\bII)}{F}} p.
	\end{equation*}
	is an equivalence, which is an immediate consequence of the first part of  \cref{prop:coYoneda}. In the oplax case, the second part of \cref{prop:coYoneda} allow us to use the very same argument to deduce the claim also in this situation.
	
	To complete the proof, we need to show that if $f\colon\bII\to\bEE$ is a functor, its left Kan extension $(h_{\bII})_!(f)$ is (op)lax cocontinuous. To see this, note that by \cref{ex:LKEPresheaves} and the fact that the functor of left Kan extension $f_!\colon\PSh_{\CAT}(\bII)\to\PSh_{\CAT}(\bEE)$ is cocontinuous, this functor restricts to a map $f_!\colon \Sml_{\CAT}^{\laxoplax}(\bII)\to\Sml_{\CAT}^{\laxoplax}(\bEE)$. Combining this observation with the fact that by \cref{prop:coYoneda} we have an inclusion $\PSh_{\Cat}(\bII)\into\Sml_{\CAT}^{\laxoplax}(\bII)$, we end up with a commutative square
	\begin{equation*}
		\begin{tikzcd}
			\PSh_{\Cat}(\bII)\arrow[d, hookrightarrow] \arrow[r, "f_!"] & \Sml_{\CAT}^{\laxoplax}(\bEE)\arrow[d, hookrightarrow]\\
			\PSh_{\CAT}(\bII)\arrow[r, "f_!"] & \PSh_{\CAT}(\bEE).
		\end{tikzcd}
	\end{equation*}
	Since both the left vertical inclusion and the lower horizontal map are cocontinuous and the right vertical map is fully faithful, the upper horizontal map must be cocontinuous as well. Thus, by combining this functor with the left adjoint $\lambda\colon \Sml_{\CAT}^{\laxoplax}(\bEE)\to\bEE$ from \cref{prop:smallPresheavesReflection}, we end up with a cocontinuous functor $\lambda f_!\colon \PSh_{\Cat}(\bII)\to\bEE$. By its very construction and \cref{ex:LKEPresheaves}, the restriction of this map along $h_{\bII}$ can be identified with $f$. Hence the first part of the proof shows that we must have $(h_{\bII})_!(f)\simeq \lambda f_!$, so that $(h_{\bII})_!(f)$ is (op)lax cocontinuous, as claimed.
\end{proof}

\begin{corollary}
	For any 2-category $\bII$, the Yoneda embedding $h_{\bII}\colon\bII\into\PSh_{\Cat}(\bII)$ exhibits $\PSh_{\Cat}(\bII)$ as the free cocompletion of $\bII$, in the sense that $\PSh_{\Cat}(\bII)$ is cocomplete and the functor of left Kan extension $(h_{\bII})_!\colon\FUN(\bII,\bDD)\into\FUN(\PSh_{\Cat}(\bII),\bDD)$ induces an equivalence
	\begin{equation*}
		\FUN(\bII,\bDD)\simeq\FUN^{\cocont{2}}(\PSh_{\Cat}(\bII),\bDD)
	\end{equation*}
	for any cocomplete 2-category $\bDD$, where the right-hand side is the full sub-2-category of $\FUN(\PSh_{\Cat}(\bII),\bDD)$ spanned by the cocontinuous functors.\qed
\end{corollary}

\begin{remark}\label{rem:universalAdjointFunctorTheorem}
	let $f\colon\bCC\to\bEE$ be a functor where $\bCC$ is small and $\bEE$ is (op)lax cocomplete. Assume furthermore that $\bEE$ is locally small. Then $(h_{\bCC})_!(f)\colon \PSh(\bCC)\to \bEE$ admits a right adjoint that is explicitly given by the composition
	\begin{equation*}
		\bEE\xrightarrow{h_{\bEE}}\PSh_{\Cat}(\bEE)\xrightarrow{f^\ast}\PSh_{\Cat}(\bCC).
	\end{equation*} 
	In fact, this follows immediately from the fact (see the argument in the proof of \cref{thm:laxCocompletion}) that $(h_{\bCC})_!(f)$ is given by the composition $\lambda f_!\colon \PSh_{\Cat}(\bCC)\to\Sml_{\CAT}^{\laxoplax}(\bEE)\to \bEE$.
\end{remark}

\subsection{The canonical $\Cat$-tensoring of (op)lax cocomplete 2-categories}
\label{sec:tensoring}
Let $\bCC$ be an (op)lax cocomplete 2-category. Then $\FUN(\bCC,\bCC)$ is (op)lax cocomplete as well. Therefore, \cref{thm:laxCocompletion} implies that there is a unique (op)lax cocontinuous functor $\Cat\to \FUN(\bCC,\bCC)$ that carries the point $[0]\in\Cat$ to $\id_{\bCC}$. Moreover, since $\id_{\bCC}$ is (op)lax cocontinuous and since $\FUN^{\cocont{\laxoplax}}(\bCC,\bCC)$ is closed under partially (op)lax colimits in $\Fun(\bCC,\bCC)$\footnote{this easily follows from the fact that partially lax colimits can be computed point-wise}, this functor takes values in $\FUN^\cc(\bCC,\bCC)$. Consequently, by transposing this map, one obtain a functor 
\begin{equation*}
	-\otimes -\colon \Cat \times\bCC\to\bCC
\end{equation*}
that is (op)lax cocontinuous in either variable. Furthermore, this functor fits into an equivalence
\begin{equation*}
	\bCC(-\otimes -, -)\simeq \Fun (-, \bCC(-,-)).
\end{equation*}
In fact, by \cref{thm:laxCocompletion} and (op)lax cocontinuity of $-\otimes -$ in the first variable, such an equivalence exists already if there is an equivalence between both sides once we evaluate at $[0]\in\Cat$, which follows from the fact that by construction we have $[0]\otimes -\simeq \id_{\bEE}$. We thus conclude:

\begin{proposition}\label{prop:tensoring}
	For every (op)lax cocomplete 2-category $\bCC$, there is a unique functor
	\begin{equation*}
		-\otimes -\colon \Cat \times\bCC\to\bCC
	\end{equation*}
	that is (op)lax cocontinuous in either variable and that comes with an equivalence $[0]\otimes -\simeq\id_{\bCC}$. Furthermore, this functor fits into an equivalence
	\begin{equation*}
		\bCC(-\otimes -, -)\simeq \Fun(-, \bCC(-,-)).
	\end{equation*}
	of functors $\Cat^\op\times\bCC^\op\times\bCC\to\Cat$.\qed
\end{proposition}
By dualising \cref{prop:tensoring}, we furthermore obtain:
\begin{corollary}\label{cor:cotensoring}
	For every (op)lax complete 2-category $\bCC$, there is a unique functor
	\begin{equation*}
		(-)^{(-)}\colon \Cat^\op\times \bCC\to\bCC
	\end{equation*}
	that is (op)lax continuous in either variable and that comes with an equivalence $(-)^{[0]}\simeq\id_{\bCC}$. Furthermore, this functor fits into an equivalence
	\begin{equation*}
		\bCC(-, (-)^{(-)})\simeq \Fun(-, \bCC(-,-))
	\end{equation*}
	of functors $\Cat^\op\times\bCC^\op\times\bCC\to\Cat$.\qed
\end{corollary}

\begin{remark}\label{rem:variantsTensoringCotensoring}
	If $\bCC$ is an arbitrary 2-category, we always find an embedding $\bCC\into\bEE$ where $\bEE$ is cocomplete and the embedding preserves all partially (op)lax colimits that exist in $\bCC$. In fact, such an embedding is provided by the map $\bCC\into \bEE=\PSh(\bCC^\op)^\op$ that is provided by the Yoneda embedding. Thus \cref{prop:tensoring} implies that we have a tensoring
	\begin{equation*}
		-\otimes -\colon \Cat\times\bEE\to\bEE,
	\end{equation*}
	and given $\II\in\Cat$ and $c\in\bCC\subset\bEE$, the object $\II\otimes c$ is contained in $\bCC$ precisely if the fully lax colimit of the constant diagram $\II\to\bCC$ (or equivalently the fully oplax colimit of the constant diagram $\bII^\op\to \bCC$) with value $c$ exists in $\bCC$.
	The dual statement for (op)lax limits holds as well.
\end{remark}

\section{Presentable 2-categories}\label{sec:2presentable}
In this chapter, we develop a few aspects of the notion of presentability in 2-category theory. Recall that a 1-category is defined to be presentable if it is cocomplete and accessible. We will likewise define a 2-category to be \emph{presentable} if it is cocomplete and if the underlying 1-category is accessible. In other words, we understand accessibility as an entirely 1-dimensional concept, which allows us to leverage much of the well-developed theory of presentable 1-categories to keep our arguments rather concise. 

We begin in \cref{subsec:localObjects} by briefly discussing 2-dimensional Bousfield localisations. We use these results in \ref{subsec:2present} to characterise presentable 2-categories, in a similar fashion as the Lurie-Simpson characterisation of presentable 1-categories (see \cite[Theorem~5.5.1.1]{LurieHTT}). Lastly, we derive 2-dimensional adjoint functor theorems in \cref{subsec:adjointFunctorTheorems}.

\subsection{Local objects in a 2-category}\label{subsec:localObjects}

\begin{definition}\label{def:localObjects}
	Let $\bCC$ be a 2-category and let $S$ be a set of morphisms in $\bCC$. An object $c\in\bCC$ is said to be \emph{$S$-local} if for every map $s\colon x\to y$ in $S$ the functor
	\begin{equation*}
		s^\ast\colon\bCC(y,c)\to\bCC(x,c)
	\end{equation*}
	is an equivalence. We denote by $\Loc_S(\bCC)$ the full sub-2-category of $\bCC$ that is spanned by the $S$-local objects.
\end{definition}

\begin{lemma}\label{lem:LocalObjects1vs2dimensional}
	Let $\bCC$ be an (op)lax cocomplete 2-category, and let $S$ be a set of maps in $\bCC$ that is preserved by the functor $[1]\otimes -$. Then we have $\Loc_S(\bCC)^{\leq 1}=\Loc_S(\bCC^{\leq 1})$.
\end{lemma}
\begin{proof}
	An object $c\in\bCC$ is $S$-local in $\bCC$ if and only if the two maps
	\begin{equation*}
		\bCC(y, c)^\core\to\bCC(x, c)^\core\qquad\qquad \Cat([1], \bCC(y, c))^\core \to \Cat([1], \bCC(x, c))^\core
	\end{equation*}
	are equivalences, for all $s\colon y\to x$ in $S$. Now the left map can be identified with $\bCC^{\leq 1}(y, c)\to\bCC^{\leq 1}(x,c)$, and the right map can be identified with $\bCC^{\leq 1}([1]\otimes y, c)\to\bCC^{\leq 1}([1]\otimes x, c)$ (see \cref{prop:tensoring}). Since by assumption $[1]\otimes s\in S$, the claim follows.
\end{proof}

\begin{remark}\label{rem:SaturationLocalObjects}
	Suppose that $\bCC$ is an (op)lax cocomplete 2-category, and let $S$ be a set of maps in $\bCC$. Let $T$ be the smallest collection of maps in $\bCC$ that contains $S$ and that is preserved by $[1]\otimes -$. Then $\Loc_S(\bCC)=\Loc_T(\bCC)$. In fact, given $c\in\bCC$, the functor $\bCC(-, c)\colon \bCC\to\Cat^\op$ is cocontinuous, which immediately implies that the class of maps in $\bCC$ that are inverted by $\bCC(-,c)$ is closed under all partially (op)lax colimits in $\AR_{\bCC}$. This is sufficient to deduce the claim. Moreover, note that the class $T$ is still small. To see this, let us set $S_0 = S$, and let us inductively define $S_n$ as the union of $S_{n-1}$ with the image of $S_{n-1}$ along $[1]\otimes -$. Then it is easy to see that we have $T=\bigcup_n S_n$, so that $T$ must be small as well. Together with \cref{lem:LocalObjects1vs2dimensional}, we thus conclude that $\Loc_S(\bCC)^{\leq 1}=\Loc_T(\bCC^{\leq 1})$.
\end{remark}

\begin{lemma}\label{lem:adjunctions1vs2dimensional}
	Let $\bCC$ be an (op)lax cocomplete 2-category, and let $f\colon \bCC\to\bDD$ be an (op)lax cocontinuous functor. Suppose that $f^{\leq 1}$ is a left adjoint. Then $f$ is a left adjoint as well.
\end{lemma}
\begin{proof}
 	Let $g\colon \bDD^{\leq 1}\to\bCC^{\leq 1}$ be the right adjoint of $f^{\leq 1}$, and let $\epsilon\colon f^{\leq 1}g\to \id$ be the adjunction counit. We need to show that for every object $d\in\bDD$ and every $c\in\bCC$, the map
 	\begin{equation*}
 		\bCC(c, g(d))\to \bDD(f(c), fg(d))\xrightarrow{\epsilon_\ast} \bCC(f(c), d)
 	\end{equation*}
 	is an equivalence. By assumption, it induces an equivalence on core groupoids, so it suffices to check that the induced map 
 	\begin{equation*}
 		\Cat([1], \bCC(c, g(d)))^\core\to \Cat([1], \bCC(f(c), d))^\core
 	\end{equation*}
 	is an equivalence as well. By \cref{prop:tensoring}, \cref{rem:variantsTensoringCotensoring} and the assumption on $f$, this map can be identified with the composition
 	 \begin{equation*}
 		\bCC([1]\otimes c, g(d))^\core\to \bDD(f([1]\otimes c), fg(d))^\core\xrightarrow{\epsilon_\ast}\bCC(f([1]\otimes c), d)^\core,
 	\end{equation*}
 	hence the claim follows.
\end{proof}

\begin{lemma}\label{lem:PresheavesPresentable}
	For any small 2-category $\bII$, the $1$-category $\PSh_{\Cat}(\bII)^{\leq 1}$ is presentable.
\end{lemma}
\begin{proof}
	By \cite[Theorem~3.8.1]{LurieGoodwillie}, the $1$-category $\PSh_{\Cat}(\bII)^{\leq 1}$ can be presented by a combinatorial simplicial model category and is thus presentable by \cite[Proposition~A.3.7.6]{LurieHTT}.
\end{proof}

\begin{proposition}\label{prop:localObjectsPresheaves}
	Let $\bII$ be a small 2-category, and let $S$ be a set of maps in $\PSh_{\Cat}(\bII)$. Then the inclusion $\iota\colon\Loc_S(\PSh_{\Cat}(\bII))\subset\PSh_{\Cat}(\bII)$ admits a left adjoint $l$. Moreover, $\iota$ preserves $\kappa$-filtered colimits for some regular cardinal $\kappa$.
\end{proposition}
\begin{proof}
	By enlarging $S$ if necessary, we may assume without loss of generality that $S$ is preserved by the functor $[1]\otimes -$, cf.~\cref{rem:SaturationLocalObjects}. In particular, \cref{lem:LocalObjects1vs2dimensional} implies that $\Loc_S(\PSh_{\Cat}(\bII))^{\leq 1}=\Loc_S(\bCC^{\leq 1})$. Since $\PSh_{\Cat}(\bII)^{\leq 1}$ is presentable by \cref{lem:PresheavesPresentable} and $S$ is small, this implies that the inclusion $\Loc_S(\bCC)^{\leq 1}\into\bCC^{\leq 1}$ preserves $\kappa$-filtered colimits (for some regular cardinal $\kappa$) and admits a left adjoint. Now the claim follows from (the dual of) \cref{lem:adjunctions1vs2dimensional} in light of the observation that $\Loc_S(\PSh_{\Cat}(\bII))$ is closed under all partially (op)lax limits in $\PSh_{\Cat}(\bII)$.
\end{proof}

\subsection{2-Presentability}\label{subsec:2present}

\begin{definition}\label{def:presentability}
	A 2-category $\bCC$ is said to be \emph{presentable} if $\bCC$ is cocomplete and the underlying $1$-category $\bCC^{\leq 1}$ is accessible.
\end{definition}

\begin{remark}\label{rem:2presentableImplies1presentable}
	If $\bCC$ is presentable, then $\bCC^{\leq 1}$ is cocomplete and accessible and thus presentable.
\end{remark}

\begin{remark}\label{rem:2presentableLocallySmall}
	If $\bCC$ is (op)lax cocomplete, then $\bCC$ being locally small is equivalent to $\bCC^{\leq 1}$ being locally small. In fact, the condition is trivially necessary, and that it is also sufficient follows from the observation that given objects $c,d\in\bCC$, the 1-category $\bCC(c,d)$ being small is equivalent to $\Cat([1], \bCC(c,d))^\core$ being small and that the latter can be identified with $\bCC^{\leq 1}([1]\otimes c, d)$ by \cref{prop:tensoring}. In particular, every presentable 2-category is locally small.
\end{remark}

\begin{lemma}\label{lem:compactObjects2presentable}
	Let $\bCC$ be an (op)lax cocomplete 2-category, and suppose that $\bCC^{\leq 1}$ is accessible. Then there exists a regular cardinal $\kappa$ such that $\bCC^{\leq 1}$ is $\kappa$-accessible and an object $c\in\bCC^{\leq 1}$ is $\kappa$-compact if and only if it is $\kappa$-compact in $\bCC$, in the sense that the functor
	\begin{equation*}
		\bCC(c,-)\colon\bCC \to \Cat
	\end{equation*} 
	preserves $\kappa$-filtered colimits.
\end{lemma}
\begin{proof}
	Since $[1]\otimes -\colon\bCC^{\leq 1}\to \bCC^{\leq 1}$ is cocontinuous, it is in particular an accessible functor. Consequently, there exists a regular cardinal $\kappa$ such that $\bCC^{\leq 1}$ is $\kappa$-accessible and $[1]\otimes -$ preserves $\kappa$-compact objects. Thus, if $c\in\bCC^{\leq 1}$ is $\kappa$-compact, the fact that we have an equivalence
	\begin{equation*}
		\Cat([1], \bCC(c,-))^\core\simeq\bCC([1]\otimes c, c)^\core
	\end{equation*}
	(see \cref{prop:tensoring}) implies that $\bCC(c, -)$ preserves $\kappa$-filtered colimits. Hence the claim follows.
\end{proof}

\begin{theorem}\label{thm:characterisationPresentableCategories}
	For a 2-category $\bCC$, the following are equivalent:
	\begin{enumerate}
		\item $\bCC$ is presentable;
		\item $\bCC$ is (op)lax cocomplete and $\bCC^{\leq 1}$ is accessible;
		\item $\bCC$ is (op)lax cocomplete, locally small, and there is a regular cardinal $\kappa$ and a small full sub-2-category $\bCC_0\subset \bCC$ that consists of $\kappa$-compact objects such that every object in $\bCC$ is the colimit of a $\kappa$-filtered diagram with values in $\bCC_0$;
		\item there is a small 2-category $\bII$ such that $\bCC$ arises as a reflective sub-2-category of $\PSh_{\Cat}(\bII)$ in which the inclusion preserves $\kappa$-filtered colimits, for some regular cardinal $\kappa$.
		\item there is a small 2-category $\bII$ and a small set of maps $S$ in $\PSh_{\Cat}(\bII)$ so that $\bCC\simeq \Loc_S(\PSh_{\Cat}(\bII))$;
	\end{enumerate}
\end{theorem}
\begin{proof}
	Trivially,~(1) implies~(2). Now suppose that (2) is satisfied. Then \cref{rem:2presentableLocallySmall} implies that $\bCC$ is locally small. Moreover, \cref{lem:compactObjects2presentable} implies that there is a regular cardinal $\kappa$ such that the full sub-2-category $\bCC^\kappa$ of $\kappa$-compact objects in $\bCC$ is small and such that every object in $\bCC$ is a $\kappa$-filtered colimit of $\kappa$-compact objects. Hence (3) holds. 
	
	If (3) is satisfied, let $\iota\colon \bCC_0\subset\bCC$ denote the inclusion. Since $\bCC$ is locally small, the Yoneda extension $l=(h_{\bCC^\kappa})_!(\iota)\colon\PSh_{\Cat}(\bCC_0)\to\bCC$ admits a right adjoint $r$ that is explicitly given by sending $c\in\bCC$ to the presheaf $\bCC(\iota(-), c)$ (cf.~\cref{rem:universalAdjointFunctorTheorem}). Since $\bCC_0$ consists of $\kappa$-compact objects in $\bCC$ and since $\kappa$-filtered colimits in $\PSh_{\Cat}(\bCC_0)$ are computed point-wise, we deduce that $r$ preserves $\kappa$-filtered colimits. To deduce that (4) holds, it thus suffices to see that $r$ is fully faithful. So let $c,d\in\bCC$, and consider the natural map
	\begin{equation*}
		\bCC(c,d)\to \PSh(\bCC_0)(r(c), r(d)).
	\end{equation*}
	Since $r$ preserves $\kappa$-filtered colimits and $c$ is a $\kappa$-filtered colimit of objects in $\bCC_0$, we may reduce to the case where $c$ is already $\kappa$-compact. In this case, $r(c)$ is represented by $c$, so that Yoneda's lemma implies the claim.
	
	Now if (4) holds, $\bCC$ is cocomplete, and $\bCC^{\leq 1}$ is an accessible Bousfield localisation of $\PSh_{\Cat}(\bII)^{\leq 1}$. Since the latter is presentable by \cref{lem:PresheavesPresentable}, we find that $\bCC$ is presentable, i.e.\ that (1) holds.
	
	Lastly, we show that (4) and (5) are equivalent. The fact that (5) implies (4) follows immediately from \cref{prop:localObjectsPresheaves}. Conversely, suppose that (4) holds. Since the first part of the proof in particular implies that $\PSh_{\Cat}(\bII)$ is presentable, \cref{lem:compactObjects2presentable} implies that by enlarging $\kappa$ if necessary, we may assume that $\PSh_{\Cat}(\bII)^{\leq 1}$ is $\kappa$-accessible and an object in $\PSh_{\Cat}(\bII)$ is $\kappa$-compact if and only if it is $\kappa$-compact in $\PSh_{\Cat}(\bII)^{\leq 1}$. Define $S$ to be the set of maps between $\kappa$-compact objects in $\PSh_{\Cat}(\bII)$ that are inverted by the localisation functor. Clearly, we have $\bCC\subset \Loc_S(\PSh_{\Cat}(\bCC))$, so that it suffices to show the converse inclusion. So let us pick an $S$-local object $x\in\PSh_{\Cat}(\bII)$, and let $\eta\colon \id\to rl$ be the unit of the adjunction $(l\dashv r)\colon\PSh_{\Cat}(\bII)\leftrightarrows\bCC$. Note that we can identify $\bCC$ with the full sub-2-category of $\PSh_{\Cat}(\bII)$ spanned by the $T$-local objects, where $T$ is the collection of maps of the form $\eta_z\colon z\to rl(z)$ for $z\in\PSh_{\Cat}(\bII)$. Thus, to show that $x$ is contained in $\bCC$, we need to verify that the map
	\begin{equation*}
		\eta_z^\ast\colon \PSh_{\Cat}(\bII)(rl(z), x)\to\PSh_{\Cat}(\bII)(z, x)
	\end{equation*}
	is an equivalence. Since $rl$ preserves $\kappa$-filtered colimits and since every object in $\PSh_{\Cat}(\bII)$ is obtained as a $\kappa$-filtered colimit of $\kappa$-compact objects in $\PSh_{\Cat}(\bII)$, we may assume without loss of generality that $z$ is itself $\kappa$-compact. Since in this case $\eta_z$ is by definition contained in $S$, the claim follows.
\end{proof}

\begin{proposition}\label{prop:localisationPresentable}
	Let $\bCC$ be a presentable 2-category and let $\bDD\into\bCC$ be a reflective sub-2-category. Suppose that the inclusion preserves $\kappa$-filtered colimits for some regular cardinal $\kappa$. Then $\bDD$ is presentable.
\end{proposition}
\begin{proof}
	$\bDD$ being a reflective sub-2-category implies that $\bDD$ is cocomplete, and the inclusion preserving $\kappa$-filtered colimits implies that $\bDD^{\leq 1}$ is an accessible Bousfield localisation of $\bCC^{\leq 1}$ and hence accessible.
\end{proof}

\begin{proposition}\label{prop:functorCatPresentable}
	Let $\bCC$ be a presentable 2-category and let $\bII$ be a small 2-category. Then $\Fun(\bII,\bCC)$ is presentable.
\end{proposition}
\begin{proof}
	By \cref{thm:characterisationPresentableCategories}, we can find a small 2-category $\bJJ$ such that $\bCC$ is a reflective sub-2-category of $\PSh_{\Cat}(\bJJ)$ where the inclusion preserves $\kappa$-filtered colimits. Hence we obtain a fully faithful right adjoint $\Fun(\bII,\bCC)\into\PSh_{\Cat}(\bII^\op\times\bJJ)$ that preserves $\kappa$-filtered colimits (because they are computed object-wise). Thus another application of \cref{thm:characterisationPresentableCategories} yields the claim.
\end{proof}

\begin{remark}\label{rem:accfunctorcat}
	The result from \cref{prop:functorCatPresentable} can be slightly refined: if $\bII$ is $\kappa$-small in the sense of \cref{def:kappasmall2cat} and $\bCC^{\leq 1}$ is $\kappa$-accessible, then $\FUN(\bII, \bCC)^{\leq 1}$ is also $\kappa$-accessible. In fact, if we set $\II=\bII^{\leq 1}$ and $\CC=\bCC^{\leq 1}$ and if $i\colon \II\subset\bII$ denotes the inclusion, \cref{cor:existenceLKE} gives rise to an adjunction
	\[
		i_! \colon \FUN(\II,\bCC) \llra \FUN(\bII,\bCC) \colon i^*
	\]
	in which $i_!$ preserves $\kappa$-compact objects since $i^\ast$ commutes with $\kappa$-filtered colimits.
	Now as $ \FUN(\II,\bCC)^{\leq 1} \simeq \Fun(\II,\CC)$, it follows from \cite[Proposition 5.4.4.3]{LurieHTT} that $\FUN(\II,\bCC)^{\leq 1}$ is $\kappa$-accessible. We define $\widetilde{\DD}_0 \subset \FUN(\bII,\bCC)^{\leq 1}$ to be the full subcategory generated by objects of the form $i_!F$ for $F$ a $\kappa$-compact object. Finally, let $\DD_0$ be the closure of $\widetilde{\DD}_0$ under $\kappa$-small colimits. Let us point out that $\DD_0$ still consists of $\kappa$-compact objects.
	In order to conclude that $\FUN(\bII,\bCC)^{\leq 1}$ is generated under $\kappa$-filtered colimits by $\DD_0$, we need to show that the mapping space functors $\Nat_{\bII,\bCC}(F,-)^{\simeq}$ are jointly conservative for $F \in \DD_0$  (see \cite[Proposition 1.7]{harr2023}). We will show that the functors $\Nat_{\bII,\bCC}(F,-)^{\simeq}$ are jointly conservative for $F \in \widetilde{\DD}_0$. Indeed, if $G \to L$ becomes an equivalence after applying $\Nat_{\bII,\bCC}(F,-)^{\simeq}$ for $F \in \widetilde{\DD}_0$ it follows that $i^*G \to i^*L$ must be an equivalence in $\Fun(\II,\CC)$ since $\kappa$-compact objects are jointly conservative. Thus, the claim follows from $i^*$ being conservative.
	
	Furthermore, it follows from the colimit formula for left Kan extensions and \cref{cor:smallcolimofcompacts} that every $F\in \DD_0$ factors through $\bCC^{\kappa}$. Hence $\FUN(\bII,\bCC)$ is generated under $\kappa$-filtered colimits from $\FUN(\bII,\bCC^\kappa)$.
\end{remark}

\begin{remark}\label{rem:co2presentable}
	If $\bCC$ is presentable then it follows that $\bCC^\co$ is also again presentable since $(\bCC^\co)^{\leq 1} \simeq \bCC^{\leq 1}$.
\end{remark}

\subsection{Adjoint functor theorems}\label{subsec:adjointFunctorTheorems}

\begin{proposition}\label{prop:adjointFunctorTheoremLeft}
	Let $\bCC$ be a presentable 2-category and let $\bDD$ be a locally small 2-category. Let $f\colon \bCC\to\bDD$ be a functor. Then the following are equivalent:
	\begin{enumerate}
		\item $f$ is a left adjoint;
		\item $f$ is cocontinuous;
		\item $f$ is (op)lax cocontinous.
	\end{enumerate}
\end{proposition}
\begin{proof}
	The only non-trivial part is showing that (3) implies (1).
	By \cref{lem:adjunctions1vs2dimensional}, it suffices to show that $f^{\leq 1}$ is a left adjoint, which follows from the $1$-categorical adjoint functor theorem, see \cite[Theorem~4.1.3]{Nguyen2020}.
\end{proof}

\begin{corollary}\label{cor:sheavesRepresentable}
	Let $\bCC$ be a presentable 2-category. Then a presheaf $F\colon\bCC^\op\to\Cat$ is representable if and only if it preserves (op)lax limits.
\end{corollary}
\begin{proof}
	By \cref{prop:adjointFunctorTheoremLeft}, $F$ preserves (op)lax limits if and only if it has a left adjoint $l$. If this is the case, the fact that we have $\Cat([0],-)\simeq\id_{\Cat}$ implies that $F$ is representable by $l([0])$. The converse direction is trivial.
\end{proof}

\begin{proposition}\label{prop:adjointFunctorTheoremRight}
	Let $\bDD$ be a presentable 2-category and let $\bCC$ be locally small such that each $c\in\bCC_{\leq1}$ is $\kappa$-compact for some regular cardinal $\kappa$. Let $g\colon \bDD\to\bCC$ be a functor that preserves $\tau$-filtered colimits for some cardinal $\tau$. Then the following are equivalent:
	\begin{enumerate}
		\item $g$ is a right adjoint;
		\item $g$ is continuous;
		\item $g$ is (op)lax continuous.
	\end{enumerate}
\end{proposition}
\begin{proof}
	The only non-trivial part is to show that (3) implies (1). For this, the dual of \cref{lem:adjunctions1vs2dimensional} implies that it suffices to show that $g^{\leq 1}$ is a right adjoint, which follows from the $1$-categorical adjoint functor theorem, see \cite[Theorem~4.1.1]{Nguyen2020}.
\end{proof}

%% file: internfib.tex
\section{Fibrations in a 2-category}\label{sec:internfib}

In this chapter, we will develop the theory of internal fibrations in an arbitrary 2-category $\bCC$ that admits oriented pullbacks (see \cref{subsec:oriented}). The main idea, which is due to Street, is that the notion of a (co)cartesian fibration of $1$-categories can be entirely expressed using \emph{free} (co)cartesian fibrations and adjunctions. Since free (co)cartesian fibrations can moreover be defined via oriented pullbacks, the theory of fibrations can be built by only making use of abstract 2-categorical constructions and can hence be interpreted internal to any 2-category which admits oriented pullbacks.

In \cref{subsec:freeFibrations}, we provide a few preliminary results on free fibrations in a 2-category, before we define and study what we call \emph{$\epsilon$-fibrations} (for $\epsilon\in\{0,1\}$) in a 2-category in \cref{subsec:fibrations}.
\subsection{Free fibrations}\label{subsec:freeFibrations}
Let us fix a 2-category $\bCC$ that has oriented pullbacks (see \cref{subsec:oriented}). Observe that there is a commutative diagram
\begin{equation*}
	\begin{tikzcd}
		\AR_{\bCC}\arrow[rr, "\iota"]\arrow[dr, "\ev_1"'] &&\ARlax_{\bCC}\arrow[dl, "\ev_1"]\\
		&\bCC.
	\end{tikzcd}
\end{equation*}
\begin{proposition}\label{prop:semiLaxPullbackRightAdjoint}
	The inclusion $\iota\colon\AR_{\bCC}\into\ARlax_{\bCC}$ has a relative right adjoint $\Free^{0}$ over $\bCC$ that carries a map $p\colon x\to c$ to its oriented pullback $\Free^{0}_c(p)\colon\Free^0_c(x) \to c$ along $\id_c$.
\end{proposition}
\begin{proof}
	Let us fix a map $p\colon x\to c$. By its universal property, the oriented pullback $\Free^{0}_c(p)\colon\Free^{0}_c(x)\to c$ fits into a universal lax square
	\begin{equation*}
		\begin{tikzcd}
			\Free^{0}_c(x)\arrow[d, "\Free^{0}_c(p)"']\arrow[r] & x\arrow[d, "p"]\arrow[dl, Rightarrow, "\sigma"', shorten >=1.5em , shorten <=1.5em ]\\
			c\arrow[r, "\id_c"]& c.
		\end{tikzcd}
	\end{equation*}
	To show that this construction defines a right adjoint of $\iota$, we need to verify that for every $q\colon y\to c$ the composition
	\begin{equation*}
		\AR_{\bCC}(q,\Free^{0}_c(p))\into\ARlax_{\bCC}(q, \Free^{0}_c(p))\xrightarrow{\sigma_\ast}\ARlax_{\bCC}(q, p)
	\end{equation*}
	is an equivalence. Using the description of the mapping categories in $\AR_{\bCC}$ and in $\ARlax_{\bCC}$ as strong and oriented pullbacks (cf.\  \cref{prop:mapcatlaxarrow} and \cref{cor:mapcatstrongarrow}), respectively, this amounts to showing that in the diagram
	\begin{equation*}
		\begin{tikzcd}
			\AR_{\bCC}(q,\Free^{0}_c(p))\arrow[d]\arrow[r] & \bCC(y,\Free^{0}_c(x))\arrow[r]\arrow[d] & \bCC(y,x)\arrow[d] \arrow[dl, Rightarrow, shorten >=1.5em, shorten <=1.5em]\\
			\bCC(d,c)\arrow[r] & \bCC(y,c)\arrow[r]& \bCC(y,c)
		\end{tikzcd}
	\end{equation*}
	in which the left square is a strong pullback and the right square is a oriented pullback, the composite square is an oriented pullback as well. The claim thus follows from \cref{prop:pasting}. Hence, we obtain a right adjoint $\Free^{0}$ of $\iota$ with the property that the adjunction counit $\Free^{0}_c(p)\to p$ recovers the lax square $\sigma$ from above. Since $\ev_1(\sigma)$ recovers the identity on $\id_c$, we deduce from \cite[Proposition 5.2.2.]{AGH24} that $\Free^{0}$ defines a relative right adjoint of $\iota$.
\end{proof}

Recall from \cref{def:laxslices} that for any $c\in\bCC$, the \emph{lax slice} $\bCC_{\downslash c}$ is defined as the pullback
	\begin{equation*}
		\begin{tikzcd}
		\bCC_{\downslash c}\arrow[r]\arrow[d]&{\ARlax_{\bCC}}\arrow[d, "\ev_1"]\\
		{[0]}\arrow[r, "c"] & \bCC.
		\end{tikzcd}
	\end{equation*}
By passing to fibres over $c\in\bCC$, the inclusion $\iota\colon\AR_{\bCC}\into\ARlax_{\bCC}$ now gives rise to a functor
\begin{equation*}
	\iota_c\colon \bCC_{/c}\into \bCC_{\downslash c}
\end{equation*}

Since taking pullbacks defines a functor of 2-categories, \cref{prop:semiLaxPullbackRightAdjoint} now implies:
\begin{corollary}\label{cor:semiLaxPullbackRightAdjointFixedBase}
	If $\bCC$ has oriented pullbacks, then $\iota_c$ admits a right adjoint $\Free^{0}_c$ that is given by the oriented pullback along the identity on $c$.\qed
\end{corollary}

\begin{remark}\label{rem:freeFibExplicitly}
	If $p\colon x\to c$ is a map, the map $\Free^{0}_c(p)\colon\Free^{0}_c(x)\to c$ fits into a commutative diagram
	\begin{equation*}
		\begin{tikzcd}
			\Free^{0}_c(x)\arrow[d]\arrow[r]\arrow[rr, bend left, "\Free^{0}_c(p)"] & c^{[1]}\arrow[d, "d_0"]\arrow[r, "d_1"] & c\\
			x\arrow[r, "p"] & c &
		\end{tikzcd}
	\end{equation*}
	where the square is a (strong) pullback, $c^{[1]}$ is the oriented pullback of the identity on $c$ along itself and $d_1,d_0\colon c^{[1]}\rightrightarrows c$ are the associated two projections. To see this, one can reduce to the case where $\bCC=\Cat$, in which case this follows by explicit verification.
\end{remark}

\subsection{$\epsilon$-Fibrations}\label{subsec:fibrations}
Recall from \cref{def:reflections} that a map $x\to y$ in $\bCC$ is said to be a \emph{left inclusion} if it is a fully faithful right adjoint. In this case, its left adjoint is referred to as a \emph{left reflection}. We may now define:
\begin{definition}\label{def:fibrations}
	A map $p\colon x\to c$ in $\bCC$ is said to be  a \emph{0-fibration} if the adjunction unit $\eta_p\colon x\to \Free^{0}_c(x)$ is a left inclusion in $\bCC_{/c}$ with left adjoint $l_p$.  A map $p\to q$ in $\AR(\bCC)$ between 0-fibrations 	defines a \emph{morphism of 0-fibrations} if the naturality square 
	\begin{equation*}
		\begin{tikzcd}
			p\arrow[d]\arrow[r, "\eta_p"] & \Free^{0}_c(p)\arrow[d]\\
			q\arrow[r, "\eta_q"] & \Free^{0}_d(q) 
		\end{tikzcd}
	\end{equation*}
	is left adjointable. We denote by $\Fib^{0}_{\bCC}\into\AR_{\bCC}$ the subcategory spanned by the 0-fibrations and the morphisms of 0-fibrations.
\end{definition}
\begin{remark}\label{rem:minimalConditionsFibration}
	A map $p\colon x\to c$ is a 0-fibration if and only if the map $\eta_p\colon x\to \Free^{0}_c(x)$ admits a retraction $l_p$ in $\bCC_{/c}$ which defines a left adjoint in $\bCC$. In fact, this is clearly necessary. Conversely, note that $\eta_p$ is fully faithful (being the canonical map from the strong limit of $p\colon [1]\to\bCC$ to its lax limit). Consequently, if $\eta_p$ admits a retraction $l_p$ in $\bCC_{/c}$ which defines a left adjoint in $\bCC$, then the adjunction counit must be an equivalence and therefore lifts to a $2$-morphism in $\bCC_{/c}$. The adjunction unit, on the other hand, lifts to a $2$-morphism in $\bCC_{/c}$ as it becomes an equivalence when whiskered with $l_p$ and therefore (since $l_p$ is a retraction in $\bCC_{/c}$) a fortiori when whiskered with the structure map $\Free^{0}_{c}(x)\to c$. Hence the claim follows.
\end{remark}
\begin{remark}\label{rem:morphismsOfFibrations}
	A morphism of 0-fibrations is simply a 0-fibration in $\AR_{\bCC}$.
\end{remark}

\begin{example}\label{ex:fibrationsCat}
	In $\Cat$, a 0-fibration is precisely a cocartesian fibration, and a morphism of 0-fibrations is precisely a morphism of cocartesian fibrations, i.e.\ a commutative square in which the vertical maps are cocartesian fibrations and the top horizontal map preserves cocartesian edges.
\end{example}

\begin{example}\label{ex:fibrationsFunctorCategories}
	If $\bII$ is a 2-category, then a map in $\Fun(\bII,\bCC)$ is a 0-fibration precisely if it takes values in $\Fib_{\bCC}^0$ when viewed as a functor $\bII\to \AR_{\bCC}$. Likewise, a commutative square in $\Fun(\bII,\bCC)$ defines a morphism of 0-fibrations if it takes values in $\Fib_{\bCC}^0$ when viewed as a functor $[1]\times\bII\to\AR_{\bCC}$. In other words, it defines a morphism of 0-fibrations precisely if it does so point-wise, i.e.\ if for every $i\in \bII$ the induced square in $\bCC$ defines a morphism of 0-fibrations. In particular, we may describe 0-fibrations and morphisms of 0-fibrations representably: a map $f\colon x\to c$ in $\bCC$ is a 0-fibration if and only if the induced map $f_\ast\colon\bCC(-, x)\to \bCC(-, c)$ takes values in $\Fib_{\Cat}^{0}$ when viewed as a map $\bCC^\op\to \AR_{\Cat}$, and a commutative square in $\bCC$ defines a morphism of 0-fibrations precisely if $\bCC(z,-)$ carries it to a morphism of 0-fibrations in $\Cat$ for every $z\in\bCC$.
\end{example}

\begin{remark}\label{rem:cart2cells}
	Let $p\colon x\to c$ be a 0-fibration, and let $a\in\bCC$ be an arbitrary object. According to \cref{ex:fibrationsFunctorCategories}, the induced map $p_\ast\colon\bCC(a,x)\to\bCC(c,p)$ is a cocartesian fibration of 1-categories. We may explicitly describe the cocartesian morphisms of this fibration: given
	$f,g \colon a \to x$ and a 2-morphism $\alpha \colon f \Rightarrow g$, the universal property of $\Free^{0}_c(x)$ supplies us with a 2-morphism $\Xi_{\alpha} : T_f \Rightarrow{} T_g$, where $T_f,T_g \colon a \to \Free_c^{0}(x)$ and $\Xi_{\alpha}$ is induced by the morphism of cones
	\[\begin{tikzcd}
		a && x \\
		& c.
		\arrow[""{name=0, anchor=center, inner sep=0}, curve={height=6pt}, from=1-1, to=1-3]
		\arrow[""{name=1, anchor=center, inner sep=0}, curve={height=-12pt}, from=1-1, to=1-3]
		\arrow["pg"', from=1-1, to=2-2]
		\arrow["p", from=1-3, to=2-2]
		\arrow["\alpha", shorten <=2pt, shorten >=2pt, Rightarrow, from=1, to=0]
	\end{tikzcd}\]
	Unraveling the definitions, we see that $\alpha$ defines a cocartesian morphism for $p_\ast$ precisely if the whiskering of $\Xi_{\alpha}$ with $\ell_p \colon \Free_c^{0}(x) \to x$ is an invertible 2-morphism.
\end{remark}

\begin{remark}\label{rem:1fib}
	There is an evident dual notion of a \emph{1-fibration} in $\bCC$: a map $x\to c$ is a 1-fibration if it is a 0-fibration in $\bCC^\co$, and a square in $\bCC$ is a morphism of 1-fibrations if it is a morphism of 0-fibrations in $\bCC^\co$. Thus, all statements about 0-fibration dualise in the appropriate way to statements about 1-fibrations, so that we may restrict our attention to 0-fibrations. If we do not wish to specify whether we mean $0$- or 1-fibration, we will use the placeholder $\epsilon\in\{0,1\}$.
\end{remark}

\begin{lemma}\label{lem:freeFibrationIsFibration}
	For every map $p\colon x\to c$, the induced map $\Free^{0}_c(p)\colon\Free^{0}_c(x)\to c$ is a 0-fibration. Moreover, if $p$ is already a 0-fibration, the map $l_p\colon \Free^{0}_c(p)\to p$ defines a morphism of 0-fibrations.
\end{lemma}
\begin{proof}
	Note that the second claim is a formal consequence of the first, so it suffices to show the first statement. To that end, note that since $\bCC$ has oriented pullbacks, it is cotensored over $\Cat^\fin$, i.e.\ we have a $2$-functor $(-)^{(-)}\colon(\Cat^\fin)^\op\times\bCC\to\bCC$ that sends a finite $\infty$-category $K$ and an object $c\in\bCC$ to the (fully) lax limit of the constant diagram $K\to\bCC$ with value $c$. Consequently, as the map $\sigma_1\colon [2]\to [1]$ admits a right adjoint $\delta_1\colon [1]\into [2]$, the induced map $c^{[1]}\to c^{[2]}$ must be a right adjoint. Now it follows straightforwardly from the definitions that this map is precisely the unit $\eta_{\Free^{0}_c(\id_c)}\colon\Free^{0}_c(c)\to\Free^{0}_c(\Free^{0}_c(c))$. Hence the first claim follows in the special case where $p$ is the identity on $c$ (see \cref{rem:minimalConditionsFibration}). Finally, in the general case, \cref{rem:freeFibExplicitly} implies that the unit $\eta_{\Free^{0}_c(p)}\colon \Free^{0}_c(x)\to\Free^{0}_c(\Free^{0}_c(x))$ is the pullback of $\eta_{\Free^{0}_c(\id_c)}$ along $p$ (where we regard $l_{\Free^{0}(\id_c)}\dashv \eta_{\Free^{0}(\id_c)}$ as an adjunction internally in $\bCC_{/c}$ via the projections $\Free^{0}(\Free^{0}(\id_c))\to \Free^{0}(\id_c)\to c$) and must therefore be a left inclusion as well.
\end{proof}

\begin{lemma}\label{lem:mappingCategoriesFib}
	Let $p\colon x\to z $ and $q\colon y\to d$ be 0-fibrations. Then the inclusion $\Fib_{\bCC}^0(p,q)\into\AR_{\bCC}(p,q)$ fits into a pullback square
	\begin{equation*}
		\begin{tikzcd}[column sep=huge]
			\Fib_{\bCC}^0(p,q)\arrow[d, hookrightarrow]\arrow[r] & \AR_{\bCC}(p,q)\arrow[d, hookrightarrow, "l_p^\ast"]\\
			\AR_{\bCC}(p,q)\arrow[r, "(l_q)_\ast\circ\Free^{0}(-)"] & \AR_{\bCC}(\Free^{0}_c(p),q).
		\end{tikzcd}
	\end{equation*}
\end{lemma}
\begin{proof}
	Since $l_p$ is a left reflection, the functor $l_p^\ast$ is a right inclusion and therefore in particular fully faithful. Hence the claim is equivalent to the statement that a map $\phi\colon p\to q$ in $\AR_{\bCC}$ defines a morphism of 0-fibrations if and only if the whiskering of the adjunction unit $\id_{\Free^{0}_c(p)}\to \eta_pl_p$ with the composition
	\begin{equation*}
		\Free^{0}_c(p)\xrightarrow{\Free^{0}(\alpha)}\Free^{0}_d(q)\xrightarrow{l_q} q
	\end{equation*}
	is an equivalence. Since this precisely recovers the mate of the naturality square 
	\begin{equation*}
		\begin{tikzcd}
			p\arrow[r, "\alpha"]\arrow[d, "\eta_p"] & q\arrow[d, "\eta_q"]\\
			\Free^{0}_c(p)\arrow[r, "\Free^{0}(\alpha)"] & \Free^{0}_d(q),
		\end{tikzcd}
	\end{equation*}
	the claim follows.
\end{proof}

\begin{proposition}\label{prop:freeFibrationIsFreeFibration}
	The functor $\Free^{0}\colon \AR_{\bCC}\to\AR_{\bCC}$ takes values in $\Fib_{\bCC}^{0}$ and moreover defines a functor $\Free^{0}\colon\AR_{\bCC}\to\Fib_{\bCC}^{0}$ that is a relative left adjoint of the inclusion $\Fib_{\bCC}^{0}\into\AR_{\bCC}$ over $\bCC$.
\end{proposition}
\begin{proof}
	It follows from \cref{lem:freeFibrationIsFibration} that $\Free^{0}$ takes values in $\Fib_{\bCC}^{0}$. Thus, to complete the proof, it suffices to show that for every map $f\colon x\to c$ and every 0-fibration $p\colon z\to d$ the composition
	\begin{equation*}
		\AR_{\bCC}(f, p)\xrightarrow{\Free^{0}(-)} \Fib_{\bCC}^{0}(\Free^{0}(f), \Free^{0}(p))\xrightarrow{(l_p)_{\ast}}\Fib_{\bCC}^{0}(\Free^{0}(f), p)
	\end{equation*}
	is an equivalence. To see this, observe that we have a commutative diagram
	\begin{equation*}
		\begin{tikzcd}
			\AR_{\bCC}(f,p)\arrow[r, "\Free^{0}(-)"]\arrow[dr, "(\eta_p)_\ast"'] & \AR_{\bCC}(\Free^{0}(f),\Free^{0}(p))\arrow[d, "\eta_f^\ast"] \arrow[r, "(l_p)_\ast"] & \AR_{\bCC}(\Free^{0}(f),p)\arrow[d, "\eta_f^\ast"]\\
			& \AR_{\bCC}(f,\Free^{0}(p))\arrow[r, "(l_p)_\ast"] & \AR_{\bCC}(f,p)
		\end{tikzcd}
	\end{equation*}
	which already shows that $\eta_f^\ast$ defines a left inverse of $(l_p)_\ast\Free^{0}(-)$. By the two-out-of-six property of equivalences, it thus suffices to show that the composition $(l_p)_\ast\Free^{0}(-)\eta_f^\ast$ is an equivalence. To that end, note that by making use of \cref{lem:mappingCategoriesFib}, we obtain a commutative diagram
	\begin{equation*}
		\begin{tikzcd}
			\Fib_{\bCC}^0(\Free^{0}(f), p)\arrow[d, hookrightarrow] \arrow[rr] && \AR_{\bCC}(\Free^{0}(f), p)\arrow[d, hookrightarrow, "l_{\Free^{0}(f)}^\ast"]\\
			\AR_{\bCC}(\Free^{0}(f), p)\arrow[d, "\eta_f^\ast"]\arrow[r, "\Free^{0}(-)"] & \AR_{\bCC}(\Free^{0}(\Free^{0}(f)), \Free^{0}(p))\arrow[d, "\Free^{0}(\eta_f)^\ast"] \arrow[r, "(l_p)_\ast"]& \AR_{\bCC}(\Free^{0}(\Free^{0}(f)), p)\arrow[d, "\Free^{0}(\eta_f)^\ast"]\\
			\AR_{\bCC}(f,p)\arrow[r, "\Free^{0}(-)"] & \AR_{\bCC}(\Free^{0}(f), \Free^{0}(p))\arrow[r, "(l_p)_\ast"] & \AR_{\bCC}(\Free^{0}(f), p).
		\end{tikzcd}
	\end{equation*}
	By the previous commutative diagram, the upper horizontal map can be identified with the canonical inclusion $\Fib_{\bCC}^0(\Free^{0}(f), p)\into \AR_{\bCC}(\Free^{0}(p), q)$. Consequently, the claim follows once we show that the composition $l_{\Free^{0}(f)}\Free^{0}(\eta_f)$ is an equivalence. When viewing this map as a commutative square in $\bCC$, the lower horizontal map is the identity on $c$, so that it suffices to show that the upper horizontal map is an equivalence. Now we observe that both $\Free^{0}(\eta_f)\colon \Free^{0}(x)\to\Free^{0}(\Free^{0}(x))$ and $l_{\Free^{0}(f)}\colon\Free^{0}(\Free^{0}(x))\to\Free^{0}(x)$ are obtained as the pullback of $\Free^{0}(\eta_{\id_c})\colon c^{[1]}\to c^{[2]}$ and $l_{\Free^{0}(\id_c)}\colon c^{[2]}\to c^{[1]}$ along $f\colon x\to c$ (where we regard $\Free^{0}(\eta_{\id_c})$ and $l_{\Free^{0}(\id_c)}$ as maps in $\bCC_{/c}$ via the projections $c^{[2]}\to c$ and $c^{[1]}\to c$ which are induced by the inclusions of the initial vertex in $[2]$ and $[1]$, respectively). Consequently, we may assume without loss of generality that $f=\id_c$. In this case, the claim follows from the observation that $\Free^{0}(\eta_{\id_c})$ can be identified with $s_0\colon c^{[1]}\to c^{[2]}$ and $l_{\Free^{0}(\id_c)}$ with $d_1\colon c^{[2]}\to c^{[1]}$. 
\end{proof}

\begin{lemma}\label{lem:pullbackFibration}
	Let $p\colon x\to c$ be a 0-fibration in $\bCC$, and let
	\begin{equation*}
		\begin{tikzcd}
			y\arrow[d, "q"]\arrow[r] & x\arrow[d, "p"]\\
			d\arrow[r] & c
		\end{tikzcd}
	\end{equation*}
	be a pullback square in $\bCC$. Then $q$ is a 0-fibration, and the square defines a morphism of 0-fibrations.
\end{lemma}
\begin{proof}
	To begin with, note that when viewed as a morphism in $\AR_{\bCC}$, the pullback square in the statement of the lemma is itself a pullback of the square
	\begin{equation*}
		\begin{tikzcd}
			x\arrow[d, "p"]\arrow[r, "\id"] & x\arrow[d, "p"]\\
			c\arrow[r, "\id"] & c.
		\end{tikzcd}
	\end{equation*}
	As the latter is trivially a morphism of 0-fibrations, \cref{rem:morphismsOfFibrations} implies (by implicitly replacing $\bCC$ with $\AR_{\bCC}$) that the entire claim already follows once we verify that $q$ is a 0-fibration. To that end, note that there is a commutative diagram
	\begin{equation*}
		\begin{tikzcd}
			& y\arrow[rr]\arrow[dd] \arrow[dl, "\eta_y"]&& x\arrow[dd]\arrow[dl, "\eta_x"]\\
			\Free^{0}_d(y)\arrow[rr, crossing over]\arrow[dd] && \Free^{0}_c(x)\\
			& d\arrow[rr]\arrow[dl, "s_0"] && c\arrow[dl, "s_0"] \\
			d^{[1]}\arrow[rr] && c^{[1]}\arrow[from=uu, crossing over]
		\end{tikzcd}
	\end{equation*}
	in which both the front and the back square are pullbacks. Furthermore, observe that both $s_0\colon d\to d^{[1]}$ and $s_0\colon c\to c^{[1]}$ admit left adjoints given by $d_0\colon d^{[1]}\to d$ and $d_0\colon c^{[1]}\to c$, respectively, so that the bottom square lifts to a morphism in $\Fun(\Adj,\bCC)$. Also, since the left adjoint $l_x$ of $\eta_x$ defines a left adjoint in $\bCC_{/c}$, the right square lifts to a morphism in $\Fun(\Adj, \bCC)$ as well, so that we end up with a cospan in $\Fun(\Adj, \bCC)$. Consequently, by taking the pullback of this cospan and by using that the forgetful functor $\Fun(\Adj, \bCC)\to \AR_{\bCC}$ preserves pullbacks, we conclude that the map $\eta_y$ lifts to an object in $\Fun(\Adj, \bCC)$ and therefore admits a left adjoint, as desired.
\end{proof}

\begin{proposition}\label{prop:fibrations2functorial}
	The functor $\ev_1\colon\Fib_{\bCC}^{0}\to\bCC$ is a $(1,0)$-fibration of 2-categories. Dually, the functor $\ev_1\colon\Fib_{\bCC}^{1}\to\bCC$ is a $(1,1)$-fibration of 2-categories.
\end{proposition}
\begin{proof}
	We deal without loss of generality with the case of 0-fibrations. To begin with, observe the right adjoint $s_0\colon\bCC\to \AR_{\bCC}$ of $\ev_1\colon \AR_{\bCC}\to\bCC$ restricts to a right adjoint of $\ev_1\colon \Fib_{\bCC}^{0}\to \bCC$. In fact, $s_0$ acts by carrying $c\in\bCC$ to $\id_c$, which is trivially a 0-fibration, and the adjunction unit $\eta\colon\id_{\AR_{\bCC}}\to s_0\ev_1$ carries a map $p\colon x\to c$ to the square
	\begin{equation*}
		\begin{tikzcd}
			x\arrow[r, "p"]\arrow[d, "p"] & c\arrow[d, "\id"]\\
			c\arrow[r, "\id"] & c,
		\end{tikzcd}
	\end{equation*}
	which is easily seen to be a morphism of 0-fibrations whenever $p$ is a 0-fibration. As a consequence, the map $\Fib_{\bCC}^{0}(-,-)\to \bCC(\ev_1(-), \ev_1(-))$ that is obtained by the action of $\ev_1$ on mapping categories can be identified with $\eta_\ast\colon \Fib_{\bCC}^{0}(-,-)\to \Fib_{\bCC}^{0}(-, s_0\ev_1(-))$. Now if $p\colon x\to c$ is a 0-fibration, $d\to c$ is a map in $\bCC$ and $q\colon y\to d$ is the pullback of $p$ along this map, we know from \cref{lem:pullbackFibration} that the associated pullback square defines a map $\sigma\colon q\to p$ in $\Fib_{\bCC}^{0}$. Furthermore, the fact that limits in $\Fib_{\bCC}^{0}$ are computed in $\AR_{\bCC}$ (by \cref{prop:freeFibrationIsFreeFibration}) and the explicit description of the adjunction unit $\eta\colon \id_{\AR_{\bCC}}\to s_0\ev_1$ imply that the naturality square
	\begin{equation*}
		\begin{tikzcd}[column sep=large]
			q\arrow[r, "\sigma"]\arrow[d] & p\arrow[d]\\
			s_0\ev_1(q)\arrow[r, "s_0\ev_1(\sigma)"] & s_0\ev_1(p)
		\end{tikzcd}
	\end{equation*}
	is a pullback in $\Fib_{\bCC}^{0}$. As a consequence, the induced square
	\begin{equation*}
		\begin{tikzcd}
			\Fib_{\bCC}^{0}(-, q)\arrow[r]\arrow[d] & \Fib_{\bCC}^{0}(-, p)\arrow[d]\\
			\Fib_{\bCC}^{0}(-, s_0\ev_1(q))\arrow[r] & \Fib_{\bCC}^{0}(-, s_0\ev_1(q))
		\end{tikzcd}
	\end{equation*}
	must be a pullback as well, so that $\sigma\colon q\to p$ is $\ev_1$-cartesian.
	
	To complete the proof, we need to show that $\Fib_{\bCC}^{0}(-,-)\to \Fib_{\bCC}^{0}(-, s_0\ev_1(-))$ takes values in $\Fib_{\Cat}^{0}$ when viewed as a functor $(\Fib_{\bCC}^{0})^\op\times\Fib_{\bCC}^{0}\to\AR_{\Cat}$. By $2$-functoriality of $\Fib_{\bCC}^0(-,-)$ and \cref{ex:fibrationsFunctorCategories}, this amounts to showing that if $p\colon x\to c$ and $q\colon y\to d$ are 0-fibrations and $p\to q$ is a map in $\Fib_{\bCC}^0$, the associated commutative square
	\begin{equation*}
		\begin{tikzcd}
			p\arrow[d]\arrow[r] & q\arrow[d]\\
			s_0\ev_1(p)\arrow[r] & s_0\ev_1(q)
		\end{tikzcd}
	\end{equation*}
	is a morphism of 0-fibrations in $\Fib_{\bCC}^0$ (so that in particular the vertical maps are 0-fibrations). Since the inclusion $\Fib_{\bCC}^0\into\AR_{\bCC}$ is a right adjoint and by again making use of \cref{ex:fibrationsFunctorCategories}, this explicitly means that (1) the two vertical maps in the above square encode morphisms of 0-fibrations in $\bCC$, and that (2) the image of the above square along the two functors $\ev_0,\ev_1\colon \AR_{\bCC}^0\rightrightarrows \bCC$ yields morphisms of 0-fibrations in $\bCC$. In other words, we need to show that in the commutative cube
	\begin{equation*}
		\begin{tikzcd}
			& c\arrow[rr]\arrow[dd, "\id", near end] && d\arrow[dd, "\id", near end]\\
			x\arrow[ur, "p"]\arrow[rr, crossing over]\arrow[dd, "p", near start] && y\arrow[ur, "q"]&\\
			& c\arrow[rr] && d\\
			c\arrow[ur, "\id"]\arrow[rr] && d\arrow[from=uu, crossing over, "q", near start]\arrow[ur, "\id"] &
		\end{tikzcd}
	\end{equation*}
	in $\bCC$, the four vertical faces are morphisms of 0-fibrations. Since this is evidently the case, the claim follows.
\end{proof}

%% file: fibdescent.tex
\section{Fibrational descent and 2-topoi}\label{sec:fibdescent}
 This chapter constitutes the core the document. We start by giving the axiom of \emph{fibrational descent} which characterizes the class of 2-topoi among all presentable 2-categories. The key idea behind this axiom is that it provides a local-to-global principle for $\epsilon$-fibrations, akin to how the descent axiom in 1-topos theory provides a local-to-global principle for \emph{all} maps in a 1-topos. We make this precise in \cref{subsec:fibDescent}. In \cref{subsec:const}, we illustrate basic constructions of 2-topoi which provide a wide array of examples, including 2-categories of categorical presheaves. In the final two sections of this chapter, we explore different characterizations of 2-topoi, first in terms of a categorified variant of the Lawvere-Tierney axioms (\cref{subsec:LTaxioms}), and second in terms of certain Bousfield localizations of categorical presheaves (\cref{subsec:giraud}).
 
 \subsection{Fibrational descent}\label{subsec:fibDescent}
 In this section, we introduce and discuss our main axiom of \emph{fibrational descent}. As mentioned above, this axiom aims at providing a local-to-global principle for $\epsilon$-fibrations in a presentable 2-category $\bCC$. We will phrase this axiom using a $2$-dimensional version of \emph{cartesian transformations} of diagrams:
 
\begin{definition}\label{def:Ecarttransf}
	Let $\bCC$ be a presentable 2-category and let $(\bII,E)$ be a marked 2-category. 
	Given functors $p,q\colon \bII \to \bCC$ and $\epsilon\in\{0,1\}$, we say that a strong natural transformation $p\Rightarrow q$ is 
	$(E,\epsilon)$-cartesian if the following conditions are satisfied:
	\begin{enumerate}
		\item For every object $i \in \bII$, the morphism $p(i) \to q(i)$ is an $\epsilon$-fibration (see \cref{def:fibrations} and \cref{rem:1fib}).
		\item For every morphism $u\colon i \to j$ in $\bII$, the corresponding commutative square
		 \[
		  	\begin{tikzcd}
		  		p(i) \arrow[d] \arrow[r] & p(j) \arrow[d] \\\
		  		q(i) \arrow[r] & q(j)
		  	\end{tikzcd}
		 \] 
		 defines a 1-morphism in $\Fib_{\bCC}^{\epsilon}$. Moreover, if $u$ belongs
		  to $E$ we require this diagram to be a pullback square.
		\item  For every 2-morphism $u \xRightarrow{} v$ in $\bII$,
		 the corresponding 2-dimensional diagram
        \[\begin{tikzcd}
	     {p(i)} & {p(j)} & {} \\
	     {q(i)} & {q(j)}
	    \arrow[""{name=0, anchor=center, inner sep=0}, curve = {height=8pt}, from=1-1, to=1-2]
	    \arrow[""{name=1, anchor=center, inner sep=0}, curve={height=-8pt}, from=1-1, to=1-2]
	    \arrow[from=1-1, to=2-1]
	    \arrow[from=1-2, to=2-2]
	    \arrow[""{name=2, anchor=center, inner sep=0}, curve={height=-8pt}, from=2-1, to=2-2]
		\arrow[""{name=3, anchor=center, inner sep=0}, curve={height=8pt}, from=2-1, to=2-2]
		\arrow[shorten <=2pt, shorten >=2pt, Rightarrow, from=1, to=0]
		\arrow[shorten <=2pt, shorten >=2pt, Rightarrow, from=2, to=3]
       \end{tikzcd}\]
       defines an $\epsilon$-cartesian 2-morphism in $\Fib_{\bCC}^{\epsilon}$.
	\end{enumerate}
	\end{definition}

	\begin{remark}
		Let $\bCC$ be a presentable 2-category and let $(\bII,E)$ be a marked 2-category. Then it follows that we can identify $(E,\epsilon)$-cartesian transformations with functors $F \colon \bII \to \Fib_{\bCC}^{\epsilon}$ which send marked edges in $\bII$ to $1$-cartesian edges in $\Fib_{\bCC}^{\epsilon}$ and send every 2-morphisms in $\bII$ to $\epsilon$-cartesian 2-morphisms.
	\end{remark}

	\begin{definition}
		Let $\bCC$ be a presentable 2-category and let $(\bII,E)$ be a marked 2-category. We define $\eCart(\bII,\Fib_{\bCC}^{\epsilon})$ as the full sub-2-category  of $\FUN(\bII,\Fib_{\bCC}^{\epsilon})$ spanned by the $(E,\epsilon)$-cartesian transformations.
	\end{definition}

Given $\epsilon \in \{0,1\}$ we shall write “$\epsilon$lax” to mean either “lax” if $\epsilon=0$ or “oplax” if $\epsilon=1$. Similarly, we will write “$\epsilon\op$” to mean either “$\op$” if “$\epsilon=0$” or “$\coop$” if $\epsilon=1$. We may now define fibrational descent as follows:
\begin{definition}\label{def:fibdescent}
		Let $\bCC$ be a presentable 2-category. We say that $\bCC$ satisfies $\epsilon$-\emph{fibrational descent}
		if for every small marked 2-category $(\bII,E)$, the following conditions hold:
		\begin{itemize}
			\item[F1)] Let $\overline{\alpha}\colon\overline p \xRightarrow{} \overline q$ be a natural transformation
			   between functors
			 $\overline{p},\overline{q}\colon \bII^{\colimcone}_{\epsilonlax} \to \bCC$ (see \cref{rem:conesimplified}) such that the restriction $\alpha\colon p\Rightarrow q$ of $\overline{\alpha}$ along the inclusion $\bII\into \bII^\colimcone_{\epsilonlax}$ is $(E,\epsilon)$-cartesian and such that $\overline{q}$ is an $\Eepsilonlax$ colimit cone.
			 Then
			 the following are equivalent:
			 \begin{enumerate}
			 	\item The functor $\overline{p}$ defines an $\Eepsilonlax$  colimit cone.
			 	\item The natural transformation $\overline{p} \xRightarrow{}\overline{q}$ is 
			 	$(E^{\colimcone}_{\epsilonlax},\epsilon)$-cartesian.
			 \end{enumerate}
			 \item[F2)] Let $\overline\alpha\colon \overline{p} \Rightarrow \overline{k}$ and $\overline{\beta}\colon \overline{k} \Rightarrow \overline{q}$ be 
			 natural transformations between $\epsilonlax$ colimit cones $\overline{p},\overline{k},\overline{q}\colon \bII^\colimcone_{\epsilonlax}\to \bCC$ such that the restrictions $\alpha \circ \beta$ and $\beta$ of $\overline{\alpha}\circ\overline{\beta}$ and $\overline{\beta}$, respectively, along the inclusion $\bII\into\bII^\colimcone_{\epsilonlax}$ are $(E,\epsilon)$-cartesian. 
			 Then $\alpha$ is point-wise a map of $\epsilon$-fibrations if and only if $\overline{\alpha}$ is.
		\end{itemize}
	If $\bCC$ satisfies $\epsilon$-fibrational descent for both $\epsilon=0$ and $\epsilon=1$, we shall simply say that $\bCC$ satisfies \emph{fibrational descent}.
	\end{definition}

	\begin{construction}\label{const:adjointablesq}
		Let $\bCC$ be a presentable 2-category and let $(\bII,E)$ be a small marked 2-category. We consider a commutative diagram
		\[
			\begin{tikzcd}
				\Cart_{E^{\colimcone}_{\epsilonlax}}(\bII^{\colimcone}_{\epsilonlax},\Fib_{\bCC}^{\epsilon}) \arrow[d, hookrightarrow] \arrow[r,"\res^{\epsilon}"] & \eCart(\bII,\Fib_{\bCC}^{\epsilon}) \arrow[d, hookrightarrow] \\
				\FUN(\bII^{\colimcone}_{\epsilonlax},\bCC^{{[}1{]}}) \arrow[d, "(\ev_1)_\ast"] \arrow[r] & \FUN(\bII,\bCC^{{[}1{]}}) \arrow[d, "(\ev_1)_\ast"] \\
				\FUN(\bII^{\colimcone}_{\epsilonlax},\bCC) \arrow[r] & \FUN(\bII,\bCC)
			\end{tikzcd}
		\]
		where the horizontal maps are induced by restriction along $\bII \into \bII^{\colimcone}_{\epsilonlax}$ and where $\ev_1\colon\bCC^{[1]}\to \bCC$ is evaluation at $1\in[1]$. We note that the horizontal morphisms in the bottom square are right adjoints, with corresponding left adjoints given by left Kan extension. Since $\ev_1$ preserves (co)limits, we conclude that the bottom square is horizontally left adjointable.
		
		Now let $q\colon \bII \to \bCC$ be a diagram, and let $\overline{q}\colon \bII^{\colimcone}_{\epsilonlax} \to \bCC$ be the associated $\epsilonlax$ colimit cone. Let us furthermore set $c=q(\ast)$. Then, by passing to the fibres of the vertical maps in the above diagram over $q$ and $\overline{q}$, respectively, we obtain an induced functor
		\[
			\Cart_{E^{\colimcone}_{\epsilonlax}}(\bII^{\colimcone}_{\epsilonlax},\Fib_{\bCC}^{\epsilon})_{\overline{q}} \to \eCart(\bII,\Fib_{\bCC}^{\epsilon})_{q},
		\]
		which by using \cref{rem:cat2complete} we may identify with the canonical map
		\[
			\res^{\epsilon}_q\colon \Fib^{\epsilon}_{/c} \to \lim^{\eoplax}_{\bII^{\epsilon\op}} \Fib^{\epsilon}_{/q(-)}.
		\]
	\end{construction}

	\begin{lemma}\label{lem:glue}
		Let $\bCC$ be a presentable 2-category satisfying $\epsilon$-fibrational descent and let $(\bII,E)$ be a marked 2-category. Then the functor $\res^\epsilon$ from \cref{const:adjointablesq} admits a left adjoint
		\begin{equation*}
			\glue^\epsilon\colon \Cart_{E}(\bII, \Fib_{\bCC}^\epsilon)\to \Cart_{E^{\colimcone}_{\epsilonlax}}(\bII_{\epsilonlax}^\triangleright, \Fib_{\bCC}^\epsilon),
		\end{equation*} 
		and the top square from \cref{cons:freefib} is horizontally left adjointable. 
	\end{lemma}
	\begin{proof}
		It suffices to show that the left adjoint of the restriction functor
		\begin{equation*}
			\FUN(\bII_{\epsilonlax}^\triangleright, \bCC^{[1]})\to\FUN(\bII, \bCC^{[1]})
		\end{equation*}
		restricts to a left adjoint of $\res^\epsilon$. Given $f \in \eCart(\bII,\Fib_{\bCC}^{\epsilon})$, we consider the composite $f \colon\bII \to \Fib_{\bCC}^{\epsilon} \to \bCC^{[1]}$ and let $\overline{f}\colon\bII^{\colimcone}_{\epsilonlax} \to \bCC^{[1]}$ be the associated $\epsilon$lax colimit cone. Using axiom F1) in the definition of $\epsilon$-fibrational descent, we see that $\overline{f}$ factors through $\Fib_{\bCC}^{\epsilon}$ and that it defines an $(E^{\colimcone}_{\epsilonlax},\epsilon)$-cartesian transformation.  A similar argument using axiom F2) shows the claim for morphisms $f \to g$ in $ \eCart(\bII,\Fib_{\bCC}^{\epsilon})$.
	\end{proof}

	\begin{definition}
		We call the left adjoint
		\[
		 	\glue^{\epsilon} \colon \eCart(\bII,\Fib_{\bCC}^{\epsilon}) \to \Cart_{E^{\colimcone}_{\epsilonlax}}(\bII^{\colimcone}_{\epsilonlax},\Fib_{\bCC}^{\epsilon})
		 \] 
		 from \cref{lem:glue} 
		 the \emph{gluing functor}.
	\end{definition}

	\begin{proposition}\label{prop:fibrepresents}
		Let $\bCC$ be a presentable 2-category satisfying $\epsilon$-fibrational descent. Let $(\bII,E)$ be a marked 2-category and consider a functor $q\colon \bII \to \bCC$ with $\epsilonlax$ colimit given by $\overline{q}\colon \bII^{\colimcone}_{\epsilonlax} \to \bCC$. Let us set $c=q(\ast)$. Then the adjunction $\glue^{\epsilon} \dashv \res^{\epsilon}$ becomes an equivalence upon passage to fibres
		\[
			\glue_{q}^{\epsilon}\colon \lim^{\eoplax}_{\bII^{\epsilon\op}}\Fib^{\epsilon}_{/q(-)}   \llra  \Fib^{\epsilon}_{/c} \colon \res^{\epsilon}_q .
		\]
	\end{proposition}
	\begin{proof}
		Since the restriction functor $\bII \to \bII^{\colimcone}_{\epsilonlax}$ is fully faithful, it follows from \cref{cor:existenceLKE} that $\glue_q^{\epsilon}$ is fully faithful. To finish the proof, we need to show that the counit $\glue_{q}^{\epsilon}\res_q^{\epsilon} \to \id$ is an equivalence. Given any $p\colon x \to c$ in $\Fib^{\epsilon}_{/c} \simeq \Cart_{E^{\colimcone}_{\epsilonlax}}(\bII^{\colimcone}_{\epsilonlax},\Fib_{\bCC}^{\epsilon})_{\overline{q}}$, F1) implies that $p$ is automatically a colimit cone. Unpacking the definitions, we see that the map 
		\[
			\glue_{q}^{\epsilon}\res_q^{\epsilon}(p \colon x \to c) \to (p \colon x \to c)
		\]
		is a morphism of colimit cones induced by a natural transformation between the corresponding diagrams which is a point-wise equivalence. The claim now follows.
	\end{proof}

	\begin{definition}\label{def:2topoimorph}
		We say that a presentable 2-category $\bCC$ is a \emph{2-topos } if it satifies fibrational descent. A morphism of 2-topoi $L \colon \bCC \to \bDD$ is defined to be a left adjoint satisfying the following properties:
		\begin{enumerate}
			\item The functor $L$ preserves oriented pullbacks (cf.\ \cref{subsec:oriented}).
			\item The functor $L$ preserves the terminal object. 
		\end{enumerate}
		We denote the corresponding locally-full sub-2-category of $\CCAT$ by $\LTTop$.
	\end{definition}

    \begin{remark}\label{rem:dubucjoyal}
    	Our definition of morphism of 2-topoi is perhaps surprising since we are not requiring $L$ to preserve all finite partially (op)lax limits, as was expected to be the case \cite[pg. 269]{sigmalim}. This choice will be justified by our 2-categorical version of Giraud's theorem (see \cref{thm:giraud}) together with \cref{ex:idempotentCompleteCategories} and \cref{rem:pullbackidem}, where we give a fundamental example of a morphism of 2-topoi which does not preserve strong pullbacks.
    \end{remark}
    
\subsection{Constructions of 2-topoi}\label{subsec:const}
In this section, we provide certain stability results for the class of 2-topoi that can be used to construct a wealth of examples. We begin by establishing that what out to be the archetypical example of a 2-topos, the 2-category $\Cat$ of 1-categories, is indeed a 2-topos:

\begin{theorem}\label{thm:cat2topos}
	The 2-category $\Cat$ of 1-categories is a 2-topos.
\end{theorem}
\begin{proof}
	It is clear that $\Cat$ is 2-presentable, so we only need to show that $\Cat$ satisfies $\epsilon$-fibrational descent for $\epsilon \in \{0,1\}$.
	To begin with, recall that $0$-fibrations in $\Cat$ are precisely cocartesian fibrations and $1$-fibrations correspond to cartesian fibrations. We will only verify the case of $0$-fibrational descent since the remaining case is formally dual.

	Let $\alpha \colon p \to q$ be an $E$-cartesian transformation between functors $p,q \colon \bII \to \Cat$ (where $(\bII,E)$ is a small marked 2-category) which we view as a map of $(0,1)$-fibrations $\pi_{\alpha}\colon \Un_{\bII}(p) \to \Un_{\bII}(q)$ over $\bII$. It follows from \cref{prop:carttransun} that $\pi_{\alpha}$ is a $(0,1)$-fibration with corresponding functor $F \colon \Un_{\bII}(q) \to \Cat_2$. Let us remind the reader that an edge is 0-cartesian for $\pi_{\alpha}$ if it can be written as a composition $v \circ u$, where $u$ is cocartesian with respect to the projection to $\bII$ and $v$ lies over the identity on some $i\in\bII$ and is cocartesian with respect to the fibration $p(i) \to q(i)$.
	
	We now establish some properties of $F$:
	\begin{enumerate}
		\item Let $(i,x \in q(i)) \in  \Un_{\bII}(q)$. Then it follows that $F(i)= p(i) \times_{q(i)} \{x\}$, which is a 1-category by assumption.
		\item Let $\nu \colon f \to g$ be a 2-morphism in $\Un_{\bII}(q)$ between $f,g\colon a\rightrightarrows b$. We claim that $F(\nu)$ is invertible. To see this, let us set $a=(i,x \in q(i))$ and $b=(j,y \in q(j))$ and let $f=(u \colon i \to j,f_u \colon q(u)(x) \to y)$ and similarly $g=(v \colon i \to j,f_v \colon q(v)(x) \to y)$. We identify $\nu$ with the data of a 2-morphism $\overline{\nu} \colon u \to v$ and a commutative diagram 
		\[\begin{tikzcd}
			{q(u)(x)} & {y} \\
			{q(v)(x)}.
			\arrow["{f_u}", from=1-1, to=1-2]
			\arrow["{q(\overline{\nu})}"', from=1-1, to=2-1]
			\arrow["{f_v}"', from=2-1, to=1-2]
		\end{tikzcd}\]
		
		To prove the claim, it suffices to show that the components of the natural transformation $F(\nu) \colon F(f) \to F(g)$ between the functors $F(f),F(g) \colon F(i)=p(i)\times_{q(i)} \{x\} \to p(j) \times_{q(j)} \{y\}=F(j)$ are invertible. For $\ell \in F(i)$ it follows that we have a commutative diagram in $p(j)$ of the form
		\[
		\begin{tikzcd}
			p(u)(l) \arrow[d,swap,"p(\overline{\nu})"] \arrow[r] & \hat{\ell}_u \arrow[d] \\
			p(v)(l) \arrow[r] & \hat{\ell}_v
		\end{tikzcd}
		\]
		where the horizontal morphisms are cartesian morphisms for the map $p(j) \to q(j)$ and lie over $f_u$ and $f_v$ respectively. Since $\alpha$ is $E$-cartesian it follows that $p(\overline{\nu})$ must be a natural transformation which is point-wise given by cocartesian edges (cf.\ \cref{rem:cart2cells}). We conclude that $\hat{\ell}_u \to \hat{\ell}_v$ is a cartesian morphism over the identity morphism on $y$ and is therefore an equivalence.

		\item Let $\varphi$ be a 1-morphism in $ \Un_{\bII}(q)$ which lies over an edge in $E$ and is $0$-cartesian with respect to the projection $ \Un_{\bII}(q) \to \bII$. We claim that $F(\varphi)$ is invertible in $\Cat$. To see this, let $\varphi:(i,a_i) \to (j,b_j)$ be given by $\varphi=(u \colon i \to j, q(u)(a_i) \simeq b_j)$. We observe that we have a pullback diagram
		\[
		\begin{tikzcd}
			p(i) \arrow[r,"p(u)"] \arrow[d] & p(j) \arrow[d] \\
			q(i) \arrow[r,"q(u)"] & q(j),
		\end{tikzcd}
		\]
		and we furthermore observe that $F(\varphi)$ is identified with the induced map on fibres
		\[
		p(i)\times_{q(i)}\{a_i\} \xrightarrow{\simeq} p(j)\times_{q(j)}\{p(u)(a_i)\}
		\]
		which is an equivalence since the diagram is pullback.
	\end{enumerate}
	The claims (1)-(3) guarantee that we have a factorization
	\begin{equation*}
		\begin{tikzcd}
			\Un_{\bII}(q)\arrow[r, "\lambda"] \arrow[rr, bend left, "F"] & L_{E}(\Un_{\bII}(q))\arrow[r, "\hat F"] & \Cat
		\end{tikzcd}
	\end{equation*}
	where $\lambda$ denotes the localization of $\Un_{\bII}(q)$ at every 2-morphism and at the $0$-cartesian morphisms lying over $E$. Let $\rho_\alpha\colon X\to L_E(\Un_{\bII}(q))$ be the unstraightening of $\hat F$. We now have a pullback diagram
	\[
	\begin{tikzcd}
		\Un_{\bII}(p) \arrow[d, "\pi_{\alpha}"] \arrow[r] &  X \arrow[d,"\rho_{\alpha}"] \\
		\Un_{\bII}(q) \arrow[r, "\gamma"] & L_{E}(\Un_{\bII}(q))
	\end{tikzcd}
	\]
	where the vertical maps are $(0,1)$-fibrations. Following \cref{rem:cat2complete}, we see that $L_{E}(\Un_{\bII}(q))$ is the $E$-lax colimit of $q$. We further observe that $X$ is a 1-category (since $L_{E}(\Un_{\bII}(q)$ is a 1-category) and that by construction we have a factorization
	\[
		\Un_{\bII}(p) \to L_{E}(\Un_{\bII}(p)) \xrightarrow{\varphi} X.
	\]
	We are now ready to prove the main statement:
	\begin{itemize}
		\item Note that $X$ is the fully lax colimit of $\hat F\colon \L_E(\Un_{\bII}(q))\to\Cat$.
		Invoking the fact that $\gamma$ is a localization and therefore a final map of 2-categories (\cref{prop:locfinal}) and that $L_{E}(\Un_{\bII}(p))$ is the $\overline{E}$-$\lax$ colimit of $F \colon \Un_{\bII}(q) \to \Cat$, where $\overline{E}$ is the collection of $0$-cartesian edges lying over $E$, we conclude that $\varphi$ must be an equivalence. To see that we have produced an $E^{\colimcone}$-cartesian transformation we note that the outer square in the diagram
		\[
		\begin{tikzcd}
			p(i) \arrow[r] \arrow[d] & \Un_{\bII}(p) \arrow[d] \arrow[r] &   L_{E}(\Un_{\bII}(p))\simeq X \arrow[d] \\
			q(i) \arrow[r] & \Un_{\bII}(q) \arrow[r] & L_{E}(\Un_{\bII}(q))
		\end{tikzcd}
		\]
		is a pullback square since both inner squares are.

		\item Given a commutative diagram
		\[
		\begin{tikzcd}
			\Un_{\bII}(p) \arrow[d] \arrow[r] &  L \arrow[d,"\rho"] \\
			\Un_{\bII}(q) \arrow[r] & L_{E}(\Un_{\bII}(q))
		\end{tikzcd}
		\]
		such that $\rho$ is a $(0,1)$-fibration and such that for every $i \in \bII$ the outer square
		\[
		\begin{tikzcd}
			p(i) \arrow[r] \arrow[d] & \Un_{\bII}(p) \arrow[d] \arrow[r] &  L \arrow[d] \\
			q(i) \arrow[r] & \Un_{\bII}(q) \arrow[r] & L_{E}(\Un_{\bII}(q))
		\end{tikzcd}
		\]
		is a pullback, we claim that $L$ must be the $E$-lax colimit of $p$. Using a finality argument as before, this amounts to showing that the right square is a pullback square. This follows easily from our assumptions since the vertical morphisms are fibrations which allows us to use the fibre-wise criterion given in \cite[Proposition 3.2.11]{AGH24}.
		\item Given a map $\Un_{\bII}(p) \to \Un_{\bII}(\ell) \to \Un_{\bII}(q)$ associated to a map of $E$-cartesian fibrations as in F2) in \cref{def:fibdescent}, we wish to show that the the corresponding diagram
		\[
		\begin{tikzcd}
			L_{E}(\Un_{\bII}(p)) \arrow[rr] \arrow[dr] & & L_{E}(\Un_{\bII}(\ell)) \arrow[dl] \\
			& L_{E}(\Un_{\bII}(q)) &
		\end{tikzcd}
		\]
		is a map of fibrations. To see this, we note that we have a map of fibrations over $\Un_{\bII}(q)$ before localizing with amounts to a natural transformation $F \to G$ between functors $\Un_{\bII}(q) \to \Cat$. It is clear that this natural transformation descends to the localization which proves our claim.
	\end{itemize}
	Consequently, $\Cat$ satisfies the conditions of \cref{def:fibdescent} and is therefore a 2-topos.
\end{proof}

\begin{proposition}\label{prop:fun2topos}
	Let $\bCC$ be a 2-topos and let $\bSS$ be a small 2-category. Then $\FUN(\bSS,\bCC)$ is again a 2-topos.
\end{proposition}
\begin{proof}
	First, we observe that it follows from \cref{prop:functorCatPresentable} that $\bXX=\Fun(\bSS,\bCC)$ is again 2-presentable. Before we proceed to show that $\bXX$ satisfies $\epsilon$-fibrational descent,we recall (cf.\ \cref{ex:fibrationsFunctorCategories}) that a map $X\to Y$ in $\bXX$ is an $\epsilon$-fibration if it is point-wise an $\epsilon$-fibration in $\bCC$ and for every $s \to t$ the associated diagram
	\[
	\begin{tikzcd}
		X_s \arrow[r] \arrow[d] & X_t \arrow[d] \\
		Y_s \arrow[r] & Y_t
	\end{tikzcd}
	\]  
	defines a morphism in $\Fib_{\bCC}^{\epsilon}$. Moreover, a map $X \to Y$ over some $S\in\bXX$ defines a map of $\epsilon$-fibrations over $S$ if and only if its point-wise a map of fibrations.
	
	Let $(\bII,E)$ be a marked 2-category and consider a diagram $q \colon \bII \to \bXX$ with associated $\Eepsilonlax$ colimit cone $\overline{q} \colon\bII^{\colimcone}_{\epsilonlax} \to \bXX$. Suppose that we are given an $(E,\epsilon)$-cartesian transformation $\alpha \colon p \to q$ and an extension $\overline{\alpha}\colon \overline{p}\to \overline{q}$. 
	If $\overline{\alpha}$ is $(E^{\colimcone}_{\epsilonlax},\epsilon)$-cartesian, then it must also be so point-wise. Since colimits in presheaf 2-categories are computed point-wise and since $\bCC$ is by assumption a 2-topos, we conclude that $\overline{p}$ is an $\Eepsilonlax$ colimit cone. 

	Conversely, let us suppose that $\overline{p}$ is an $\Eepsilonlax$ colimit cone and let us show that $\overline{\alpha}$ is $(E^{\colimcone}_{\epsilonlax},\epsilon)$-cartesian. Again, as (co)limits are computing point-wise, it follows that $\overline{\alpha}$ is point-wise given by a morphism of $E$-$\epsilon$lax colimit cones and therefore point-wise $(E^{\colimcone}_{\epsilonlax},\epsilon)$-cartesian. Thus, we see that for every map $i \to *$ in $\bII^{\colimcone}_{\epsilonlax}$ we have a pullback square in $\bXX$ of the form
	\[
	\begin{tikzcd}
		p(i) \arrow[r] \arrow[d]  & p(\ast) \arrow[d] \\
		q(i) \arrow[r] & q(\ast).
	\end{tikzcd}
	\]
	Similarly, since $\epsilon$-cartesian 2-morphisms can be detected point-wise (cf.\ \cref{rem:cart2cells}), it follows that $\overline{p}$ sends each 2-morphism in $\bII^{\colimcone}_{\epsilonlax}$ to an $\epsilon$-cartesian 2-morphism in $\Fib_{\bXX}^{\epsilon}$. Thus, to conclude that $\overline{\alpha}$ is $(E^{\colimcone}_{\epsilonlax},\epsilon)$-cartesian, we only need to show that $p(\ast)\to q(\ast)$ is an $\epsilon$-fibration, which in turn amounts to showing that given $s \to t$ in $\bSS$ the commutative square
	\[
	\begin{tikzcd}
		p(\ast)_s \arrow[r] \arrow[d]  & p(\ast)_t \arrow[d] \\
		q(\ast)_s \arrow[r] & q(\ast)_t
	\end{tikzcd}
	\]
	defines a morphism in $\Fib_{\bCC}^{\epsilon}$. 
	Note that it suffices to show that the induced map $p(\ast)_s \to p(\ast)_t \times_{q(\ast)_t}q(\ast)_s$ is a map of fibrations over $q(\ast)_s$. Using axiom F2) in \cref{def:fibdescent}, this follows once we show that this map becomes a morphism of fibrations once pulled back along $q(i)_{s} \to q(\ast)_s$ for each $i \in \bII$. Using naturality, we can identify the latter with the map $p(i)_s \to p(i)_t \times_{q(i)_t}q(i)_s$, which is indeed a map of fibrations.
	
	Finally, to show that F2) holds, we note that this condition requires a certain 2-morphism in $\bXX$ to be invertible, which can again be checked point-wise. We conclude that F2) holds since $\bCC$ is by assumption a 2-topos. 
\end{proof}

\begin{proposition}\label{prop:inducing2toposstructure}
	Let $\bCC$ be a 2-topos and consider an adjunction 
	\[
	L\colon \bCC \llra \bDD \colon R
	\] 
	where $R$ is fully faithful and preserves $\kappa$-filtered colimits for some regular cardinal $\kappa$. Suppose that $L$ preserves oriented pullbacks. then it follows that $\bDD$ is a 2-topos and that $L$ is a morphism of 2-topoi.
\end{proposition}
\begin{proof}
	Note that $\bDD$ is presentable by \cref{prop:localisationPresentable}. It is also clear that $L$ preserves the terminal object (cf.\ \cref{def:2topoimorph}) since $R$ is fully faithful and right adjoints preserve terminal objects. Therefore, the only thing left to show is that $\bDD$ satisfies fibrational descent. Before going into the proof, let us remind the reader that our hypotheses guarantee that $L$ preserves free $\epsilon$-fibrations, and consequently fibrations and morphisms thereof.

	We will show that $\bDD$ satisfies $\epsilon$-fibrational descent for $\epsilon=0$ (the case $\epsilon=1$ is completely dual). Let $(\bII,E)$ be marked 2-category and consider a diagram $q \colon \bII \to \bDD$ with associated $E$-lax colimit cone $\overline{q} \colon \bII^{\colimcone}_{\lax} \to \bDD$. We consider an $(E,0)$-cartesian transformation $\alpha \colon p \to q$ and an extension to $\overline{\alpha} \colon \overline{p} \to \overline{q}$.
	
	First, let us suppose that $\overline{p}$ is a colimit cone, and let us show that $\alpha$ is $(E^{\colimcone}_{\lax},0)$-cartesian. Note that if $\overline{R\alpha}$ is the map induced from $R\alpha$ by passing to $E$-lax colimit cones, we may recover $\overline{\alpha}$ as the image of $\overline{R\alpha}$ along the localization functor $L$. Moreover, since $\bCC$ is a 2-topos, the latter is $(E^{\colimcone}_{\lax},0)$-cartesian, so that the claim follows once we verify that $L$ preserves $(E^{\colimcone}_{\lax},0)$-cartesian transformations. But since $L$ by assumption preserves oriented pullbacks, \cref{cor:pullbackoffree} implies that $L$ also preserves strong pullbacks of free fibrations and thus also (by a straightforward retract argument) strong pullbacks of all 0-fibrations, from which the claim follows immediately.
	
	Next, we show the converse. If $\overline{\alpha} \colon \overline{p} \to \overline{q}$ is an $(E^{\colimcone}_{\lax},0)$-cartesian transformation, then so is $R\overline{\alpha}$. Consider the pullback diagram
	\[
	\begin{tikzcd}
		P \arrow[d] \arrow[r] & R \overline{p} \arrow[d, "R\overline{\alpha}"] \\
		\overline{Rq} \arrow[r] & R \overline{q}, 
	\end{tikzcd}
	\]
	and observe that the horizontal maps become equivalence once restricted along $\bII\into \bII_{\elax}^\triangleright$. Moreover, since the vertical map on the right is $(E^{\colimcone}_{\lax},0)$-cartesian, so must be the one on the left. Hence $P$ is colimit cone, by fibrational descent in $\bCC$.  
	Note that the vertical maps in the above diagram are point-wise given by 0-fibrations. Thus, since we already observed above that $L$ preserves strong pullbacks of 0-fibrations, we find that the image of the above square along $L$ remains a pullback diagram.
	Since $L$ inverts the lower horizontal map, it must thus also invert the one on top, so that $\overline{p}$ is a colimit cone.
	
	To finish the proof, we consider $(E,0)$-cartesian transformations $p \to q$ and $\ell \to q$, together with a point-wise map of fibrations $p \to \ell$, and wish to show that this extends to a point-wise map of fibrations $\overline{p} \to \overline{\ell}$ on colimit cones.
	
	A similar argument as above shows that we can produce a commutative diagram
	\[\begin{tikzcd}
		& {R\overline{p}} && {R\overline{\ell}} \\
		{\overline{Rp}} && {\overline{R\ell}} \\
		& {R\overline{q}} && {R\overline{q}} \\
		{\overline{Rq}} && {\overline{Rq}}
		\arrow[from=1-2, to=1-4]
		\arrow[from=1-2, to=3-2]
		\arrow[from=1-4, to=3-4]
		\arrow[from=2-1, to=1-2]
		\arrow[from=2-1, to=2-3]
		\arrow[from=2-1, to=4-1]
		\arrow[from=2-3, to=1-4]
		\arrow[from=2-3, to=4-3]
		\arrow["{\operatorname{id}}"{description, pos=0.2}, from=3-2, to=3-4]
		\arrow[from=4-1, to=3-2]
		\arrow["{\operatorname{id}}"{description}, from=4-1, to=4-3]
		\arrow[from=4-3, to=3-4]
	\end{tikzcd}\]
	where the backwards-pointing morphisms are inverted by $L$. By construction, the map $\overline{Rp}\to \overline{R\ell}$ is a point-wise map of fibrations. The result now follows from the observation that $L$ preserves maps of fibrations.
\end{proof}

\begin{example}\label{ex:idempotentCompleteCategories}
	Let $\Catidem$ be the 2-category of idempotent complete 1-categories.  The canonical inclusion $R \colon \Catidem \to \Cat$ admits a left adjoint $L$ which sends a category to its idempotent completion \cite[Proposition 5.1.4.9]{LurieHTT}. We claim that $\Catidem$ is a 2-topos, and that the idempotent completion functor $L$ is a morphism of 2-topoi. In fact, as $R$ is accessible, the claim follows from combining \cref{thm:cat2topos} with \cref{prop:inducing2toposstructure} once we can verify that $L$ preserves oriented pullbacks.
	
	For a category $c \in \Cat$, let us set $L(c)=c^{\vee}$. We need to show that the canonical map
	\[
	\varphi \colon a \orientedtimesrl_{c}b  \to a^{\vee} \orientedtimesrl_{c^{\vee}}b^{\vee}
	\]
	exhibits $a^{\vee} \orientedtimesrl_{c^{\vee}}b^{\vee}$ as the idempotent completion of $a \orientedtimesrl_{c}b$. Before diving into the proof, we make some remarks:
	\begin{enumerate}
		\item Given an $\epsilon$-fibration $\pi \colon x \to c$ such that $c$ is idempotent complete, it follows from \cite[Corollary 4.3.1.11]{LurieHTT} that $x$ is idempotent complete if and only if the fibres of $\pi$ are idempotent complete.
		\item Let $H\colon \Lambda^2_2 \times [1] \to \Cat$ be a natural transformation which is point-wise fully faithful. Then the induced morphism between the associated oriented pullbacks $a_0 \orientedtimesrl_{c_0} b_0 \to a_1 \orientedtimesrl_{c_1}b_1$ is fully faithful as well. 
	\end{enumerate}
    We commence by showing that $a^{\vee} \orientedtimesrl_{c^{\vee}}b^{\vee}$ is idempotent complete. Note that due to \cref{cor:pullbackoffree} and our first observation it will be enough to prove that the fibres of $\Free^{\epsilon}_{c^{\vee}}(p^{\vee}) \to c^{\vee}$ are idempotent complete.  Let $x \colon [0] \to c^{\vee}$. Then it follows that we have a pullback square
	\[
	\begin{tikzcd}
		\Free_{c^\vee}^{(1-\epsilon)}(x)\times_{c^{\vee}}a^{\vee} \arrow[r] \arrow[d] & \Free^{\epsilon}_{c^{\vee}}(p^{\vee}) \arrow[d] \\
		{[0]} \arrow[r,"x"] & c^{\vee}
	\end{tikzcd}
	\]
	Finally, we observe that the fibres of the $(1-\epsilon)$-fibration $\Free_{c^\vee}^{(1-\epsilon)}(x)\times_{c^{\vee}}a^{\vee} \to a^{\vee}$ are spaces. We conclude after noting that spaces are always idempotent complete.
	
	To finish the proof we need to show that $\varphi$ is fully faithful and that every object in $a^{\vee} \orientedtimesrl_{c^{\vee}}b^{\vee}$ is a retract of an object of the form $\varphi(\tau)$. The claim regarding fully faithfulness follows easily from remark (2). Let $p^{\vee}(\alpha) \to  q^{\vee}(\beta)$ be an object in $a^{\vee} \orientedtimesrl_{c^{\vee}}b^{\vee}$ where $p\colon a \to c $ and $q  \colon b \to c$. We express $\alpha$ as a retraction of $j_a(u)$ and $\beta$ as a retraction of $j_b(v)$ where $j_a$ (resp.\ $j_b$) denote the corresponding Yoneda embeddings. Unpacking the definitions we obtain a retraction diagram
	\[
	\begin{tikzcd}
		p^{\vee}(\alpha) \arrow[d] \arrow[r] & j_c(p(u))\simeq p^{\vee}(j_a(u)) \arrow[d] \arrow[r] &  p^{\vee}(\alpha) \arrow[d] \\
		q^{\vee}(\beta) \arrow[r] & j_c(q(v)) \simeq q^{\vee}(j_b(v)) \arrow[r] & q^{\vee}(\beta)
	\end{tikzcd}
	\]
	To finish the proof we note that fully faithfulness of $j_c$ provides us with a map $p(u) \to q(v)$, so the result follows.
\end{example}

\begin{remark}\label{rem:pullbackidem}
		In the situation of \cref{ex:idempotentCompleteCategories}, we would like to point out that $L$ does not preserve strong pullbacks in general. Let $\operatorname{Ret}= [2] \coprod_{\{0,2\}} [0]$ and let $\tau \colon \operatorname{Idem} \to \operatorname{Ret}$ be the full subcategory on the object $1$, so that $\tau$ is the inclusion of the free-walking idempotent into the free-walking split idempotent. Then we have a pullback square
		\[
		\begin{tikzcd}
			{[0]} \arrow[r] \arrow[d] &  {[2]} \arrow[d] \\
			\operatorname{Idem} \arrow[r] &  \operatorname{Ret}
		\end{tikzcd}
		\]
		in $\Cat$ which the idempotent completion functor carries to the commutative diagram
		\[
		\begin{tikzcd}
			{[0]} \arrow[r] \arrow[d] &  {[2]} \arrow[d] \\
			\operatorname{Ret} \arrow[r,"\id"] &  \operatorname{Ret},
		\end{tikzcd}
		\]
		which is clearly not a pullback.
\end{remark}

\begin{proposition}
	Let $\bCC$ be a 2-topos, then $\bCC^{\co}$ is again a 2-topos which we call the \emph{mirror} 2-topos of $\bCC$.
\end{proposition}
\begin{proof}
	We have that $\bCC^{\co}$ is again 2-presentable (cf.\ \cref{rem:co2presentable}), and it is obvious that it satisfies $\epsilon$-fibrational descent.
\end{proof}

\begin{definition}\label{def:involutive}
	A 2-topos $\bCC$ is said to be \emph{involutive} if there exists an equivalence $\Psi \colon \bCC \xrightarrow{\simeq} \bCC^{\co}$ such that $\Psi^{\co} \circ \Psi \simeq \id$.
\end{definition}

\begin{remark}
	The definition of an involutive 2-topos serves as an abstraction of the opposite category involution $(-)^\op \colon \Cat \to \Cat^{\co}$. Even though involutive 2-topoi constitute a large class of 2-topoi (see \cref{prop:localicinvolutive}) which is well-suited to the study of synthetic category theory, it is important to remark that many interesting 2-topoi lack this property, as we will illustrate in the next example. 
\end{remark}

\begin{example}\label{ex:monad}
	The 2-category $\FUN(\Mnd,\Cat)$ of monads in $\Cat$ is a 2-topos by \cref{thm:cat2topos} and \cref{prop:fun2topos}. We claim that this 2-topos is not involutive. In fact, let us assume for contradiction that we have an involution $\Psi \colon \FUN(\Mnd,\Cat) \xrightarrow{\simeq} \FUN(\Mnd,\Cat)^{\co}$. Using the canonical involution in $\Cat$ we obtain another equivalence $\Xi \colon \FUN(\Mnd,\Cat) \simeq \FUN(\Mnd^{\co},\Cat)$. We will show that every such equivalence must send the unique correpresentable functor $\Mnd(+,-)$ to $\Mnd^\co(+,-)$ (where $+\in\Mnd$ is the unique object). Once this is shown, it will follow from Yoneda's lemma that we have an equivalence $\Mnd \simeq \Mnd^{\co}$. This will provides us with the desired contradiction, since by definition we have $\Mnd(+,+)=\Delta_{+}$ (where $+\in\Mnd$ is the unique object and where $\Delta_+$ denotes the augmented simplex category), and it is easy to see that $\Delta_{+}$ is not self-dual, i.e.\  $\Delta_{+}^{\op}\not\simeq \Delta_{+}$.\footnote{Any equivalence $f\colon\Delta_+\simeq\Delta_+^\op$ must satisfy $f([0])=\varnothing$ (since terminal objects must be preserved). We then observe that we must have $\Delta_{+}([0],[1])\simeq \Delta_{+}(f([1]),f([0]))=\Delta_+(f([1]),\varnothing)=\varnothing$ which is a contradiction.} 
	
	Let us say that a functor $F \colon \Mnd \to \Cat$ is \emph{tiny} if the correpresentable functor 
	\[
	\Nat_{\Mnd,\Cat}(F,-) \colon \FUN(\Mnd,\Cat) \to \Cat
	\] 
	preserves all small partially (op)lax colimits. Before continuing, we make a couple of observations:
	\begin{enumerate}
		\item The equivalence $\Xi$ preserves tiny objects.
		\item By an analogous argument to the one found in \cite[Proposition 5.1.6.8]{LurieHTT}, the tiny functors are precisely the retracts of the corepresentable functor $\Mnd(+,-)$.
	\end{enumerate}
	
	To complete the argument, we thus only need to verify that there is no tiny object besides $\Mnd(+,-)$. To see this, consider a retract diagram $F \to \Mnd(+,-) \to F$, providing us with an idempotent $\Mnd(+,-)\to\Mnd(+,-)$ that corresponds by Yoneda's lemma to an idempotent $+\to +$ in $\Mnd$, i.e.\ an object $[a]\in \Delta_+$ satisfying $[a]\oplus [a]= [a]$, where $\oplus$ denotes ordinal sum. Clearly, the only possible choice for this is $[a]=\varnothing$, which implies that the idempotent $\Mnd(+,-)\to\Mnd(+,-)$ must be the identity. Since this implies $F\simeq\Mnd(+,-)$, the claim follows.

\end{example}

\subsection{Lawvere-Tierney axioms for 2-topoi}\label{subsec:LTaxioms}
In 1-topos theory, the \emph{Lawvere-Tierney axioms} characterise 1-topoi as presentable 1-categories which are locally cartesian closed and which admit object classifiers. In this section, we establish analogous axioms for the theory of 2-topoi. 

To that end, we introduce the notion of a $\kappa$-compact $\epsilon$-fibration in a 2-topos, and we show that for sufficiently large cardinals $\kappa$, there exist classifiers $\Omega^{\epsilon,\kappa}$ for the class of $\kappa$-compact $\epsilon$-fibrations. Then, we show that $\epsilon$-fibrations are \emph{exponentiable} in a 2-topos. Finally, the main theorem in this section (\cref{thm:defcheck}) asserts that 2-topoi are characterised among presentable 2-categories by these two properties.

\begin{definition}\label{def:compactfib}
	Let $\bCC$ be a presentable 2-category. We say that an $\epsilon$-fibration $p \colon x \to s$ in $\bCC$ is $\kappa$-compact if for every pullback diagram
	\[
	\begin{tikzcd}
		y \arrow[r] \arrow[d]  & x \arrow[d,"p"] \\
		t \arrow[r] & s
	\end{tikzcd}
	\]
	where $t$ is a $\kappa$-compact object in $\bCC$, the object $y$ is $\kappa$-compact as well. We denote by $\Fib^{\epsilon,\kappa}_{\bCC}$ the full-sub-2-category of $\Fib^{\epsilon}_{\bCC}$ spanned by the $\kappa$-compact fibrations.
\end{definition}

\begin{lemma}\label{lem:smallkappafib}
	Let $\bCC$ be a 2-topos and let $\kappa$ be a regular cardinal such that $\bCC$ is $\kappa$-accessible and such that $\kappa$-compact objects are stable under oriented pullbacks. Let $c\in\bCC$ be an arbitrary object, and choose a $\kappa$-filtered diagram $d\colon\JJ\to\bCC^\kappa$ with $\colim d\simeq c$. Then the canonical map
	\[
	\left(\Fib^{\epsilon,\kappa}_{\bCC}\right)_{/c} \to \lim_{\JJ} \left(\Fib^{\epsilon,\kappa}_{\bCC}\right)_{/d(-)}
	\]
	is an equivalence. In particular, the left-hand side is small.
\end{lemma}
\begin{proof}
	The second claim follows immediately from the first since $\JJ$ is small and the collection of $\kappa$-compact $\epsilon$-fibrations over a $\kappa$-compact base is clearly small as well. Thus, it suffices to show that the first claim holds.
	
	To that end, by using fibrational descent, it suffices to show that if $p \colon x \to c$ is an $\epsilon$-fibration such that its restriction to $d(i)$ is $\kappa$-compact for each $i\in I$, then $p$ is $\kappa$-compact as well. Since every map $c^\prime\to d$ from a $\kappa$-compact object $c^\prime$ necessarily factors through one of the maps $d(i)\to c$ for some $i\in \II$, this is evident.
\end{proof}

\begin{proposition}\label{prop:compactsheaf}
	Let $\bCC$ be a 2-topos and let $\kappa$ be a regular cardinal such that $\bCC$ is $\kappa$-accessible and the $\kappa$-compact objects of $\bCC$ are stable under oriented pullbacks. Then we have functors
	\[
	\left(\Fib_{\bCC}^{0,\kappa}\right)_{/(-)} \colon \bCC^{\op} \to \CCat, \quad \quad \left(\Fib_{\bCC}^{1,\kappa}\right)_{/(-)} \colon \bCC^{\coop} \to \CCat
	\]
	which send partially $\epsilon$lax colimits to partially oplax limits.
\end{proposition}
\begin{proof}
	The existence of the two functors follows from \cref{lem:smallkappafib}. For the second claim, we assume without loss of generality $\epsilon=0$. Let $(\bII,E)$ be a (small) marked 2-category and consider a functor $q \colon  \bII \to \bCC$ with $E$-lax colimit $c$. Our goal is to show that we have a pullback diagram
	\[
	\begin{tikzcd}
		\left(\Fib_{\bCC}^{0,\kappa}\right)_{/c} \arrow[r] \arrow[d] & \left(\Fib_{\bCC}^{0}\right)_{/c} \arrow[d,"\simeq"] \\
		\lim^{\eoplax}_{\bII^{\op}}\left(\Fib_{\bCC}^{0,\kappa}\right)_{/d(-)} \arrow[r] & \lim^{\eoplax}_{\bII^{\op}}\left(\Fib_{\bCC}^{0}\right)_{/d(-)}.
	\end{tikzcd}
	\]
	We proceed in several steps:
	\begin{enumerate}
		\item Let us assume that the diagram $q$ factors through $\bCC^{\kappa}$ and that $\bII$ is $\kappa$-small. In this case the result follows from the fact that $\kappa$-compact objects are stable under $\kappa$-small partially lax colimits (cf.\ \cref{cor:smallcolimofcompacts}).

		\item Let us assume that $(\bII,E)$ is a $\kappa$-filtered 1-category and that $E=\sharp$. This was done already in \cref{lem:smallkappafib}.
		\item Let us assume that $\bII$ is $\kappa$-small. Then it follows from \cref{rem:accfunctorcat} that there exists a $\kappa$-filtered diagram $\JJ \to \Fun(\bII,\bCC^{\kappa})$ whose (strong) colimit is $q$. Invoking \cref{cor:existenceLKE} we observe that the E-lax colimit cone $\overline{q} \colon \bII^{\colimcone}_{\elax} \to \bCC$ can be expressed as a $\kappa$-filtered colimit $\overline{q}= \colim_{j\in\JJ} \overline{q}_j$. The result thus follows form \cref{prop:limitscommute} and the previous steps.
		
		\item For the general case, we invoke \cref{prop:2catsasfilteredcolimits} and \cref{prop:diagramdecomposition} which allows us to decompose the diagram $q \colon \bII \to \bCC$ as a filtered colimit of $\kappa$-small diagrams. We conclude using \cref{prop:limitscommute} together with the previous discussion.
	\end{enumerate}
	The result now follows.
\end{proof}

\begin{corollary}\label{cor:universeexists}
	Let $\bCC$ be a 2-topos. Then for sufficiently large regular cardinals $\kappa$, the functors 
	\[
	\left(\Fib_{\bCC}^{0,\kappa}\right)_{/(-)}^{\leq 1} \colon \bCC^{\op} \to \Cat, \quad \quad \left(\Fib_{\bCC}^{1,\kappa}\right)_{/(-)}^{\leq 1} \colon \bCC^{\coop} \to \Cat
	\]
	are representable by objects $\Omega^{\epsilon,\kappa}\in\bCC$. 
\end{corollary}
\begin{proof}
	Combine \cref{prop:compactsheaf} with \cref{cor:sheavesRepresentable}.
\end{proof}

\begin{definition}
	In the situation of \cref{cor:universeexists}, we call the objects $\Omega^{\kappa,\epsilon}$ the \emph{classifiers} of $\kappa$-compact $\epsilon$-fibrations.
\end{definition}

\begin{remark}\label{rem:universalFibration}
	As a consequence of \cref{cor:universeexists} and Yoneda's lemma, there exist \emph{universal} $\kappa$-compact $\epsilon$-fibrations
	\[
	U^{\epsilon,\kappa} \colon \Omega_*^{\epsilon,\kappa} \to \Omega^{\epsilon,\kappa},
	\]
	defined by the property that every $\kappa$-compact $\epsilon$-fibration $p\colon x\to s$ uniquely fits into a pullback square
	\[
	\begin{tikzcd}
		x \arrow[r] \arrow[d,swap,"p"] & \Omega_*^{\epsilon,\kappa} \arrow[d,"U^{\epsilon,\kappa}"] \\
		s \arrow[r] & \Omega^{\epsilon,\kappa}.
	\end{tikzcd}
	\]
\end{remark}

We now move on to the second half of the Lawvere-Tierney axioms:
\begin{definition}\label{def:exponentiability}
	A morphism $f \colon c \to d$ in $\bCC$ is said to be \emph{exponentiable} if the pullback functor (whenever defined)
	\[
	f^* \colon \bCC_{/d} \to \bCC_{/c}
	\]
	admits a right adjoint. We say that $\bCC$ is \emph{locally cartesian closed} if $\epsilon$-fibrations are exponentiable.
\end{definition}

\begin{proposition}\label{prop:condussy}
	Let $\bCC$ be a presentable 2-category satisfying $\epsilon$-fibrational descent. Then $\epsilon$-fibrations are exponentiable (cf.\ \cref{def:exponentiability}).
\end{proposition}
\begin{proof}
	Let $p\colon c \to d$ be an $\epsilon$-fibration. We wish to produce a right adjoint to the pullback functor $p^* \colon \bCC_{/d} \to \bCC_{/c}$. By \cref{prop:adjointFunctorTheoremLeft}, it is enough to show that $p^\ast$ is lax cocontinuous. To that end, let $(\bII, E)$ be a small marked 2-category, let $d\colon \bII\to \Over{\bCC}{d}$ be a diagram with $E$-lax colimit $x\to d$ and consider the pullback square
	\[
	\begin{tikzcd}
		y \arrow[r] \arrow[d] & c \arrow[d] \\
		\colim_{\bII}^{\epsilonlax} d \arrow[r] & d.
	\end{tikzcd}
	\]
	Observe that the vertical map on the left is again an $\epsilon$-fibration. Hence the claim follows immediately from $\epsilon$-fibrational descent.
\end{proof}

\begin{corollary}\label{cor:freePreservesColimits}
	Let $\bCC$ be a 2-topos and let $c\in\bCC$ be an arbitrary object. Then the endofunctor $\Free_c^\epsilon\colon\Over{\bCC}{c}\to\Over{\bCC}{c}$ preserves partially (op)lax colimits.
\end{corollary}
\begin{proof}
	It suffices to show that $\Free_c$ preserves partially (op)lax colimits when viewed as a functor $\Over{\bCC}{c}\to\bCC$. To that end, choose a marked 2-category $(\bII, E)$ and let $d\colon\bII\to\Over{\bCC}{c}$ be a diagram. Then $\Free_c^\epsilon(\colim_{\bII}^{\elaxoplax} d)$ fits into a pullback square 
	\[
	\begin{tikzcd}
		\Free_{c}^{\epsilon}(\colim_{\bII}^{\elaxoplax} d) \arrow[r] \arrow[d] & c^{[1]} \arrow[d]  \\
		\colim_{\bII}^{\elaxoplax} d \arrow[r] & c,
	\end{tikzcd}
	\]
	hence the claim follows from \cref{prop:condussy} in light of the observation that the right vertical map is a $(1-\epsilon)$-fibration.
\end{proof}

\begin{theorem}\label{thm:defcheck}
	Let $\bCC$ be a presentable 2-category. Then the following are equivalent:
	\begin{enumerate}
		\item The 2-category $\bCC$ satisfies fibrational descent.
		\item The following conditions hold:
		\begin{enumerate}
			\item[i)] For sufficiently large cardinals $\kappa$, we have classifiers for $\kappa$-compact $\epsilon$-fibrations.
			\item[ii)] Every $\epsilon$-fibration is exponentiable.
		\end{enumerate} 
	\end{enumerate}
\end{theorem}
\begin{proof}
	The implication $1) \implies 2)$ follows from \cref{prop:fibrepresents} and \cref{prop:condussy}. We now show the converse. 
	
	Let $q \colon \bII \to \bCC$ be a diagram, where $(\bII, E)$ is a small marked 2-category, and let $\overline{q}\colon\bII^{\colimcone}_{\epsilonlax} \to \bCC$ be the associated $E$-$\epsilon$lax colimit cone. Let $\alpha \colon p \to q$ be $(E,\epsilon)$-cartesian. Since $\bII$ is small, we can find a suitably large cardinal $\kappa$ such that $\bII$ is $\kappa$-small and $\alpha$ is point-wise given by $\kappa$-compact $\epsilon$-fibrations. Using $i)$, we now deduce that there is a unique extension $\overline{\alpha}\colon \overline{p}\to\overline{q}$ of $\alpha$ that is $(E^\triangleright_{\epsilonlax}, \epsilon)$-cartesian. Moreover, applying $ii)$ to the $\epsilon$-fibration $\overline{p}(\ast)\to \overline{q}(\ast)$, we deduce that $\overline{p}$ must be an  $E$-$\epsilon$lax colimit cone. This is already sufficient to deduce that F1) holds. 
	The verification of axiom F2) proceeds in a completely analogous manner.
\end{proof}

\subsection{Giraud's theorem for 2-topoi}\label{subsec:giraud}
In 1-topos theory, Giraud's theorem characterises 1-topoi, as defined axiomatically via Giraud's axioms or via the notion of descent, as those categories that arise as left exact and accessible localizations of presheaf categories \cite[Theorem 6.1.0.6]{LurieHTT}. The goal of this section is to prove the following categorified version of this result:

\begin{theorem}\label{thm:giraud}
	For a 2-category $\bCC$, the following are equivalent:
	\begin{enumerate}
		\item $\bCC$ is a 2-topos;
		\item there exists a small 2-category $\bSS$ and an adjunction
		\[
		L \colon \PSh_{\Cat}(\bSS) \llra \bCC \colon R
		\]
		such that $R$ is fully faithful and commutes with $\kappa$-filtered colimits for some regular cardinal $\kappa$ and such that $L$ preserves oriented pullbacks.
	\end{enumerate}
\end{theorem}

In order to prove \cref{thm:giraud}, we need to first establish a few auxiliary results:

\begin{proposition}\label{prop:cofree}
	Let $\bCC$ be a 2-topos, and let $c\in\bCC$ be an arbitrary object. Then the inclusion 
	\[
	U_c \colon \Fib_{/c}^{\epsilon} \to \bCC_{/c}
	\]
	creates partially (op)lax colimits.
\end{proposition}
\begin{proof}
	Let $q\colon \bII \to 
	\Fib_{/c}^{\epsilon} \to \bCC_{/c}$ be a diagram, where $(\bII,E)$ is a marked 2-category, and let us set $F(q)=\Free_{c}^\epsilon q$.
	We observe that we have a canonical transformation $q \to F(q)$ in $\FUN(\bII,\bCC_{/c})$, which is a right (resp.\ left) adjoint as $q$ factors through $\Fib_{/c}^{\epsilon}$. Via left Kan extension, we thus obtain a right (left) adjoint $\overline{q} \to \overline{F(q)}$ in $\FUN(\bII^{\colimcone}_{\elaxoplax},\bCC_{/c})$. As \cref{cor:freePreservesColimits} moreover implies that the canonical map $\overline{F(q)}\to F(\overline{q})$ is an equivalence, we conclude that $\overline{q}$ factors through $\Fib_{/c}^\epsilon$.
	
	We claim that $\overline{q}$ is a colimit cone in $\Fib^{\epsilon}_{/c}$. To see this, consider an $E$-(op)lax transformation $\alpha \colon q \to \underline{p}$, where the codomain is the constant diagram on an $\epsilon$-fibration $p\colon y \to c$. Let $\overline{q}(*)$ be given by a fibration $q\colon x \to c$ and where the transformation is point-wise given by morphisms of $\epsilon$-fibrations. To finish the proof, we need to show that the induced map $\overline{q}(\ast)\to p$ is a morphism of $\epsilon$-fibrations as well. In other words, we need to show that the (op)lax diagram
	\[
	\begin{tikzcd}
		\Free_{c}^{\epsilon}(q(\ast)) \arrow[r] \arrow[d] & \Free_{c}^{\epsilon}(p)  \arrow[d] \\
		q(\ast) \arrow[r] & p
	\end{tikzcd}
	\]
	is commutative. Since the vertical map on the left is obtained as the colimit of the map $F(q)\to q$ that is left (right) adjoint to the morphism $q\to F(q)$ considered above, it suffices to check that the above diagram commutes after restriction to each of the maps $F(q)(i)\to q(i)$ for $i\in\bII$. As the latter commutes by assumption, the claim follows. 
\end{proof}

\begin{corollary}\label{cor:Fib2Presentable}
	Let $\bCC$ be a 2-topos, and let $c\in\bCC$ be an arbitrary object. Then the 2-category $\Fib_{/c}^{\epsilon}$ is presentable.
\end{corollary}
\begin{proof}
	By \cref{prop:cofree}, the 2-category $\Fib_{/c}^{\epsilon}$ is cocomplete, so it suffices to show that the underlying 1-category is presentable. Since the inclusion $U_c\colon \Fib_{/c}^{\epsilon} \to \bCC_{/c}$ admits a left adjoint $\Free_c$ and since $U_c$ is cocontinuous by \cref{prop:cofree} and is moreover conservative, the underlying adjunction of 1-categories is monadic, and the associated monad preserves colimits. Thus we may apply \cite[Corollary~4.2.3.7]{LurieHA} to deduce the desired result.
\end{proof}

\begin{proposition}\label{prop:rightkan}
	Let $\bCC$ be a 2-topos a consider a morphism $f\colon c \to d$ in $\bCC$. Then the pullback functor
	\[
	f^* \colon \Fib_{/d}^{\epsilon} \to \Fib_{/c}^{\epsilon}
	\]
	is cocontinuous.
\end{proposition}
\begin{proof}
	We will only consider the case $\epsilon=0$ since the remaining case is strictly dual. We factor $f$ as 
	\begin{equation*}
		c \xrightarrow{\eta_d} \Free_{d}^1(c) \xrightarrow{\Free_d^1(f)} d.
	\end{equation*}
	Observe that since $\Free_d^1(f)$ is a fibration, it follows from \cref{prop:cofree} and \cref{prop:condussy} that pulling back along $\Free_d^1(f)$ preserves colimits, so that it suffices to show that pullback along $\eta_d$ preserves colimits. But the latter is a \emph{right} adjoint, with left adjoint given by the canonical projection $\Free_d^1(c)\to c$, so that $\eta_d^\ast$ must be a left adjoint and therefore cocontinuous as well.
\end{proof}

\begin{proposition}\label{prop:preGiraud}
	Let $\bCC$ be a 2-topos and consider an adjunction
	\[
	L \colon \PSh_{\Cat}(\bSS) \llra \bCC \colon R
	\]
	such that the right adjoint $R$ is fully faithful. Suppose that $\bSS$ has oriented pullbacks and let $\hat{L} \colon \bSS \to \bCC$ be the restriction of $L$ along the Yoneda embedding. Then if $\hat{L}$ preserves oriented pullbacks, so does $L$.
\end{proposition}
\begin{proof}
	Let $h_{\bSS} \colon \bSS \to \Fun(\bSS^{\op},\Cat)$ denote the Yoneda embedding. The proof will procede in several steps:
	\begin{description}
		\item[Free fibrations over a representable base]
		We first show that given $s\in\bSS$, the canonical map
		\begin{equation*}
			L\Free^\epsilon_{h_{\bSS}(s)}(-)\to \Free_{L h_{\bSS}(s)}^\epsilon(-)
		\end{equation*}
		is an equivalence. As both $\bCC$ and $\PSh_{\Cat}(\bSS)$ are 2-topoi (the latter by \cref{prop:fun2topos}), we deduce from the fact that $L$ is cocontinuous and \cref{cor:freePreservesColimits} that both sides are cocontinuous as functors $\Over{\PSh_{\Cat}(\bSS)}{h_{\bSS}(s)}\to \Over{\bCC}{Lh_{\bSS}(s)}$. Together with \cref{prop:coYoneda}, this implies that the above map is an equivalence already if it is an equivalence when restricted along the inclusion $\Over{\bSS}{s}\into\Over{\PSh_{\Cat}(\bSS)}{h_{\bSS}(s)}$ that is induced by the Yoneda embedding. Since this restriction is an equivalence by the very assumption on $\hat L$ to preserve oriented pullbacks, the claim follows.
		
		\item[Pullback of a fibration along a map of representables]
		Next, we show that $L$ preserves pullback squares of the form
		\[
		\begin{tikzcd}
			x_{|t} \arrow[r] \arrow[d] & x \arrow[d,"p"] \\
			h_{\bSS}(t) \arrow[r] & h_{\bSS}(s)
		\end{tikzcd}
		\]
		where the vertical arrows are $\epsilon$-fibrations. Note that since $p$ is an $\epsilon$-fibration, we have a retract diagram $  p \to \Free^{\epsilon}_{h_{\bSS}(s)}(p) \to p$ in which the inclusion is right (resp.\ left) adjoint to the retraction.
		Note that together with the previous step, this already implies that $L(p)$ is an $\epsilon$-fibration. Moreover, together with functoriality of pullbacks, this implies that we may reduce to the case where $p$ is a free fibration. Now the claim follows from the previous step together with \cref{prop:cofree} and \cref{prop:rightkan} in light of the assumption that $\hat L$ preserves oriented pullbacks (and therefore strong pullbacks of fibrations).
		
		\item[Fibrations over a general base]
		Let $p \colon x \to c$ be an $\epsilon$-fibration, and let us show that $L(p)$ is still an $\epsilon$-fibration. By combining \cref{prop:coYoneda} with fibrational descent in $\PSh_{\Cat}(\bSS)$, we obtain an $E$-cartesian transformation $\alpha\colon h\to k$ with colimit $p$ such that $k$ takes values in $\bSS$. Using the previous step, the image of $\alpha$ along $L$ is still an $E$-cartesian transformation. Since $L$ preserves colimits and by using fibrational descent in $\bCC$, we conclude that $L(p)$ is an $\epsilon$-fibration.
		Let us point out that the same argument can be used to show that $L$ preserves morphisms of $\epsilon$-fibrations.
		
		\item[Pullbacks along a fibration]
		Consider a pullback diagram
		\[
		\begin{tikzcd}
			x_{|s} \arrow[r] \arrow[d] & x \arrow[d,"p"] \\
			d \arrow[r] & c
		\end{tikzcd}
		\]
		in $\PSh_{\Cat}(\bSS)$,
		where the vertical maps are $\epsilon$-fibrations. We claim that the canonical map $\varphi \colon L(x)_{|s} \to L(x_{|s})$ (which is by construction and the previous step a map of $\epsilon$-fibrations) is an equivalence. Invoking fibrational descent in $\bCC$ as well as \cref{prop:coYoneda}, it suffices to show that this map becomes invertible after being pulled back along each map $L(h_{\bSS}(s)) \to L(d)$. It follows that we can reduce to the case where $d=h_{\bSS}(s)$. But then the above pullback square is one of the canonical squares that exhibit $p$ as a partially lax colimit of the $E$-cartesian transformation from the previous step. Consequently, the very same argument as in the previous step also implies that the image of this square along $L$ is still a pullback.
		
		\item[Free fibrations over a general base]
		For ease of notation let us denote $\bXX= \PSh_{\Cat}(\bSS)$. First, we wish to show that for every $c \in \bXX$ we have an adjunction
		\[
		L_c \colon  \left(\Fib^{\epsilon}_{\bXX}\right)_{/c} \llra \left(\Fib^{\epsilon}_{\bCC}\right)_{/L(c)} \colon R_c
		\]
		where $L_c$ is the canonical map induced from $L$ on the slices (by the previous step) and where $R_c$ is given by the composition
		\[
		\left(\Fib^{\epsilon}_{\bCC}\right)_{/L(c)} \to \left(\Fib^{\epsilon}_{\bXX}\right)_{/RL(c)} \to \left(\Fib^{\epsilon}_{\bXX}\right)_{/c}. 
		\]
		Here, the first functor is simply given by applying $R$ and the second functor is given by pulling back along the unit $\eta_c \colon c \to RL(c)$.
		Unwinding the definitions, the assertion follows once we verify that for every $\epsilon$-fibration $p\colon x \to s$ the diagram
		\[
		\begin{tikzcd}
			x \arrow[r, "\eta_x"] \arrow[d] & RL(x) \arrow[d] \\
			c \arrow[r,"\eta_c"] & RL(c)
		\end{tikzcd}
		\]
		is a map of $\epsilon$-fibrations. Equivalently, we need to show that the map $x \to RL(x)\times_{RL(c)} c$ is a map of $\epsilon$-fibrations over $c$. Invoking fibrational descent, we only need to verify our claim after pulling back along a representable $h_{\bSS}(s) \to c$, so that the previous step allows us to reduce to the case where $c$ is representable. In this case, the first step shows that the canonical map $RL(\Free_c^\epsilon(p))\to\Free^\epsilon_{RL(c)}(RL(p))$ is an equivalence, which in turn implies that the (op)lax square
		\begin{equation*}
			\begin{tikzcd}
				\Free_c^\epsilon(x)\arrow[d]\arrow[r] & \Free^\epsilon_{RL(c)}(RL(x))\arrow[d]\\
				c\arrow[r] &RL(c)
			\end{tikzcd}
		\end{equation*}
		is simply the naturality square associated with $\Free_c^\epsilon(x)\to c$ and therefore in particular commutative.
		
		Let us remark that fully faithfulness of $R$ implies that if $c= R(c^\prime)$ for some $c^\prime\in\bCC$, the functor $R_c$ is simply given by applying $R$ to the slices. Thus, by invoking \cref{thm:basechangetheorem} below, we obtain an equivalence $L_c \simeq L_{RL(c)} \circ (\eta_c)_{!}$ (where $(\eta_c)_!$ is left adjoint to $\eta_c^\ast$). Note that due to \cref{lem:kanextendingfree} below, the functor $(\eta_c)_{!}$ carries free $\epsilon$-fibrations over $c$ to free $\epsilon$-fibrations over $RL(c)$. Therefore, to show that $L_c$ preserves free $\epsilon$-fibrations, we only need to show that $L_{RL(c)}$ has the same property. So let $p \colon x \to RL(c)$ be a map. We obtain a commutative diagram
		\[
		\begin{tikzcd}
			\Free^{\epsilon}_{RL(c)}(x) \arrow[d] \arrow[r] & \Free^{\epsilon}_{RL(c)}(RL(x)) \arrow[d] \arrow[r] & RL(c)^{[1]} \arrow[d] \\
			x \arrow[r] & RL(x) \arrow[r] & RL(c)
		\end{tikzcd}
		\]
		in which both squares are pullbacks. Since $RL(c)^{[1]} \simeq R\left(L(c)^{[1]}\right)$, it follows from the previous step that $L\left(\Free^{\epsilon}_{RL(s)}(p)\right) \simeq \Free^{\epsilon}_{L(s)}(L(p))$, as desired.
	\end{description}
	We conclude by invoking \cref{cor:pullbackoffree}.
\end{proof}

\begin{proof}[{Proof of \cref{thm:giraud}}]
	The fact that (2) implies (1) follows immediately from \cref{prop:inducing2toposstructure}, so it suffices to show that (1) implies (2).
	Since $\bCC$ is presentable there exists a regular cardinal $\kappa$ and an adjunction (cf.\ \cref{subsec:2present}) 
	\[
	L \colon \PSh_{\Cat}(\bCC^\kappa) \llra \bCC \colon R
	\]
	such that $R$ is fully faithful and commutes with $\kappa$-filtered colimits. Enlarging $\kappa$ if necessary, we can assume that $\kappa$-compact objects are stable under oriented pullbacks. Hence the claim follows from \cref{prop:preGiraud}.
\end{proof}

\begin{remark}\label{rem:felix}
	A careful analysis of \cref{prop:preGiraud} reveals that in order to deduce that $L$ preserves oriented pullbacks it suffices for $\bCC$ to satisfy $\epsilon$-fibrational descent for \emph{one} of the variances, in addition to being locally cartesian closed (\cref{def:exponentiability}).
\end{remark}

\begin{corollary}
	$\Cat$ is the initial 2-topos. In other words, for any 2-topos $\bCC$ there exists a unique morphism of 2-topoi
	\[
	\const \colon \Cat \to \bCC.
	\]
\end{corollary}
\begin{proof}
	In light of \cref{thm:laxCocompletion}, the only thing that we need to show is that the unique left adjoint $\const\colon \Cat\to\bCC$ that is determined by sending $[0]\in\Cat$ to the terminal object in $\bCC$ defines a morphism of 2-topoi, i.e.\ preserves oriented pullbacks.
	By \cref{thm:giraud} we can reduce to the case where $\bCC=\PSh_{\Cat}(\bSS)$, which allows us to identify $\const$ with the diagonal map. The result follows.
\end{proof}

\begin{proposition}\label{prop:fibslice}
	Let $L \colon \bCC \to \bDD$ be a morphism of 2-topoi. Then for every object $c \in \bCC$ the induced functor 
	\[
	L_{c} \colon \left(\Fib^{\epsilon}_{\bCC}\right)_{/c} \to  \left(\Fib^{\epsilon}_{\bDD}\right)_{/L(c)} 
	\]
	is a morphism of 2-topoi.
\end{proposition}
\begin{proof}
	We will consider the case $\epsilon=0$ without loss of generality. To begin with, note that we already know from \cref{cor:Fib2Presentable} and \cref{prop:cofree} that $L_c$ is a cocontinuous functor between presentable 2-categories. Suppose for a moment that we already knew that $L_c$ preserves the terminal object and oriented pullbacks. Then this would also be the case for the map
	\begin{equation*}
		\Over{\left(\Fib_{\PSh_{\Cat}(\bSS)}^0\right)}{Rc}\to \Over{\left(\Fib_{\bCC}^0\right)}{c}
	\end{equation*}
	that is induced by the localization of 2-topoi $\PSh_{\Cat}(\bSS)\to\bCC$ provided by \cref{thm:giraud}. Since \cref{thm:str} and \cref{prop:carttransun} imply that we have an equivalence
	\[
	\Over{\left(\Fib_{\PSh_{\Cat}(\bSS)}^0\right)}{Rc} \simeq \Fun(\Un^{0}_{\bSS}(Rc),\Cat)
	\] 
	and the right-hand side is a 2-topos by \cref{prop:fun2topos}, we would be able to apply \cref{prop:inducing2toposstructure} to deduce that $\Over{\left(\Fib_{\bCC}^0\right)}{c}$  (and by symmetry of the argument also $\Over{\left(\Fib_{\bDD}^0\right)}{L(c)}$) is a 2-topos. In other words, the proof is finished once we verify that $L_c$ preserves the terminal object and oriented pullbacks.
	
	The part of the claim regarding preservation of the terminal object is trivially true, so are left to verify that oriented pullbacks are preserved. Let us consider a cospan of morphisms of 0-fibrations
	\[
	\begin{tikzcd}
		x \arrow[dr,swap,"p"] \arrow[r,"f"] & y \arrow[d,"q"] & z \arrow[l,swap,"g"] \arrow[dl,"\pi"] \\
		& c &.
	\end{tikzcd}
	\]	
	Unpacking the definitions, it follows that the oriented pullback $p \orientedtimesrl_{q} \pi$ in $\left(\Fib^{0}_{\bCC}\right)_{/c}$ is obtained as the (strong) pullback
	\[
	\begin{tikzcd}
		p \orientedtimesrl_{q}\pi \arrow[r] \arrow[d] & x \orientedtimesrl_{y} z \arrow[d] \\
		c \arrow[r] & c^{[1]}
	\end{tikzcd}
	\]
	where the right-most vertical map is the canonical map induced by the structure maps to $cc$ and the bottom horizontal map is induced by the morphism $s \colon [1] \to [0]$. Since $L$ is a morphism of 2-topoi, it will thus be enough to show that the map $x \orientedtimesrl_{y} z \to c^{[1]}$ is a 0-fibration. 
	
	Let us consider the diagram
	\[
	\begin{tikzcd}
		\Free_{c}^{0}(p) \arrow[d,swap,"l_x"] \arrow[r] & \Free_{c}^{0}(q) \arrow[d,"l_y"] & \Free_{c}^{0}(\pi) \arrow[d,"l_z"] \arrow[l] \\
		x \arrow[d,swap,"p"] \arrow[r,"f"]  & y \arrow[d,"q"] & z \arrow[l,swap,"g"] \arrow[d,"\pi"] \\
		c \arrow[r] & c  & c \arrow[l]
	\end{tikzcd}
	\]
	where the maps $l_x$, $l_y$ and $l_z$ are the corresponding reflections given by the fact that our cospan lives in $\left(\Fib^{0}_{\bCC}\right)_{/c}$. We observe that the diagram above defines a left adjoint map in $\Fun(\Lambda^2_2,\bCC)_{\underline{c}}$, where $\underline{c}$ is the constant cospan on $c$. Taking oriented pullbacks, we obtain a left adjoint map in $\Over{\bCC}{c^{[1]}}$ between $x \orientedtimesrl_{y} z $ and  $\Free^0_{c}(p) \orientedtimesrl_{\Free^{0}_c(q)} \Free^0_{c}(\pi) \simeq \Free^{0}_{c^{[1]}}\left(x \orientedtimesrl_{y} z\right)$, where the later equivalence is a consequence of the fact that limits commute with one another. The claim now follows.
\end{proof}

\begin{corollary}
	For every 2-topos $\bCC$ we have functors
	\[
	\Fib^{0}_{/(-)}\colon \bCC^{\op} \to \LTTop, \quad \quad \Fib^{1}_{/(-)}\colon \bCC^{\coop} \to \LTTop.
	\]
\end{corollary}
\begin{proof}
	Combine \cref{prop:fibslice} and \cref{thm:basechangetheorem} below.
\end{proof} 

\section{Synthetic category theory in a 2-topos}\label{sec:synth}
In this section we develop some basic aspects of synthetic category theory in a 2-topos. In \cref{subsec:groupoids}, we study internal groupoids in a 2-topos and show that 2-topoi admit internal notions of groupoidal cores and geometric realization (cf.\  \cref{cor:cores} and \cref{prop:groupoidification}). In \cref{subsec:univalence}, we construct internal mapping objects and show that 2-topoi satisfy a directed analogue of the univalence axiom (cf.\ \cref{thm:Univalence}). In \cref{subsec:basechange} we develop the basics of a theory of partially lax Kan extensions internal to a 2-topos.
Finally, we study two-variable fibrations in \cref{subsec:yoneda}, which give rise to an internal version of Yoneda's lemma (cf.\ \cref{thm:Yoneda}). 
  
\subsection{Groupoids in a 2-topos}\label{subsec:groupoids}
Recall that the 1-category $\Spc$ of spaces sits reflectively inside the 2-category $\Cat$. In other words, the inclusion $\Spc\into\Cat$ admits a left adjoint $\lvert-\rvert$ that carries a category $\CC$ to its \emph{groupoidification} $\lvert \CC\rvert$. Moreover, the inclusion $\Spc\into\Cat^{\leq 1}$ admits a \emph{right} adjoint $(-)^\core$ sending a category $\CC$ to its groupoid core $\CC^\core$. In this section, we provide analogues of these result for a general 2-topos.

\begin{definition}\label{def:groupoids}
	Let $\bCC$ be a 2-topos. An object $c\in\bCC$ is said to be a \emph{$\bCC$-groupoid} if the presheaf $\bCC(-,c)$ takes values in $\Spc\subset\Cat$. We denote the full sub-2-category of $\bCC$ that is spanned by the $\bCC$-groupoids by $\Grpd(\bCC)$.
\end{definition}

\begin{proposition}\label{prop:characterisationGroupoids}
	Let $\bCC$ be a 2-topos, and let $c\in\bCC$ be an object. Then the following are equivalent:
	\begin{enumerate}
		\item $c$ is a $\bCC$-groupoid;
		\item the canonical map $c\to c^{[1]}$ is an equivalence;
		\item the inclusion $\Fib_{/c}^\epsilon\to\bCC_{/c}$ is an equivalence.
	\end{enumerate}
\end{proposition}
\begin{proof}
	The functor $\bCC(-, c)$ takes values in spaces precisely if the map
	\begin{equation*}
		\bCC(-, c)\to\Cat([1],\bCC(-,c))
	\end{equation*}
	is an equivalence. By \cref{prop:tensoring}, we may identify this map with
	\begin{equation*}
		\bCC(-, c)\to\bCC(-, c^{[1]}),
	\end{equation*}
	so that Yoneda's lemma implies that (1) and (2) are equivalent. If (2) holds, then for every map $p\colon x\to c$, the map $\eta_x\colon x\to \Free_c(x)$ must be an equivalence, which immediately implies (3). Conversely, if (3) is satisfied, then the map $c\to c^{[1]}$ is a morphism of $\epsilon$-fibrations. This already implies that this map must be an equivalence. In fact, by working representably, we can reduce to the case $\bCC=\Cat$. Now if $\CC$ is a 1-category and $\CC\to \CC^{[1]}$ is the canonical map (which sends $c\in\CC$ to $\id_c$), the fact that the (co)cartesian edges in $\CC^{[1]}$ are precisely those that are sent to equivalences by the evaluation functor $\ev_\epsilon$ immediately tells us that \emph{every} edge in $\CC$ must be an equivalence, so that $\CC$ is a space.
\end{proof}

\begin{proposition}\label{prop:groupoidification}
	For any 2-topos $\bCC$, the inclusion $\iota\colon\Grpd(\bCC)\into\bCC$ preserves small (partially lax and oplax) limits and $\kappa$-filtered colimits for some regular cardinal $\kappa$. Consequently, $\iota$ admits a left adjoint $\left|-\right|\colon\bCC\to\Grpd(\bCC)$.
\end{proposition}
\begin{proof}
	By the 1-categorical reflection theorem (see \cite[Theorem~6.2]{reflectionthm}) and \cref{lem:adjunctions1vs2dimensional}, the first claim immediately implies the second. To show the first claim, first note that by the very definition of $\bCC$-groupoids and the fact that the inclusion $\Spc\into\Cat$ admits a left adjoint, the sub-2-category $\Grpd(\bCC)\into\bCC$ is closed under small partially lax and oplax limits. Moreover, let $\kappa$ be a regular cardinal such that $\bCC$ arises as a reflective sub-2-category of $\PSh_{\Cat}(\bII)$ (for some small 2-category $\bII$) where the inclusion $i\colon\bCC\into\PSh_{\Cat}(\bII)$ preserves $\kappa$-filtered colimits (see \cref{thm:characterisationPresentableCategories}). Then then the inclusion $\Grpd(\bCC)\into\bCC$ is closed under $\kappa$-filtered colimits. In fact, in this case an object $c\in\bCC$ is a $\bCC$-groupoid precisely if the underlying presheaf in $\PSh_{\Cat}(\bII)$ takes values in spaces. Since $i$ preserves $\kappa$-filtered colimits and $\Spc\into\Cat$ preserves all strong colimits, this implies the claim.
\end{proof}

\begin{corollary}\label{cor:grppresentable}
	Let $\bCC$ be a 2-topos. Then $\Grp(\bCC)$ is presentable.\qed
\end{corollary}

\begin{lemma}\label{lem:grpdstrongcolim}
	Let $\bCC$ be a 2-topos. Then the canonical inclusion $\iota \colon \Grp(\bCC) \to \bCC$ preserves strong colimits.
\end{lemma}
\begin{proof}
	We invoke \cref{thm:giraud} to produce an adjunction
	\[
	L \colon \PSh_{\Cat}(\bII) \llra \bCC \colon R
	\]
	where $L$ preserves oriented pullbacks and $R$ is fully faithful.
	Unpacking the definitions, we see that $\Grp(\bCC)$ embeds fully faithfully into $\Grpd(\PSh_{\Cat}(\bII))=\PSh_{\Spc}(\bII)$. Now the inclusion $\PSh_{\Spc}(\bII)\into\PSh_{\Cat}(\bII)$ preserves strong colimits as they are computed object-wise and as $\Spc\subset\Cat$ is closed under strong colimits. Consequently, the claim follows once we show that the left adjoint $L$ restricts to a map $\Grpd(\Fun(\bSS^\op,\Cat))\to\Grpd(\bCC)$. This follows immediately from \cref{prop:characterisationGroupoids} since $L$ preserves oriented pullbacks and hence cotensors with $[1]$.
\end{proof}

\begin{corollary}\label{cor:cores}
	Let $\bCC$ be a 2-topos. Then the inclusion $\Grpd(\bCC)\into\bCC^{\leq 1}$ has a right adjoint $(-)^\core$, referred to as the \emph{groupoid core functor}.\qed
\end{corollary}

\begin{remark}\label{rem:emptycores}
	The behaviour of the groupoid core functor can be somewhat exotic in a general 2-topos, as we hope to illustrate in the following example. Let us consider the 2-topos $\bXX=\FUN(\Adj,\Cat)$ and note that 
	\[
	\Grp(\bXX)= \FUN(\Adj,\Spc) \simeq \Spc.
	\]
	In other words, an adjunction $F=(L\dashv R)\colon\CC\leftrightarrows\DD$ is an $\bXX$-groupoid precisely if both $\CC$ and $\DD$ are spaces (so that necessarily both $L$ and $R$ are equivalences). Now let us in particular set $F=(\id_{[0]}\dashv \id_{[0]})$, and let $\alpha\colon F\to G$ be a morphism in $\bXX$, where $G$ is given by an adjunction $(L\dashv R)\colon \AA\leftrightarrows\BB$. Such a datum necessarily means that the object $\alpha_+([0])\in\AA$ is a fixed point for the adjunction $L\dashv R$. Therefore, if $G$ is an adjunction with no fixed points (which can happen easily), then it follows that $G^{\core}= \varnothing$ is the initial functor. In other words, the groupoid core of an object in a 2-topos can be empty even if that object is not initial itself.
\end{remark}

\subsection{Directed univalence}\label{subsec:univalence}
In synthetic category theory, the axiom of directed univalence asserts that if $\CC$ and $\DD$ are synthetic categories, the (synthetic) groupoid of functors $\Fun(\CC,\DD)^\core$ can be identified with the (synthetic) mapping groupoid $\Hom_{\Omega}(\CC,\DD)$ of the universe $\Omega$. In this section, we establish this result for 2-topoi.

\begin{lemma}\label{lem:inthom}
		Let $\bCC$ be a 2-topos. Then for every object $c \in \bCC$, the functor $c\times -$ has a right adjoint $\underline{\bCC}(c,-) \colon \bCC \to \bCC$.
\end{lemma}
\begin{proof}
		Since $\bCC$ is 2-presentable it will suffice to show that $c \times -$ preserves  partially lax colimits. Since this functor can be rewritten as the composite
		\begin{equation*}
			\bCC\xrightarrow{t^\ast}\Over{\bCC}{c}\xrightarrow{t_!}\bCC
		\end{equation*}
		in which $t\colon c\to \ast$ is the unique map to the terminal object, the claim follows from \cref{prop:condussy} in light of the fact that $t$ is always an $\epsilon$-fibration since $\ast\in\bCC$ is a $\bCC$-groupoid.
	\end{proof}
	
	As a consequence of \cref{lem:inthom} and Yoneda's lemma, we find:
	\begin{proposition}\label{prop:inthom}
		Every 2-topos $\bCC$ is cartesian closed, in the sense that there is a bifunctor $\underline{\bCC}(-,-)\colon\bCC^\op\times\bCC\to\bCC$ that fits into an equivalence
		\begin{equation*}
			\bCC(-\times -, -)\simeq\bCC(-, \underline{\bCC}(-,-)).
		\end{equation*}
		We refer to this bifunctor as the \emph{internal mapping} functor.\qed
	\end{proposition}

	\begin{remark}\label{rem:cotensorinternal}
		It follows easily from the universal property of the internal mapping functor that for every category $t \in \Cat$ and every object $c \in \bCC$ the power $c^t$ is given by $\underline{\bCC}(t,c)$.
	\end{remark}

	\begin{proposition}
		Let $\bCC$ be a 2-topos, and let $p \colon a \to b$ be an $\epsilon$-fibration. Then, for every $x\in\bCC$, the morphism $p_* \colon \underline{\bCC}(x,a) \to \underline{\bCC}(x,b)$ is again an $\epsilon$-fibration.
	\end{proposition}
	\begin{proof}
		Since $\underline{\bCC}(x,-)$ is a right adjoint and therefore preserves oriented pullbacks, this is immediate.
	\end{proof}

	\begin{definition}
		Let $\bCC$ be a 2-topos and let $c \in \bCC$ be an object. Given a pair of morphisms $a,b \colon \ast \rightrightarrows c$, we denote by $\Hom_c(a,b)$ the oriented pullback
		\[\begin{tikzcd}
	{\Hom_c(a,b)} & \ast \\
	\ast & c
	\arrow[from=1-1, to=1-2]
	\arrow[from=1-1, to=2-1]
	\arrow["b", from=1-2, to=2-2]
	\arrow[Rightarrow,shorten <=15pt, shorten >=15pt, from=2-1, to=1-2]
	\arrow["a"', from=2-1, to=2-2]
\end{tikzcd}\]
		Note that $\Hom_c(a,b)$ is always a $\bCC$-groupoid, which we will call the \emph{mapping groupoid} between $a$ and $b$.
	\end{definition}

	\begin{theorem}[Univalence]\label{thm:Univalence}
		Let $\bCC$ be a 2-topos and let $x,y \in \bCC$.  Then we have a natural equivalence $\underline{\bCC}(x,y)^{\simeq} \simeq \Hom_{\Omega^{\epsilon,\kappa}}(x,y)$ for sufficiently large regular cardinals $\kappa$.
	\end{theorem}
	\begin{proof}
		We pick a cardinal $\kappa$ such that $x,y$ are $\kappa$-compact the universe $\Omega^{\epsilon,\kappa}$ exists. Then it follows that the unique maps $t_x \colon x \to \ast$ and $t_y\colon y\to\ast$ to the terminal object are $\kappa$-compact $\epsilon$-fibrations corresponding to maps $x, y \colon \ast \rightrightarrows \Omega^{\epsilon,\kappa}$. Now for every $s\in\Grpd(\bCC)$, we have a chain of natural equivalences,
		\[
			\Grp(\bCC)(s,\underline{\bCC}(x,y)^{\simeq}) \simeq \bCC(s \times x,y)^{\simeq} \simeq \Fib^{\epsilon}_{/s}(t_s^\ast x,t_s^\ast y)^{\simeq}\simeq \Grp(\bCC)(s, \Hom_{\Omega^{\epsilon,\kappa}}(x,y)),
		\]
		which together with Yoneda's lemma yields the claim.
	\end{proof}

  \subsection{Synthetic Kan extensions}\label{subsec:basechange}
  In this section, we establish a theory of \emph{synthetic Kan extension} internal to a 2-topos $\bCC$. That is, given a map $f\colon c\to d$ in $\bCC$, we establish internal left and right adjoints of the pullback map $f^\ast\colon\bCC(d,\Omega^{\epsilon,\kappa})\to \bCC(c,\Omega^{\epsilon,\kappa})$ (for a suitable cardinal $\kappa$). Moreover, we also provide an internal version of \emph{lax} Kan extensions. 
  
    \begin{proposition}\label{prop:leftkan}
    	Let $\bCC$ be a 2-topos a consider a morphism $f\colon c \to d$. Then the pullback functor
    	\[
    		f^* \colon \Fib_{/d}^{\epsilon} \to \Fib_{/c}^{\epsilon}
    	\]
    	is continuous.
    \end{proposition}
    \begin{proof}
    	 By means of the commutative square 
    	\[
    		\begin{tikzcd}
    			\Fib_{/d}^{\epsilon} \arrow[r, "f^\ast"] \arrow[d, hookrightarrow] & \Fib_{/c}^{\epsilon} \arrow[d, hookrightarrow] \\
    			\bCC_{/d} \arrow[r, "f^\ast"] & \bCC_{/c},
    		\end{tikzcd}
    	\]
    	the claim follows from the observation that the vertical maps create partially (op)lax limits and the lower horizontal map is continuous.
    \end{proof}

   \begin{theorem}\label{thm:basechangetheorem}
   	Let $\bCC$ be a 2-topos. Then we have the following adjunctions:
   	\begin{enumerate}
   		\item For every $c \in \bCC$, an adjunction $U_c \colon \Fib^{\epsilon}_{/c} \llra \bCC_{/c}: \Cofree_c$, where $U_c$ is the forgetful functor. We call $\Cofree_c$ the \emph{cofree} $\epsilon$-fibration functor.
   		\item For every  map $f \colon c \to d$ in $\bCC$, an adjunction $f_! \colon \Fib^{\epsilon}_{/c} \llra \Fib^{\epsilon}_{/d} \colon f^*$. We call $f_!$ the \emph{functor of left Kan extension} along $f$.
   		\item For every map $f \colon c \to d$ in $\bCC$, an adjunction $f^* \colon \Fib^{\epsilon}_{/d} \llra \Fib^{\epsilon}_{/c} \colon f_*$. We call $f_\ast$ the \emph{functor of right Kan extension} along $f$.
   	\end{enumerate}
   \end{theorem}
   \begin{proof}
   	Combine \cref{prop:cofree}, \cref{prop:rightkan}, \cref{prop:leftkan} together with \cref{prop:adjointFunctorTheoremLeft}, \cref{prop:adjointFunctorTheoremRight} and \cref{cor:Fib2Presentable}.
   \end{proof}

   \begin{lemma}\label{lem:kanextendingfree}
   		Let $\bCC$ be a 2-topos, and consider a morphism $f\colon c \to d$ in $\bCC$. Given $p \colon x \to c$, the canonical map $f_{!}(\Free^{\epsilon}_{c}(p)) \to \Free^{\epsilon}_{d}(f p)$ is an equivalence.
   \end{lemma}
   \begin{proof}
   	This follows from the observation that the pullback functor $f^\ast$ preserves $\epsilon$-fibrations upon passing to left adjoints.
   \end{proof}

   \begin{proposition}\label{prop:functkanext}
   	Let $L \colon \bCC \to \bDD$ be a morphism of 2-topoi and consider a morphism $f \colon a \to b$ in $\bCC$. Then we have commutative diagrams
   	\[
   		\begin{tikzcd}
   			\left(\Fib_{\bCC}^{\epsilon}\right)_{/a} \arrow[d,"f_!"] \arrow[r,"L_a"] & \left(\Fib_{\bDD}^{\epsilon}\right)_{/La} \arrow[d,"(Lf)_!"] \\
   			\left(\Fib_{\bCC}^{\epsilon}\right)_{/b} \arrow[r,"L_b"] & \left(\Fib_{\bDD}^{\epsilon}\right)_{/Lb},
   		\end{tikzcd} \enspace \enspace
   		\begin{tikzcd}
   			\left(\Fib_{\bDD}^{\epsilon}\right)_{/La} \arrow[d,"(Lf)_*"] \arrow[r,"R_a"] & \left(\Fib_{\bCC}^{\epsilon}\right)_{/a} \arrow[d,"f_*"] \\
   			\left(\Fib_{\bDD}^{\epsilon}\right)_{/Lb} \arrow[r,"R_b"] & \left(\Fib_{\bCC}^{\epsilon}\right)_{/b}
   		\end{tikzcd} 
   	\]
   	where $L_a$ (resp.\ its adjoint $R_a$) is the morphism of 2-topoi given in \cref{prop:fibslice}. 
   \end{proposition}
   \begin{proof}
   It is clear from our assumptions that we have commutative diagrams
   	\[
   		\begin{tikzcd}
   			\left(\Fib_{\bCC}^{\epsilon}\right)_{/b} \arrow[d,"f^*"] \arrow[r,"L_a"] & \left(\Fib_{\bDD}^{\epsilon}\right)_{/Lb} \arrow[d,"(Lf)^*"] \\
   			\left(\Fib_{\bCC}^{\epsilon}\right)_{/a} \arrow[r,"L_b"] & \left(\Fib_{\bDD}^{\epsilon}\right)_{/La},
   		\end{tikzcd} \enspace \enspace
   		\begin{tikzcd}
   			\left(\Fib_{\bDD}^{\epsilon}\right)_{/Lb} \arrow[d,"(Lf)^{*}"] \arrow[r,"R_a"] & \left(\Fib_{\bCC}^{\epsilon}\right)_{/b} \arrow[d,"f^*"] \\
   			\left(\Fib_{\bDD}^{\epsilon}\right)_{/La} \arrow[r,"R_b"] & \left(\Fib_{\bCC}^{\epsilon}\right)_{/a}
   		\end{tikzcd} 
   	\]
   	where the vertical maps are left (resp.\  right) adjoints by \cref{thm:basechangetheorem}. Therefore, the claim follows by passing to adjoints.
   \end{proof}

   \begin{theorem}\label{thm:kaninternal}
   	Let $\bCC$ be a 2-topos and let $f \colon a \to b$ be a morphism. Then for sufficiently large regular cardinals $\kappa$ we have:
   	\begin{enumerate}
   		\item An internal adjunction
   		\[
   		 	f_! \colon \underline{\bCC}(a, \Omega^{\kappa,\epsilon}) \llra  \underline{\bCC}(b, \Omega^{\kappa,\epsilon}) \colon : f^*.
   		\]
   		We call $f_!$ the \emph{internal functor of left Kan extension} along $f$.
   		\item An adjunction
   		\[
   		 	f^* \colon \underline{\bCC}(b, \Omega^{\kappa,\epsilon}) \llra  \underline{\bCC}(a, \Omega^{\kappa,\epsilon}) \colon f_*.
   		\]
   		We call $f_\ast$ the \emph{internal functor of right Kan extension} along $f$.
   	\end{enumerate}
   \end{theorem}
   \begin{proof}
   	We deal with the case of $f_!$ without loss of generality. We commence by noting that due to \cref{thm:defcheck} we have natural equivalences
   	\begin{equation*}\tag{$\star$}
   		\bCC\left(x, \underline{\bCC}(b, \Omega^{\kappa,\epsilon}) \right)\simeq \bCC(x \times b,\Omega^{\kappa,\epsilon})\simeq \Fib^{\epsilon,\kappa}_{/x \times b}.
   	\end{equation*}
     Enlarging $\kappa$ if necessary we can assume that for every $x \in \bCC$ the map  $\bCC(x, \underline{\bCC}(b, \Omega^{\kappa,\epsilon})) \to \bCC(x, \underline{\bCC}(a, \Omega^{\kappa,\epsilon})) $ given by post-composition with $f^*$ is a right adjoint. In order to conclude that $f^*$ admits a left adjoint, it suffices to check by \cite[Corollary 5.2.12]{AGH24} that for every $u \colon x \to y$ in $\bCC$ the commutative diagram
   	\[
   		\begin{tikzcd}
   			\bCC(y,\underline{\bCC}(b, \Omega^{\kappa,\epsilon})) \arrow[r] \arrow[d] & \bCC(y,\underline{\bCC}(a, \Omega^{\kappa,\epsilon})) \arrow[d] \\
   			\bCC(x,\underline{\bCC}(b, \Omega^{\kappa,\epsilon})) \arrow[r] & \bCC(x,\underline{\bCC}(a, \Omega^{\kappa,\epsilon}))
   		\end{tikzcd} 
   	\]
   	is horizontally adjointable. We invoke \cref{thm:giraud} to obtain an adjunction $L\colon \bXX=\FUN(\bSS^{\op},\Cat) \llra \bCC \colon R$ where $R$ is fully faithful and accessible. First, we observe that the functor $L_{s} \colon \left(\Fib^{\epsilon}_{\bXX}\right)_{/s} \to  \left(\Fib^{\epsilon}_{\bCC}\right)_{/Ls}$ is compatible with pullback and left Kan extension by \cref{prop:functkanext} and that the 2-morphism we wish to show is invertible in the mate square is obtained as a whiskering of units and counits. We conclude that we can reduce to $\bCC=\bXX$. Since $
	\left(\Fib^{\epsilon}_{\bXX}\right)_{/s}$ is equivalent to  $\left(\Fib^{\epsilon}_{\Cat}\right)_{/\Un_{\bSS}^{\epsilon}(s)}$, we can further reduce our problem to a computation which usual Kan extensions. The result now follows as an easy application of the point-wise formula given in \cref{sec:kancocompletion}.
   \end{proof}

We conclude this section with a discussion of synthetic \emph{lax} Kan extensions:

   \begin{definition}
   	Let $c \in \bCC$ where $\bCC$ is a 2-topos. A \emph{lax datum} is a collection $\Xi$ of morphisms in $\bCC$ with target $c$.
   \end{definition}

   \begin{definition}
   	Let $\Xi$ be a lax datum on $c \in \bCC$. We define $\Fib_{/c}^{\epsilon,\Xi}$ to be locally full sub-2-category of $\bCC_{/c}$ defined as follows:
   	\begin{itemize}
   		\item Objects are given by $\epsilon$-fibrations.
   		\item A map $\varphi \colon x \to y$ over $\bCC$ between $\epsilon$-fibrations belongs to $\Fib_{/c}^{\epsilon,\Xi}$ if for every $(a \to c) \in \Xi$ the pullback map
   		\[
   			x \times_{c} a \xrightarrow{\varphi_{|a}} y \times_{c} a \
   		\]
   		is a morphism of fibrations over $a$.
   	\end{itemize}
   	If $\Xi = \varnothing$ we will denote $\Fib_{/c}^{\epsilon,\Xi}$ as $\Fib_{/c}^{\arglax{\epsilon}}$.
   \end{definition}

   \begin{remark}
   By construction, we always have an inclusion $\Fib_{/c}^\epsilon\into\Fib_{/c}^{\epsilon,\Xi}$. In some cases, this inclusion is an equivalence, for example in the case where the identity on $c$ belongs to $\Xi$. 
   \end{remark}

   \begin{proposition}\label{prop:fiblaxkan}
   	 Let $\bCC$ be a 2-topos and let $\Xi$ be a lax datum on $c$. Then the inclusion
   	 \[
   	 	\Fib_{/c}^{\epsilon} \into \Fib^{\epsilon,\Xi}_{/c} 
   	 \]
   	 detects limits and colimits.
   \end{proposition}
   \begin{proof}
   	We verify the case of colimits, since the computation for limits is dual. Let $x=\colim_{i \in \bII} x_i \to c$ be the partially (op)lax colimit for a functor $q: \bII \to \Fib_{/c}^{\epsilon}$ where $(\bII,E)$ is a marked 2-category. Let $\overline{q}$ be an $E$-(op)lax cone over $q$ in $\Fib^{\epsilon,\Xi}_{/c}$. Since this is also a cone in $\bCC_{/c}$ and since the inclusion $\Fib_{/c}^{\epsilon} \into \bCC $ preserves colimits,  we obtain a map $\alpha \colon x \to y=\overline{q}(\ast)$ over $c$.

   	To finish the proof we need to show that $\alpha$ defines a morphism in $\Fib^{\epsilon,\Xi}_{/c}$. Given $f \colon a \to c$ in $\Xi$, we wish to show that the morphism
   	\[
   		f_{a} \colon x \times_{c} a \to y \times_{c} a
   	\]
   	is a map of $\epsilon$-fibrations over $a$. Since $f^*$ preserves colimits by \cref{prop:rightkan}, it follows that $x \times_c a \simeq \colim_{i \in \bII}x_i \times_{c} a$ so that $f_a$ can be identified by the map induced by the cone $\{x_i \times_c a \to y \times_c a\}_{i \in \bII}$. Note that by definition of our cone, each of these morphisms is a map of fibrations over $a$ so it follows that $f_a$ is a map of fibrations. Thus, we conclude that $x \to c$ is a colimit of $q$ in $\Fib^{\epsilon,\Xi}_{/c} $. 
   \end{proof}
    
    As a direct consequence of the previous propsition together with our adjoint functor theorems, we obtain the following theorem:

   \begin{theorem}\label{thm:laxkan}
   	Let $\bCC$ be a 2-topos and let $\Xi$ be a lax datum on $c$. Given a morphism $f \colon c \to d$ in $\bCC$, we denote by $f_{\Xi}^{*}$ the composite
   	\[
   		\Fib^{\epsilon}_{/d} \xrightarrow{f^*} \Fib^{\epsilon}_{/c} \to \Fib^{\epsilon,\Xi}_{/c}.
   	\]
   	Then we have:
   	\begin{enumerate}
   		\item An adjunction $f_{\Xi}^{*} \colon \Fib^{\epsilon}_{d} \llra \Fib^{\epsilon,\Xi}_{/c} \colon f_{\Xi,*}$. We call $f_{\Xi,\ast}$ the \emph{functor of right $\Xi$-$\epsilon$lax Kan extension} along $f$.
   		\item An adjunction $f_{\Xi,!} \colon \Fib^{\epsilon,\Xi}_{/c}  \llra \Fib^{\epsilon}_{d}  \colon f_{\Xi}^{*}$. We call $f_{\Xi, !}$ the \emph{functor of left $\Xi$-$\epsilon$lax Kan extension} along $f$.\qed
   	\end{enumerate}
   \end{theorem}

   \subsection{The Yoneda embedding}\label{subsec:yoneda}

   \begin{definition}
   	Let $\bCC$ be a 2-topos. Given $s,t \in \bCC$ we define the 2-category $\BFib^{\epsilon}_{/(s,t)}$ as the locally full sub-2-category of $\left(\Fib^{\epsilon}_{/s}\right)_{/s \times t}$ where:
   	\begin{enumerate}
   		\item Objects are given by morphisms of $\epsilon$-fibrations $p \colon x \to s \times t$ over $s$, such that the composite $x \to s \times t \to t$ defines a $(1-\epsilon)$-fibration satisfying the following property:
   		\begin{itemize}
   			\item[A1)] The adjunction $l_x \colon \Free_{t}^{1-\epsilon}(x) \llra x:\eta_x$ over $t$ witnessing $x \to t$ as an $(1-\epsilon)$-fibration defines an adjunction in $\Fib^{\epsilon}_{/s}$
   				\[\begin{tikzcd}
				{\Free_{t}^{1-\epsilon}(x) } && x \\
								& s
				\arrow["{l_x}"', shorten <=5pt, shorten >=5pt, from=1-1, to=1-3]
				\arrow["\pi"', from=1-1, to=2-2]
				\arrow["{\eta_x}"', shift right=2, shorten <=5pt, shorten >=5pt, from=1-3, to=1-1]
				\arrow[from=1-3, to=2-2]
				\end{tikzcd}\]
   			where $\pi$ is the $\epsilon$-fibration obtained as the composite $\Free_{t}^{1-\epsilon}(x) \to x \to s$.
   		\end{itemize}
   		\item Morphisms are given by commutative triangles  $x \to y \to s \times t$ over $s$, such that the associated composite  $x \xrightarrow{f} y \to t$ defines a morphism of $(1-\epsilon)$-fibrations over $t$ which makes the following diagram 
   		\[
   			\begin{tikzcd}
   				\Free_{t}^{1-\epsilon}(x) \arrow[r] \arrow[d,"\ell_x"] & \Free_{t}^{1-\epsilon}(y) \arrow[d,"\ell_y"] \\
   				 x \arrow[r] & y
   			\end{tikzcd}
   		\]
   		commute over $s$.
   	\end{enumerate}
   		We call  $\BFib^{\epsilon}_{/(s,t)}$ the 2-category of $\epsilon$-bifibrations.
   	\end{definition}
   	\begin{lemma}\label{lem:bifibfunctor}
   		Let $\bCC$ be a 2-topos and let $t \in \bCC$. Then we have functors 
   		\[
   			\BFib^{0}_{/(-,t)} \colon \bCC^{\op} \to \widehat{\Cat}_2,\enspace s \mapsto \BFib^{0}_{/(s,t)},
   		\]
   		\[
   			\BFib^{1}_{/(-,t)} \colon \bCC^{\coop} \to \widehat{\Cat}_2,\enspace s \mapsto \BFib^{1}_{/(s,t)}.
   		\]
   	\end{lemma}
   	\begin{proof}
   		We observe that since $\Fib^{\epsilon}_{/-}$ is a functor and our choice of slicing object $(-) \times t \to (-)$, is functorial we only need to show that $\BFib^{\epsilon}_{_{/(-,t)}}$ is stable under the functoriality of $\Fib^{\epsilon}_{/-}$.

   		Let $\varphi \colon q \to s$. We claim that $\varphi^*x \to q$ defines an object in $\BFib^{\epsilon}_{/(q,t)}$. To this end we consider the commutative diagram, consisting of pullback squares,
   		  \[\begin{tikzcd}
		{\Free_{t}^{1-\epsilon}(\varphi^*x)} & {\Free_{t}^{1-\epsilon}(x)} & {} \\
		{\varphi^* x} & x & {} \\
		q & s
		\arrow[from=1-1, to=1-2]
		\arrow["{\ell_{\varphi^*x}}"',swap, from=1-1, to=2-1]
		\arrow["{\ell_x}", from=1-2, to=2-2]
		\arrow[from=2-1, to=2-2]
		\arrow[from=2-1, to=3-1]
		\arrow[from=2-2, to=3-2]
		\arrow[from=3-1, to=3-2]
		\end{tikzcd}\]
		and observe that by functoriality of the pullback functor we obtain the desired adjunction over $q$. It is easy to see that since the original adjunction $\ell_x \colon \Free_{t}^{1-\epsilon}(x) \llra x:\eta_x$ is an adjunction over $t$, so will the adjunction obtained via pullback. 
		Let $f \colon x \to y$ be a morphism in $ \BFib^{\epsilon}_{/(s,t)}$, with associated commutative diagram 
	    \[
   			\begin{tikzcd}
   				\Free_{t}^{1-\epsilon}(x) \arrow[r] \arrow[d,"\ell_x"] & \Free_{t}^{1-\epsilon}(y) \arrow[d,"\ell_y"] \\
   				 x \arrow[r] & y.
   			\end{tikzcd}
   		\]
   		Our construction above guarantees that after pullback along $\varphi$, we obtain a commutative diagram
   		\[
   			\begin{tikzcd}
   				\Free_{t}^{1-\epsilon}(\varphi^{*}x) \arrow[r] \arrow[d,"\ell_{\varphi^{*}x}"] & \Free_{t}^{1-\epsilon}(\varphi^{*}y) \arrow[d,"\ell_{\varphi^{*}y}"] \\
   				 \varphi^{*}x\arrow[r] & \varphi^{*}y
   			\end{tikzcd}
   		\]
   		with commutes both over $q$ and over $t$. Now let $\Xi \colon \alpha \to \beta$ be a 2-morphism between 1-morphisms $\alpha,\beta \colon q \to s$. To finish the proof we must show that the components of the associated natural transformation land in $\BFib^{\epsilon}_{/(q,t)}$. Let us assume without loss of generality that $\epsilon=0$. Appealing to functoriality on 2-morphisms, we obtain a commutative diagram of $0$-fibrations over $q$,
   		 \[
   		 	\begin{tikzcd}
   		 		\Free_{t}^{1}(\alpha^*x) \arrow[d] \arrow[r] & \Free_{t}^{1}(\beta^*x) \arrow[d] \\
   		 		\alpha^*x \arrow[r] & \beta^* x
   		 	\end{tikzcd}
   		 \]
   		 where the vertical maps are the morphisms witnessing $\alpha^*x \to t$ and $\beta^* x \to t$ as $1$-fibrations. The result now follows.
   	\end{proof}

   	\begin{proposition}\label{prop:bifibsheaf}
   		For a 2-topos $\bCC$, the two functors
   		\[
   			\BFib^{0}_{/(-,t)} \colon \bCC^{\op} \to \widehat{\Cat}_2,\enspace s \mapsto \left(\BFib^{0}_{/(s,t)}\right),
   		\]
   		\[
   			\BFib^{1}_{/(-,t)} \colon \bCC^{\coop} \to \widehat{\Cat}_2,\enspace s \mapsto \left(\BFib^{1}_{/(s,t)}\right),
 		\]
   		from \cref{lem:bifibfunctor} send partially $\epsilon$lax colimits of partially oplax limits.
   	\end{proposition}
   	\begin{proof}
   		We deal with the case $\epsilon=0$ without loss of generality. Let $q: \bII \to \bCC$ where $(\bII,E)$ is a marked 2-category and let $\overline{q}: \bII^{\colimcone}_{\elax} \to \bCC$ be an $E$-lax colimit cone. We need to show that the colimit functor
   		\[
   			\lim_{\bII^{\op}}^{E-\oplax}\BFib^{0}_{/(q(-),t))} \to (\Fib^{0}_{/s})_{/q(\ast) \times t}
   		\]
   		factors through $\BFib^{0}_{/(s,t)}$. By assumption, we have a compatible family of adjunctions $\Free_{t}^{1-\epsilon}(x_i) \llra x_i$ over $q(i)$. Using fibrational descent and the fact that free fibrations commute with colimits it follows that we get an adjunction $\Free_{t}^{1-\epsilon}(x) \llra x$ over $q(\ast)$. Moreover, F2) in \cref{def:fibdescent} guarantees that $\ell_x \colon \Free_{t}^{1-\epsilon}(x) \to x$ (resp.\ $\eta_x \colon x \to \Free_{t}^{1-\epsilon}(x)$) is a morphism of fibrations.

   		It is clear by construction that this adjunction is also an adjunction over $t$, so it follows that the colimit $x \to q(\ast)$ lands in $\BFib^{0}_{/(q(\ast),t)}$. The claim about morphisms is analogous and left to the reader. 
   	\end{proof}

   	\begin{definition}
   		Let $\kappa$ be a regular cardinal. We say that an $\epsilon$-bifibration is $\kappa$-compact if it is underlying $\epsilon$-fibration is $\kappa$-compact. We denote by $\BFib^{\epsilon,\kappa}_{/(s,t)}$ the full-sub-2-category on $\kappa$-compact bifibrations.
   	\end{definition}

   	\begin{lemma}\label{lem:kappabifib}
   		Let $\kappa$ be a regular cardinal. Then the 2-category $\BFib^{\epsilon,\kappa}_{/(s,t)}$ is essentially small, provided that $\kappa$ fullfils the following conditions:
   		\begin{enumerate}
   		 	\item The object $t$ is $\kappa$-compact.
   		 	\item The terminal object is $\kappa$-compact.
   		 	\item The set of $\kappa$-compact objects is stable under oriented pullbacks. 
   		 \end{enumerate} 
   	\end{lemma}
   	\begin{proof}
   		The size of our 2-category of $\kappa$-compact $\epsilon$-bifibrations is bounded by the size of the class of maps $x \to s \times t$. If $s$ is $\kappa$-compact our assumptions guarantee that \emph{both} $x$ and $s \times t$ are themselves $\kappa$-compact which implies the result. More generally, we use \cref{prop:bifibsheaf} to adapt the argument given in \cref{lem:smallkappafib}.
   	\end{proof}

   	\begin{proposition}\label{prop:compactbifibsheaf}
   		Let $\bCC$ be a 2-topos and let $t \in \bCC$. Given a regular cardinal $\kappa$ satisfying the hypothesis of \cref{lem:kappabifib}, we have functors
   		\[
   			\BFib^{0,\kappa}_{/(-,t)} \colon \bCC^{\op} \to \CCat, \quad \quad   \BFib^{0,\kappa}_{/(-,t)} \colon \bCC^{\coop} \to \CCat
   		\]
   		which send partially $\epsilon$lax colimits to partially oplax limits.
   	\end{proposition}
   	\begin{proof}
   		The question regardness smallness was already studied in \cref{lem:kappabifib} and the functoriality follows \cref{lem:bifibfunctor} together with our definition of $\kappa$-compactness of bifibrations. The only thing to check is that our discussion in \cref{prop:bifibsheaf} restrict to $\kappa$-compact bifibrations. This is clear from our definitions together with \cref{prop:compactsheaf}.
   	\end{proof}

   	\begin{corollary}\label{cor:internalfib}
   		Let $\bCC$ be a 2-topos and $t \in \bCC$. Then for sufficiently large $\kappa$ it follows that that the functor $\left(\BFib^{\epsilon,\kappa}_{/(-,t)}\right)^{\leq 1}$ is represented by an object $\underline{\Fib}^{1-\epsilon,\kappa}_{/t}\in\bCC$.\qed
   	\end{corollary}

   	\begin{definition}
   		The map $\ev_{1-\epsilon} \colon c^{[1]} \to c$ defines an object $\BFib^{\epsilon}_{/(c,c)}$ (for a sufficiently large regular cardinal $\kappa$). Invoking \cref{cor:internalfib} we obtain morphisms
   		\[
   			\mathcal{Y}_c \colon c \to \underline{\Fib}^{1,\kappa}_{/c}, \enspace \mathcal{Y}^{\co}_c \colon c \to \underline{\Fib}^{0,\kappa}_{/c}
   		\]
   		which we call the Yoneda embedding (resp.\ coYoneda embedding).
   	\end{definition}


   	\begin{remark}
   		We would like to point out that if we specialize to $\bCC=\Cat$ then $\underline{\Fib}_{/c}^0 \simeq \Cocart(c)^{\op}$, so that the coYoneda embedding can be identified with a map $\mathcal{Y}^{\co}_c \colon c^\op \to \Cocart(c)$.
   	\end{remark}

   	\begin{theorem}\label{thm:Yoneda}
   		The Yoneda (resp.\ coYoneda) embedding $\mathcal{Y}_c \colon c \to \underline{\Fib}^{1,\kappa}_{/c}$ is fully faithful.
   	\end{theorem}
   	\begin{proof}
   		We will give a proof for the Yoneda embedding, since the proof for the coYoneda embedding is formally dual. For the rest of the proof we will abuse notation by denoting $\left(\BFib^{0}_{(-,c)}\right)^{\leq 1}$ simply as $\BFib^{0}_{(-,c)}$.

   		 We wish to show that we have a pullback diagram
   		\[
   			\begin{tikzcd}
   				c^{[1]} \arrow[d] \arrow[r] & \left(\underline{\Fib}_{/c}^{1,\kappa}\right)^{[1]} \arrow[d] \\
   				c \times c \arrow[r] & \underline{\Fib}_{/c}^{1,\kappa} \times \underline{\Fib}_{/c}^{1,\kappa}.
   			\end{tikzcd}
   		\]
   		 More generally and regardless of cardinality assumptions, we will show that we have a pullback diagram of functors $\bCC^{\op} \to \CAT$
   		\[
   			\begin{tikzcd}
   				\bCC(-,c^{[1]}) \arrow[d] \arrow[r] & \left(\BFib^{0}_{/(-,c)}\right)^{[1]} \arrow[d] \\
   				\bCC(-,c \times c) \arrow[r] & \BFib^{0}_{(-,c)} \times \BFib^{0}_{/(-,c)}.
   			\end{tikzcd}
   		\]
   		We let $P$ denote the pullback of the diagram above and let $\Xi \colon \bCC(-, c^{[1]}) \to P$ be associated morphism. Our goal is to show that $\Xi$ is an equivalence. We commence by observing:
   		\begin{itemize}
   			\item The category $P(a)$ is given by a pullback diagram
   			  \[
   			  	\begin{tikzcd}
   			  		P(a) \arrow[r] \arrow[d] & \BFib^{0}_{/(a \times [1],c)} \arrow[d] \\
   			  	\bCC(a,c) \times \bCC(a,c) \arrow[r] & \BFib^{0}_{/(a,c)} \times \BFib^{0}_{/(a,c)}
   			  	\end{tikzcd}
   			  \]
   			  where the right-most vertical map is given by taking pullbacks along the maps $a \times \{i\} \to a \times [1]$ for $i\in \{0,1\}$. We now look at the map $\bCC(a,c) \to \BFib^{0}_{/(a,c)}$ which is given by
   			  \[
   			  	\bCC(a,c) \times [0] \xrightarrow{\id \times s} \bCC(a,c) \times \BFib^{0}_{/(c,c)} \to \BFib^{0}_{/(a,c)}
   			  \]
   			  where the map $s$ selects the object $c^{[1]} \in \BFib^{0}_{/(c,c)}$ and the second functor is simply induced by the functoriality studied in \cref{lem:bifibfunctor}.
   			  \item Note that since $a \times [1]$ is given by the lax colimit of the constant functor $[1] \to \bCC$ on $a$, we obtain an equivalence
   		      \[
   			     \BFib^{0}_{/(a \times [1],c)} \simeq \left(\BFib^{0}_{/(a,c)}\right)^{[1]}.
   			  \]
   		\end{itemize}
   		We are now ready to show that for every $a \in \bCC$ the induced map $\Xi_a \colon \bCC(a,c)^{[1]} \to P(a)$ is an equivalence.
   		\begin{itemize}
   			\item \textbf{Essential surjectivity:}  By construction an object in $P(a)$ is given by a pair of maps $f,g \colon a \to c$ and a map $\varphi \colon \Free^{1}_{c}(f) \to \Free^{1}_{c}(g)$ such that:
   			\begin{itemize}
   				\item The map $\varphi$ commutes with both projections to $a$ and to $c$.
   				\item The map $\varphi$ is a map of $0$-fibrations over $a$ and a map of $1$-fibrations over $c$.
   			\end{itemize}
   			We appeal to the free fibration adjunction (cf.\ \cref{prop:freeFibrationIsFreeFibration}) to see that the map $\varphi$ can be identified with a map $\hat{\varphi} \colon a \to \Free^1_{c}(g)$ in $\bCC_{/c}$. We can further use the adjunction between the lax and strong slice (\cref{prop:semiLaxPullbackRightAdjoint}) to see that $\hat{\varphi}$ corresponds to a diagram of the form 
   			\[
   				\begin{tikzcd}
				a & a \\
				& c
			\arrow["u", from=1-1, to=1-2]
			\arrow[""{name=0, anchor=center, inner sep=0}, "f"', from=1-1, to=2-2]
			\arrow["g", from=1-2, to=2-2]
			\arrow[shorten >=2pt, Rightarrow, from=0, to=1-2].
			\end{tikzcd}
   			\]
   			The fact that $\varphi$ commutes with the projection to $a$ guarantees that $u=\id$. We conclude that the objects of $P(a)$ are given by morphisms $[1] \to \bCC(a,c)$, so by objects of $\bCC(a,c)^{[1]}$. Unpacking the definitions, we conclude that $\Xi_a$ is essentially surjective.
   			\item \textbf{Fully-faithfulness:} Let $\varphi_i \colon \colon \Free^{1}_{c}(f_i) \to \Free^{1}_{c}(g_i)$ for $i \in \{0,1\}$ be a pair of objects in $P(a)$. Further unraveling reveals that we have a pullback diagram
   			\[
   				\begin{tikzcd}
   					P(a)(\varphi_0,\varphi_1) \arrow[r] \arrow[d] & \BFib^{0}_{/(a,c)}( \Free^{1}_{c}(f_0) , \Free^{1}_{c}(f_1) ) \arrow[d,"\varphi_1 \circ -"] \\
   					\BFib^{0}_{/(a,c)}( \Free^{1}_{c}(g_0), \Free^{1}_{c}(g_1)  ) \arrow[r,"-\circ \varphi_0"] & \BFib^{0}_{/(a,c)}( \Free^{1}_{c}(f_0), \Free^{1}_{c}(g_1)  ). 
   				\end{tikzcd}
   			\]
   			The same argument as above shows that we have natural equivalences
   			\[
   				\BFib^{0}_{/(a,c)}( \Free^{1}_{c}(u), \Free^{1}_{c}(v) )\simeq \bCC(a,c)(u,v).
   			\]
   			Finally, let $\psi_0,\psi_1$ be the objects in $\bCC(a,c)^{[1]}$ corresponding to $\varphi_0$ and $\varphi_1$, respectively. Then our pullback diagram implies that we have an equivalence
   			\[
   			 	\bCC(a,c)^{[1]}(\psi_0,\psi_1) \simeq P(a)(\varphi_0,\varphi_1), 
   			 \] 
   			 which is induced by $\Xi_a$. Consequently, it follows that $\Xi_a$ is fully faithful.
   		\end{itemize}
       We conclude that we have the desired pullback diagram. To finish the proof we observe that once we fix a sufficiently large cardinal $\kappa$, the map $\bCC(-,c) \to \BFib^{0}_{/(-,c)}$ factors through the sub-functor $\BFib^{0,\kappa}_{/(-,c)}$. We thus obtain a sequence of maps
       \[
        	\bCC(-,c^{[1]})  \to \BFib^{0,\kappa}_{/(-,c)} \to \BFib^{0}_{/(-,c)}
       \]
       where both the right map and the composite  map are fully faithful. The result follows.
   	\end{proof}



%% file: localic.tex
\section{$1$-Localic 2-topoi}\label{sec:localic}
The goal of this chapter is to study the relationship between 2-topoi and 1-topoi. Recall from \cref{subsec:groupoids} that if $\bCC$ is a 2-topos, we obtain an associated 1-category $\Grpd(\bCC)$ of $\bCC$-groupoids. We will show in \cref{subsec:underlying1topos} that this 1-category is a 1-topos. Moreover, the assignment $\bCC\mapsto \Grpd(\bCC)$ defines a functor $\LTTop\to\LTop$ from the 2-category of 2-topoi to the 2-category of 1-topoi. The main result of this chapter, which we establish in \cref{subsec:localicReflection}, is to show that this functor admits a fully faithful left adjoint $\Shv_{\Cat}(-)$ that carries a 1-topos $\CC$ to the 2-topos $\Shv_{\Cat}(\bCC)$ of \emph{sheaves} on $\bCC$, which is by definition the 2-category of limit-preserving categorical presheaves on $\bCC$. We will refer to a 2-topos that is contained in the essential image of this functor as being \emph{1-localic}. Thus, our result implies that every 2-topos admits a universal approximation by a 1-localic 2-topos. 

\subsection{The underlying 1-topos of a 2-topos}\label{subsec:underlying1topos}
Recall from \cref{def:groupoids} that if $\bCC$ is a 2-topos, we let $\Grpd(\bCC)$ be the full sub-2-category of $\bCC$ that is spanned by the \emph{1-truncated} objects, i.e.\ by those objects $x\in\bCC$ for which $\bCC(-,x)$ takes values in $\SS\subset\Cat$. By its very construction, this defines a 1-category. We now find:

\begin{theorem}\label{thm:underlying1topos}
	Let $\bCC$ be a 2-topos. Then $\Grpd(\bCC)$ is a 1-topos.
\end{theorem}
\begin{proof}
	We start by noting that since $\Grpd(\bCC)$ presentable \cref{cor:grppresentable}, the only thing left to show is that $\Grpd(\bCC)$ satisfies descent.
	
	Let $q: \II \to \Grpd(\bCC)$ be a diagram with colimit cone $\overline{q}$. Note that by \cref{lem:grpdstrongcolim}, this is also a colimit in $\bCC$. Now consider the commutative diagram
	\[
	\begin{tikzcd}
		\Grpd(\bCC)_{/\overline{q}(\ast)} \arrow[r, hookrightarrow] \arrow[d] & \left(\Fib^0_{/\overline{q}(\ast)}\right)^{\leq 1} \arrow[d,"\simeq"] \\
		\lim_{\II^\op}\Grpd(\bCC)_{/q(-)} \arrow[r] & \lim_{\II^\op}\left(\Fib^0_{/q(-)}\right)^{\leq 1}
	\end{tikzcd}
	\]
	in which the horizontal functors are fully faithful and the right vertical map is an equivalence on account of fibrational descent in $\bCC$. To finish the proof, we need to show that the above diagram above is a pullback, which amounts to showing that if  $x\to \overline{q}(\ast)$ is a map in $\bCC$ (which is automatically an $\epsilon$-fibration by part (3) in \cref{prop:characterisationGroupoids}) such that for each  $i \in \bII$ the object $x_i=x\times_{q(i)}\overline{q}(\ast)$ is a $\bCC$-groupoid, then $x$ is a $\bCC$-groupoid as well. Since we can recover $x$ as the colimit of the diagram $\II\to \bCC,~i\mapsto x_i$, this is a consequence of \cref{lem:grpdstrongcolim}.
\end{proof}

\begin{remark}\label{rem:restrictionofmapstogrp}
	Let $L \colon \bCC \to \bDD$ be a morphism of 2-topoi. Then the same argument as the one given in \cref{lem:grpdstrongcolim} shows that $L$ preserves groupoid objects and thus induces a colimit-preserving functor $\hat{L} \colon \Grpd(\bCC) \to \Grpd(\bDD)$. Moreover, since every pullback square in $\Grpd(\bCC)$ is an oriented pullback in $\bCC$ and since the terminal object is always a $\bCC$-groupoid, the map $\hat{L}$ is left exact and therefore defines a morphism of $1$-topoi
\end{remark}

\begin{definition}\label{def:Grpd2Top1Top}
	We denote by $\Grpd(-) \colon \LTTop \to \LTop$ the functor of 2-categories which assigns to every 2-topos $\bCC$ its underlying 1-topos $\Grpd(\bCC)$ of $\bCC$-groupoids.
\end{definition}

\subsection{$1$-Localic reflection}\label{subsec:localicReflection}
The goal of this section is to construct a fully faithful left adjoint to the functor $\Grpd(-)\colon\LTTop\to\LTop$. We begin with the following definition:

\begin{definition}\label{def:sheaves1topos}
	Let $\AA$ be a 1-topos. We define the category of \emph{sheaves on }$\AA$, denoted by $\Shv_{\Cat}(\AA)$,
	 as the full sub-2-category of $\Fun(\AA^\op,\Cat)$ consisting of limit-preserving functors.
\end{definition}

\begin{remark}\label{rem:SheavesAsKappaSheaves}
	Let $\AA$ be a $1$-topos, and let $\kappa$ be a regular cardinal such that $\AA$ is $\kappa$-accessible. Let $\iota\colon \AA^\kappa\into\AA$ be the full subcategory of $\kappa$-compact objects. By \cref{thm:laxCocompletion}, we obtain a fully faithful functor
	\begin{equation*}
		(h_{\AA^\kappa})_!(h_{\AA}\iota)\colon \PSh_{\Cat}(\AA^\kappa)\into\PSh_{\Cat}(\AA).
	\end{equation*}
	We claim that this functor restricts to an equivalence
	\begin{equation*}
		\Shv_{\Cat}^\kappa(\AA^\kappa)\simeq \Shv_{\Cat}(\AA),
	\end{equation*}
	where the left-hand side denotes the full subcategory of $\PSh_{\Cat}(\AA^\kappa)$ that preserves $\kappa$-small limits. In fact, it is well-known that this is the case on the underlying $1$-categories, which is sufficient to deduce the claim. 
\end{remark}

\begin{proposition}\label{prop:sheaves2topos}
	Let $\AA$ be a 1-topos. Then the category $\Shv_{\Cat}(\AA)$ is a 2-topos.
\end{proposition}
\begin{proof}
	Choose a cardinal $\kappa$ such that $\AA$ is $\kappa$-accessible, and let $\iota\colon\AA^\kappa\subset\AA$ be the full subcategory of $\kappa$-compact objects. Assume furthermore (by enlarging $\kappa$ if necessary) that $\AA^\kappa$ is closed under finite limits in $\AA$. By \cref{rem:SheavesAsKappaSheaves}, we obtain an equivalence $\Shv_{\Cat}(\AA)\simeq\Shv_{\Cat}^\kappa(\AA^\kappa)$. Thus, it suffices to show that $\Shv_{\Cat}^\kappa(\AA^\kappa)$ is a 2-topos. Note that $\Shv_{\SS}^\kappa(\AA^\kappa)\simeq \Ind_\kappa(\AA^\kappa)\simeq\AA$, and we have an adjunction
	\begin{equation*}
	(l^\prime\dashv i^\prime)\colon \PSh_{\SS}(\AA^\kappa)\leftrightarrows\Shv_{\SS}^\kappa(\AA^\kappa)
	\end{equation*}
	where $i$ preserves $\kappa$-filtered colimits and where $l$ is left exact since $\AA^\kappa$ is closed under finite limits in $\AA$ (see \cite[Proposition~6.1.5.2]{LurieHTT}). Furthermore, observe that the inclusion $i_{\leq 1}\colon\Shv_{\Cat}^\kappa(\AA^\kappa)_{\leq 1}\into\PSh_{\Cat}(\AA^\kappa)_{\leq 1}$ can be identified with the restriction of
	\begin{equation*}
		i^\prime_\ast\colon\Fun(\Delta^\op,\Shv_{\SS}^\kappa(\AA^\kappa))\into\Fun(\Delta^\op,\PSh_{\SS}(\AA^\kappa))
	\end{equation*}
	to complete Segal objects. Consequently, we deduce from \cref{lem:adjunctions1vs2dimensional} that we have an adjunction
	\begin{equation*}
		(l\dashv i)\colon \PSh_{\Cat}(\AA^\kappa)\leftrightarrows\Shv_{\Cat}^\kappa(\AA^\kappa)
	\end{equation*}
	in which the inclusion $i$ preserves $\kappa$-filtered colimits
	and where $l_{\leq 1}$ can be identified with the restriction of
	\begin{equation*}
		l^\prime_\ast\colon\Fun(\Delta^\op,\PSh_{\SS}(\AA^\kappa))\to\Fun(\Delta^\op,\Shv_{\SS}^\kappa(\AA^\kappa))
	\end{equation*}
	to complete Segal objects.
	Using \cref{prop:inducing2toposstructure}, it now suffices to verify that $l$ preserves oriented pullbacks. Note that as limits of simplicial objects are computed level-wise, the functor $l^\prime_\ast$ is preserves finite limits. Hence $l$ preserves strong pullbacks. Consequently, the proof is finished once we verify that for every $x\in\PSh_{\Cat}(\AA^\kappa)$, the canonical map $l(x^{[1]})\to l(x)^{[1]}$ is an equivalence. Upon identifying this map as a morphism of complete Segal objects in $\Shv_{\SS}^\kappa(\AA^\kappa)$ and using the Segal conditions, it will be enough to see that this map induces equivalences in $\Shv_{\SS}^\kappa(\AA^\kappa)$ on level $0$ and level $1$. But if $x_{\bullet}$ denotes the complete Segal object in $\PSh_{\SS}(\AA^\kappa)$ that is associated with $x$, then we compute $(x^{[1]})_0= x_1$ and $(x^{[1]})_1= (x^{[1]\times [1]})_0=x_2\times_{x_0} x_2$. Thus the claim also follows from $l^\prime$ preserving finite limits.
\end{proof}

\begin{proposition}\label{prop:UMPSheaves}
	Let $\bCC$ be a 2-topos, and let $\AA$ be a $1$-topos. Then the Yoneda embedding $h_{\AA}\colon\AA\into\Shv_{\Cat}(\AA)$ induces an equivalence $\AA\simeq\Grpd(\Shv_{\Cat}(\AA))$, and precomposition with $h_{\AA}$ gives rise to an equivalence
	\begin{equation*}
		\LTTop(\Shv_{\Cat}(\AA),\bCC)\simeq \LTop(\AA,\Grpd(\bCC)).
	\end{equation*}
\end{proposition}
\begin{proof}
	The first claim is clear on account of the fact that we have $\Grpd(\Shv_{\Cat}(\AA))\simeq\Shv_{\SS}(\AA)$, together with the $1$-categorical adjoint functor theorem. As for the second, claim, pick a regular cardinal $\kappa$ so that the inclusion $\iota\colon\AA^\kappa\subset\AA$ is closed under finite limits and $\AA$ is $\kappa$-accessible. By \cref{rem:SheavesAsKappaSheaves}, we may identify $\Shv_{\Cat}(\AA)\simeq\Shv_{\Cat}^{\kappa}(\AA^\kappa)$. As in the proof of \cref{prop:sheaves2topos}, we obtain an adjunction
	\begin{equation*}
		(l\dashv i)\colon \PSh_{\Cat}(\AA^\kappa)\leftrightarrows\Shv_{\Cat}^\kappa(\AA^\kappa)
	\end{equation*}
	in which $l$ is a morphism of 2-topoi and $i$ is fully faithful and preserves $\kappa$-filtered colimits. We now obtain a commutative diagram
	\begin{equation*}
		\begin{tikzcd}
			\LTTop(\Shv_{\Cat}(\AA),\bCC)\arrow[r, "h_{\AA}^\ast"]\arrow[d, "l^\ast", hookrightarrow] & \LTop(\AA,\Grpd(\bCC))\arrow[d, "\iota^\ast", hookrightarrow]\\
			\LTTop(\PSh_{\Cat}(\AA^\kappa), \bCC)\arrow[r, hookrightarrow, "h_{\AA^\kappa}^\ast"] & \Cat(\AA^\kappa, \Grpd(\bCC)),
		\end{tikzcd}
	\end{equation*}
	which already implies that $h_{\AA}^\ast$ is fully faithful since all of the other maps are fully faithful. To complete the proof, it now suffices to verify that every map of $1$-topoi $f\colon \AA\to\Grpd(\bCC)$ is in the essential image of $h_{\AA}^\ast$. Now $f\iota\colon \AA^\kappa\to \Grpd(\bCC)\into\bCC$ preserves oriented pullbacks, hence we deduce from \cref{prop:preGiraud} that its Yoneda extension $\PSh_{\Cat}(\AA^\kappa)\to\bCC$ is a morphism of 2-topoi. By \cref{rem:universalAdjointFunctorTheorem}, its right adjoint is given by the composition
	\begin{equation*}
	\bCC\xrightarrow{h_{\bCC}}\PSh_{\Cat}(\bCC)\xrightarrow{(f\iota)^\ast}\PSh_{\Cat}(\AA^\kappa).
	\end{equation*}
	In particular, since $f\iota$ preserves $\kappa$-small colimits (see \cref{lem:grpdstrongcolim}), this right adjoint takes values in  $\Shv_{\Cat}^\kappa(\AA^\kappa)$. Since the inclusion $\Shv_{\Cat}^\kappa(\AA^\kappa)\subset\PSh_{\Cat}(\AA^\kappa)$ is continuous, we thus obtain a morphism of 2-topoi
	\begin{equation*}
		\Shv_{\Cat}(\AA)\to\bCC
	\end{equation*}
	that recovers $f$ when we restrict along the Yoneda embedding $h_{\AA}$. Hence the claim follows.
\end{proof}

\begin{theorem}\label{thm:localicReflection}
	The assignment $\AA\mapsto \Shv_{\Cat}(\AA)$ determines a fully faithful functor $\LTop \to\LTTop$ that is left adjoint to $\Grpd(-)$.
\end{theorem}
\begin{proof}
	This follows immediately from~\cref{prop:UMPSheaves}.
\end{proof}

\begin{remark}
	The proof of \cref{prop:UMPSheaves} shows that if $\bCC$ is a 2-topos, the counit $\Shv_{\Cat}(\Grpd(\bCC))\to\bCC$ of the adjunction $\Shv_{\Cat}(-)\dashv \Grpd(-)$ can be explicitly described as follows: choose a regular cardinal $\kappa$ so that $\Grpd(\bCC)$ is $\kappa$-accessible and the collection of $\kappa$-compact objects in $\Grpd(\bCC)$ is closed under finite limits. Then one has $\Shv_{\Cat}(\Grpd(\bCC))\simeq\Shv_{\Cat}^\kappa(\Grpd(\bCC)^\kappa)$, and the counit can be obtained as the left Kan extension of the inclusion $\Grpd(\bCC)^\kappa\into\bCC$ along $h_{\Grpd(\bCC)^\kappa}\colon\Grpd(\bCC)^\kappa\into\Shv_{\Cat}^\kappa(\Grpd(\bCC)^\kappa)$. In particular, \cref{rem:universalAdjointFunctorTheorem} implies that the right adjoint of the counit is given by
	\begin{equation*}
	 \bCC\simeq\Fun^{\cont}(\bCC^\op,\Cat)\xrightarrow{\iota^\ast}\Shv_{\Cat}(\Grpd(\bCC)),
	\end{equation*}
	where $\Fun^{\cont}(\bCC^\op,\Cat)$ denotes the 2-category of continuous functors and the left equivalence is obtained by the Yoneda embedding, see \cref{cor:sheavesRepresentable}.
\end{remark}

\begin{definition}\label{def:localic2topoi}
	We say that a 2-topos $\bCC$ is \emph{1-localic} if lies in the essential image of $\Shv_{\Cat}(-)$.
\end{definition}

\begin{remark}
	It is easy to encounter 2-topoi in nature which are not 1-localic. Coming back to the example $\bCC=\Fun(\Adj,\Cat)$ given in \cref{rem:emptycores}, we observe that the associated localic reflection adjunction recovers $\Cat$, so that $\bCC$ cannot possibly be 1-localic.

	A second non-1-localic example is the 2-topos $\Catidem$ from  \cref{ex:idempotentCompleteCategories}. Again, as every space is idempotent complete, it follows that we have $\Grpd(\Catidem)= \Spc$, so that again the localic reflection of this 2-topos is $\Cat$. 
\end{remark}

\begin{proposition}\label{prop:localicinvolutive}
	Every localic 2-topos $\bCC$ is involutive (see \cref{def:involutive}).
\end{proposition}
\begin{proof}
	By definition, we may identify $\bCC= \Shv_{\Cat}(\CC)$ for some 1-topos $\CC$. Thus the canonical involution of $\Cat$ induces an involution
	\[
		\Shv_{\Cat}(\CC) \xrightarrow{\simeq} \Shv_{\Cat^\co}(\CC) \simeq \Shv_{\Cat^\co}(\CC^\co) \xrightarrow{\simeq} \Shv_{\Cat}(\CC)^{\co},
	\]
	so that $\bCC$ is indeed involutive.
\end{proof}